\documentclass[a4paper]{amsart}
\usepackage{
 amsmath,
 amsthm,
 amssymb,
 mathrsfs,
 mathtools,
 svninfo,
 fullpage,
 type1cm,
 enumerate,
 color}

\DeclareFontFamily{U}{rsfs}{\skewchar\font127 }
\DeclareFontShape{U}{rsfs}{m}{n}{%
  <-6> rsfs5
  <6-8> rsfs7
  <8-> rsfs10
}{}
\usepackage[PostScript=Rokicki]{diagrams}

\usepackage[pagebackref,final,breaklinks=true]{hyperref}

\newcommand{\tbt}[4]{\begin{pmatrix}#1 & #2 \\ #3 & #4\end{pmatrix}}
\newcommand{\stbt}[4]{\left(\begin{smallmatrix}#1 & #2 \\ #3 & #4\end{smallmatrix}\right)}

\svnInfo $Id: klz1b.tex 394 2016-08-25 14:58:36Z davidloeffler $

\theoremstyle{plain}
  \newtheorem{theorem}{Theorem}[subsection]
  \newtheorem{lemma}[theorem]{Lemma}
  \newtheorem{proposition}[theorem]{Proposition}
  \newtheorem{corollary}[theorem]{Corollary}
  \newtheorem{definition}[theorem]{Definition}
  \newtheorem{lettertheorem}{Theorem}
  
  \newtheorem{lettercorollary}[lettertheorem]{Corollary}

\theoremstyle{remark}
  \newtheorem{remark}[theorem]{Remark}
  \newtheorem{notation}[theorem]{Notation}
  \newtheorem{hypothesis}[theorem]{Hypothesis}

\DeclareMathOperator{\TSym}{TSym}
\DeclareMathOperator{\Sym}{Sym}
\DeclareMathOperator{\Hom}{Hom}
\DeclareMathOperator{\Spec}{Spec}
\DeclareMathOperator{\Eis}{Eis}
\DeclareMathOperator{\Fil}{Fil}
\DeclareMathOperator{\GL}{GL}
\DeclareMathOperator{\SL}{SL}
\DeclareMathOperator{\pr}{pr}
\DeclareMathOperator{\loc}{loc}
\DeclareMathOperator{\myPr}{Pr}\renewcommand{\Pr}{\myPr}
\DeclareMathOperator{\mom}{mom}
\DeclareMathOperator{\norm}{norm}
\DeclareMathOperator{\Char}{char}
\DeclareMathOperator{\Gal}{Gal}
\DeclareMathOperator{\Ind}{Ind}
\DeclareMathOperator{\Hyp}{Hyp}
\DeclareMathOperator{\Col}{Col}
\DeclareMathOperator{\Rank}{rank}
\DeclareMathOperator{\Frac}{Frac}
\DeclareMathOperator{\comp}{comp}
\DeclareMathOperator{\id}{id}
\DeclareMathOperator{\Tate}{Tate}

\newcommand{\dR}{\mathrm{dR}}
\newcommand{\et}{\text{\textup{\'et}}} 
\newcommand{\mot}{\mathrm{mot}}
\newcommand{\ord}{\mathrm{ord}}
\newcommand{\Par}{\mathrm{par}}

\newcommand{\sF}{\mathscr{F}}
\newcommand{\sH}{\mathscr{H}}
\newcommand{\sL}{\mathscr{L}}

\newcommand{\cE}{\mathcal{E}}
\newcommand{\cF}{\mathcal{F}}
\newcommand{\cO}{\mathcal{O}}
\newcommand{\cP}{\mathcal{P}}
\newcommand{\cL}{\mathcal{L}}
\newcommand{\cW}{\mathcal{W}}

\newcommand{\frP}{\mathfrak{P}}
\newcommand{\frQ}{\mathfrak{Q}}
\newcommand{\frS}{\mathfrak{S}}

\newcommand{\CC}{\mathbf{C}}
\newcommand{\FF}{\mathbf{F}}
\newcommand{\HH}{\mathbf{H}} 
\newcommand{\QQ}{\mathbf{Q}}
\newcommand{\ZZ}{\mathbf{Z}}
\newcommand{\DD}{\mathbf{D}} 
\newcommand{\NN}{\mathbf{N}}
\newcommand{\TT}{\mathbf{T}}

\newcommand{\bfa}{\mathbf{a}}

\newcommand{\bff}{\mathbf{f}}
\newcommand{\bfg}{\mathbf{g}}

\renewcommand{\j}{\mathbf{j}}
\renewcommand{\k}{\mathbf{k}}

\DeclareFontFamily{U}{wncy}{}
\DeclareFontShape{U}{wncy}{m}{n}{<->wncyr10}{}
\DeclareSymbolFont{mcy}{U}{wncy}{m}{n}
\DeclareMathSymbol{\Sha}{\mathord}{mcy}{"58}

\newcommand{\Qp}{{\QQ_p}}
\newcommand{\Zp}{{\ZZ_p}}
\newcommand{\Gm}{\mathbf{G}_m}
\newcommand{\RGt}{\widetilde{R\Gamma}}
\newcommand{\Ht}{\widetilde{H}}
\newcommand{\Lp}{L_{\frP}}
\newcommand{\hpr}{\hat\pr}

\newcommand{\into}{\hookrightarrow}
\newcommand{\htimes}{\mathop{\hat\otimes}}

\newcommand{\cEI}{{}_c\mathcal{EI}}
\newcommand{\cRI}{{}_c\mathcal{RI}}
\newcommand{\cBF}{{}_c\mathcal{BF}}
\newcommand{\BF}{\mathcal{BF}}

\numberwithin{equation}{subsection}

\begin{document}

\title{Rankin--Eisenstein classes and explicit reciprocity laws}

\author{Guido Kings}
\address{Fakult\"at f\"ur Mathematik, Universit\"at Regensburg, 93040 Regensburg, Germany}
\email{guido.kings@ur.de}

\author{David Loeffler}
\address{Mathematics Institute, Zeeman Building, University of Warwick, Coventry CV4 7AL, UK}
\email{d.a.loeffler@warwick.ac.uk}
\thanks{
  }

\author{Sarah Livia Zerbes}
\address{Department of Mathematics, University College London, London WC1E 6BT, UK}
\email{s.zerbes@ucl.ac.uk}

\thanks{Supported by: SFB 1085 ``Higher invariants'' (Kings);
  Royal Society University Research Fellowship ``$L$-functions and Iwasawa theory''
  and NSF Grant No. 0932078 000 (Loeffler);
  EPSRC First Grant EP/J018716/1, Leverhulme Trust Research Fellowship ``Euler systems and Iwasawa theory'', and NSF Grant No. 0932078 000 (Zerbes).}

\begin{abstract}
We construct three-variable $p$-adic families of Galois cohomology classes attached to Rankin convolutions of modular forms, and prove an explicit reciprocity law relating these classes to critical values of L-functions. As a consequence, we prove finiteness results for the Selmer group of an elliptic curve twisted by a 2-dimensional odd irreducible Artin representation when the associated $L$-value does not vanish.

\textcolor{red}{Edited November 2023 to add a corrigendum -- see end of file}
\end{abstract}


\maketitle
\setcounter{tocdepth}{1}
\tableofcontents

\section{Introduction}

 \subsection{Overview}

  One of the most basic questions in number theory is the study of the cohomology of Galois representations, and in particular the relation between these groups and the values of $L$-functions. In this paper, we are interested in one case of this: the $L$-function and the Galois representation associated to the convolution of two modular forms.

  Let $f = \sum a_n(f) q^n$ and $g = \sum a_n(g) q^n$ be two modular cusp forms, of levels $N_f, N_g$, characters $\varepsilon_f, \varepsilon_g$, and weights $k + 2, k' + 2$, with $k, k' \ge -1$. We assume that $f$ and $g$ are eigenvectors for the Hecke operators. We define the Rankin--Selberg $L$-function by
  \[ L(f, g, s) = L_{(N_f N_g)}(\varepsilon_f \varepsilon_g, 2s -2 - k - k') \cdot \sum_{n \ge 1} a_n(f) a_n(g) n^{-s}, \]
  where $L_{(N_f N_g)}(\varepsilon_f \varepsilon_g,s)$ denotes the Dirichlet $L$-function with the Euler factors at the primes dividing $N_fN_g$ removed. Up to finitely many bad Euler factors at the primes dividing $N_f N_g$, this is the $L$-function associated to the compatible family of Galois representations
  \[ M_{\Lp}(f \otimes g) = M_{\Lp}(f) \otimes M_{\Lp}(g), \]
  where $M_{\Lp}(f)$ and $M_{\Lp}(g)$ are Deligne's $p$-adic Galois representations attached to $f$ and $g$ (and to a prime $\frP$ of their common coefficient field $L$).

  According to the Bloch--Kato conjecture, we expect that the values of the $L$-function $L(f, g, s)$ for integer values of $s$ determine the behaviour of a \emph{Selmer group} (a subgroup of the Galois cohomology determined by local conditions; we shall recall its definition in \S \ref{sect:selgeneralities} below). Specifically, for $j \in \mathbf{Z}$ the value $L(f, g, 1 + j)$ should be related to the Bloch--Kato Selmer group $H^1_{\mathrm{f}}(\QQ, M_{\Lp}(f \otimes g)^*(-j))$.

  The main result of this paper is to show (under some technical hypotheses) that, for $j$ in the \emph{critical range} $\min(k, k') + 1 \le j \le \max(k, k')$, we have the implication
  \[ L(f,g, 1+j) \ne 0 \quad \Rightarrow \quad H^1_{\mathrm{f}}(\QQ, M_{\Lp}(f \otimes g)^*(-j)) \text{ is finite}.\]

  This statement has many interesting consequences. For instance, following a beautiful idea of Bertolini, Darmon and Rotger, we see that it has powerful applications to the arithmetic of Artin twists of elliptic curves: one can consider the special case where $f$ is a weight $2$ form corresponding to an elliptic curve $E$, and $g$ a weight $1$ form corresponding to a two-dimensional odd Artin representation $\rho$ with splitting field $F$. Then
  \[
   H^{1}_\mathrm{f}(\QQ, M_{\Lp}(f \otimes g)^*)\text{ is finite }
   \Rightarrow
   E(F)^{\rho}\mbox{ and } \Sha_{p^{\infty}}(E/F)^{\rho} \mbox{ are finite,}
  \]
  where $E(F)^{\rho}$ and $\Sha_{p^{\infty}}(E/F)^{\rho}$ are the $\rho$-isotypic component.

  \subsubsection*{Constructing an Euler system}

   In order to obtain the desired bounds for Selmer groups, we construct an ``Euler system'' for the Galois representation $M_{\Lp}(f \otimes g)^*(-j)$: a collection of classes
   \[
    c_m \in H^1(\QQ(\mu_m), M_{\Lp}(f \otimes g)^*(-j))
   \]
   over cyclotomic fields, satisfying norm-compatibility relations as $m$ varies which mirror the Euler product of the Rankin--Selberg $L$-function. By the work of Kolyvagin and Rubin, it is known that if a nontrivial Euler system exists for some Galois representation $T$, it forces very strong bounds on the Selmer groups of $T$.

   We construct our Euler system using certain cohomology classes (\emph{\'etale Rankin--Eisenstein classes}) which were introduced and studied in \cite{KLZ1a}. In the simplest case (when $r = r' = 2$ and $j = 0$) these  classes arise from modular units, via the push-forward under the diagonal embedding $\Delta: Y_1(N)\to Y_1(N)\times Y_1(N)$; this weight 2 case was studied extensively in \cite{LLZ14}.

  \subsubsection*{P-adic deformation}

   The Rankin--Eisenstein classes used to build our Euler system are constructed by geometric techniques, and these geometric methods can only be used for a certain range of weights. Specifically, the Rankin--Eisenstein class for $(f, g, j)$ is only defined when $0 \le j \le \min(k, k')$. This has no overlap with the critical range, which is $\min(k, k') + 1 \le j \le \max(k, k')$. In order to access this wider range of $L$-values, we use deformation in $p$-adic analytic families, constructing Euler systems in the critical range as $p$-adic limits of Rankin--Eisenstein classes.

  \subsubsection*{Explicit reciprocity laws}

   In order to use an Euler system to bound Selmer groups, one needs to know that the Euler system concerned is not zero. Such non-vanishing results are typically obtained as a consequence of an \emph{explicit reciprocity law}, relating the cohomology classes in the Euler system to values of $L$-functions.

   In our case we prove an explicit reciprocity law relating our Euler system to the $3$-variable $p$-adic Rankin $L$-function introduced by Hida. In the critical range $\min(k, k') < j \le \max(k, k')$, this $p$-adic $L$-function coincides up to some explicit factors with $L(f, g, 1+j)$. Thus, when this $L$-value is non-zero, one can put the Euler system machinery into action to get the desired finiteness results for Selmer groups.\\

 In the next sections we describe in a little more detail how we shall carry out the programme sketched above.

  Note that parts of this programme have already been carried out by Bertolini, Darmon and Rotger \cite{BDR-BeilinsonFlach, BDR-BeilinsonFlach2}, and there is some overlap between our results and theirs. We shall describe later in this introduction some of the similarities, and some of the differences, between their approach and ours.

 \subsection{Rankin--Eisenstein classes}

  The Euler system we are going to use is built out of the \emph{Rankin--Eisenstein classes}, which were studied in our previous paper \cite{KLZ1a}.

  These classes can be defined in multiple cohomology theories. The basic classes lie in \emph{motivic cohomology} (a cohomology theory taking values in $\mathbf{Q}$-vector spaces, closely related to algebraic $K$-theory). Motivic cohomology has canonical maps to many other cohomology theories (such as \'etale cohomology), and one obtains Rankin--Eisenstein classes in these cohomology theories as the images of the motivic Rankin--Eisenstein classes.

  Motivic Rankin--Eisenstein classes were first introduced (although not under this name) by Beilinson in the weight $2$ case, and in general by Scholl (unpublished). They are classes in motivic cohomology
  \[
   \Eis^{[k, k', j]}_{\mot, 1, N} \in H^{3}_\mot\left(Y_1(N)^{2},\TSym^{[k,k']}(\sH_{\QQ})(2-j)\right),
  \]
  for integers $0\le j\le \min(k,k')$. Here $Y_1(N)^{2}$ is the product of $Y_1(N)$ with itself, $\sH_{\QQ}$ is the motivic sheaf on $Y_1(N)$ associated to the first homology of the universal elliptic curve, $\TSym^k$ denotes the symmetric tensors of degree $k$ (cf.\ \S \ref{sect:multlin} below), and $\TSym^{[k,k']}\sH_{\QQ}$ is the sheaf $\TSym^{k}(\sH_{\QQ}) \boxtimes \TSym^{k'}(\sH_{\QQ})$ on $Y_1(N)^2$.

  We point out that the condition on $j$, that $0 \le j \le \min(k,k')$, is precisely the range in which $L(f, g, s)$ is forced to vanish to order exactly 1 at $s = 1 + j$, due to the form of the archimedean Gamma factors\footnote{More precisely, the order of vanishing is exactly 1 in this range except in one exceptional case, when $k = k' = j$ and $f$ and $g$ are complex conjugates of each other; in this exceptional case the order of vanishing is 0.}.

  The construction of these classes is (perhaps surprisingly) fairly simple. Beilinson has defined a canonical class (the \emph{motivic Eisenstein class})
  \[
   \Eis_{\mot,1,N}^{k}\in H^{1}_\mot\left(Y_1(N),\TSym^{k}(\sH_{\QQ})(1)\right)
  \]
  for any integer $k \ge 0$. (For $k = 0$, the motivic cohomology group is simply $\cO(Y_1(N))^\times \otimes \QQ$, and $\Eis_{\mot,1,N}^{0}$ is the Siegel unit $g_{0, 1/N}$.) The motivic Rankin--Eisenstein class is then defined by pushing forward the class $\Eis_{\mot,1,N}^{k + k' - 2j}$ along the diagonal inclusion
  \[ \Delta: Y_1(N) \into Y_1(N) \times Y_1(N).\]

  The aim of Beilinson and Scholl was to compute the image of this class in Beilinson's absolute Hodge cohomology, a cohomology theory built up from real-analytic differential forms. They showed that, for any two eigenforms $f, g$ of levels dividing $N$, the cup-product of the Hodge Rankin--Eisenstein class with a differential form coming from $f$ and $g$ computes the first derivative $L'(f,g, j+1)$ of the Rankin--Selberg $L$-function, as predicted by Beilinson's conjecture. This result of Beilinson and Scholl relies crucially on computations by Beilinson, who had expressed the
  image of $\Eis_{\mot,1,N}^{k+k'-2j}$ in absolute Hodge cohomology by explicit real-analytic Eisenstein series.

  This beautiful and fundamental result was complemented in our earlier paper \cite{KLZ1a} by a corresponding computation for the image of the motivic Rankin--Eisenstein class in Besser's $p$-adic rigid syntomic cohomology (a cohomology theory built up from $p$-adic rigid-analytic differential forms), for a prime $p \nmid N$, extending the case $k = k' = j = 0$ treated in \cite{BDR-BeilinsonFlach}. There, in a completely parallel way, we find that pairing this syntomic Rankin--Eisenstein class with de Rham classes arising from $f$ and $g$ gives the special value $L_p(f,g,j+1)$ of Hida's $p$-adic  Rankin--Selberg $L$-function. In this syntomic computation, the role of Beilinson's formula for the Hodge Eisenstein class on $Y_1(N)$ is played by an explicit formula for the syntomic Eisenstein class in terms of $p$-adic Eisenstein series, due to Bannai and the first author \cite{BannaiKings}. The value $L_p(f, g, 1 + j)$ lies outside the range of interpolation of Hida's $p$-adic $L$-function, and thus is not straightforwardly related to any complex $L$-value; however, our computation shows that the non-critical complex $L$-value $L(f, g, 1 + j)$ and the non-critical $p$-adic $L$-value $L_p(f, g, 1 + j)$ are linked by the fact that they appear as the complex and $p$-adic regulators, respectively, of the same motivic cohomology class, confirming a conjecture of Perrin-Riou.

 \subsection{Statements of the main results}

  In the present paper we study the \'etale Rankin--Eisenstein class $\Eis^{[k, k', j]}_{\et, 1, N}$, defined as the image of the motivic Rankin--Eisenstein class $\Eis^{[k, k', j]}_{\mot, 1, N}$ in \'etale cohomology. For eigenforms $f, g$ of weights $k + 2, k' + 2$ and levels dividing $N$, we may project this \'etale Rankin--Eisenstein class into the $(f, g)$-isotypical component, giving a class
  \[
   \Eis_{\et}^{[f, g, j]}\in  H^1_\et\left(\ZZ[1/Np], M_{\Lp}(f\otimes g)^*(-j)\right),
  \]
  where $M_{\Lp}(f\otimes g)^*$ denotes the tensor product of Galois representations associated to $f$ and $g$ (with coefficients in $\Lp$, the completion of the coefficient field $L$ at a prime $\frP \mid p$).

  Our first aim is to interpolate $\Eis_{\et}^{[f, g, j]}$  in all three variables, i.e., to replace $f$ and $g$ by Hida families and the twist $j$ by the universal character $\j$ of the cyclotomic Iwasawa algebra $\Lambda_\Gamma$. In slightly rough terms, our first main result can be formulated as follows:

  \begin{lettertheorem}[{Theorem \ref{thm:BFeltsinterp2}}]
   \label{lthm:iwasawaelt}
   Let $\bff, \bfg$ be Hida families, and let $M(\bff)^{*}$, $M(\bfg)^{*}$ be the associated $\Lambda$-adic Galois representations. Then for each $m \ge 1$ coprime to $p$, and each $c > 1$ coprime to $6pmN$, there is a \emph{Beilinson--Flach class}
   \[ \cBF_{m}^{\bff, \bfg} \in H^1_{\et}\left(\ZZ[\mu_m, \tfrac{1}{mNp}], M(\bff)^* \otimes M(\bfg)^*\otimes \Lambda_\Gamma(-\j)\right), \]
   with the following properties: when $m = 1$, the specialisations of this class recover the Rankin--Eisenstein classes $\Eis_{\et}^{[f, g, j]}$ for all classical specialisations $f$ of $\bff$ and $g$ of $\bfg$, and all integers $j$ for which the Rankin--Eisenstein class is defined; and the Beilinson--Flach classes satisfy compatibility relations of Euler system type as $m$ varies.
  \end{lettertheorem}

  \begin{remark}
   The dependence on the auxilliary parameter $c$ is very minor: it appears in the factors relating the Rankin--Eisenstein class to the specialisations of the Beilinson--Flach classes. Unfortunately, it is not possible to remove this dependence entirely without introducing undesirable denominators. This reflects the fact that Rankin $L$-functions can have simple poles at $s = k + 1$ if the two forms are conjugates of each other.
  \end{remark}

  We note that the construction of the Beilinson--Flach classes $\cBF_{m}^{\bff, \bfg}$, and their Euler system compatibility relations, were essentially already obtained in \cite{LLZ14} using only weight 2 Rankin--Eisenstein classes. The novel part of the above result is to show that the specialisations of the classes $\cBF_{m}^{\bff, \bfg}$ in fact recover the Rankin--Eisenstein classes for all $(k, k', j)$. Very few results of this kind, relating cohomology classes arising from geometry in different weights, were previously known, and these have so far always been proved as a consequence of an explicit reciprocity law, relating geometric classes to values of $L$-functions (as in \cite{Castella-Heegner-cycles} for Heegner cycles, and \cite[Theorem 5.10]{DR-diagonal-cycles-II} for diagonal cycles on triple products of Kuga--Sato varieties\footnote{In fact, Henri Darmon has recently informed us that he and his coauthors have used the methods introduced in this paper to give a direct proof of this result, avoiding the use of explicit reciprocity laws.}). In contrast, we prove Theorem A by an intrinsic geometric method, and we shall in fact obtain a relation to $L$-values as a \emph{consequence} of this theorem. See \S \ref{sect:introthmA} below for an outline of the proof of Theorem A.

  Our second main theorem is to relate the classes $\cBF_{m}^{\bff, \bfg}$ to values of $L$-functions. This relation goes via a map arising in $p$-adic Hodge theory: a generalisation of Perrin-Riou's ``big logarithm'' map (due to the second and third author \cite{LZ}). The second main result of this paper is then the following explicit reciprocity law, again stated in a rather rough form:

  \begin{lettertheorem}[{Theorem \ref{thm:explicitrecip}}]
   \label{lthm:explicitrecip}
   The image of $\cBF^{\bff, \bfg}_1$ under Perrin-Riou's big logarithm is Hida's $p$-adic Rankin--Selberg $L$-function (up to an explicit non-zero factor depending on $c$).
  \end{lettertheorem}

  This theorem generalises a result \cite{BDR-BeilinsonFlach2} of Bertolini, Darmon and Rotger, which is concerned with the special case $j=0$ and $f$ a fixed form of weight $2$. We shall give an outline of the proof of Theorem \ref{lthm:explicitrecip} in \S \ref{sect:introthmB} below.

  With these two theorems in hand, we can put the machinery of Euler systems to work. It is clear from Theorem \ref{lthm:explicitrecip} that the non-vanishing of a specialization of $\cBF^{\bff, \bfg}_1$ is completely controlled by the non-vanishing of the corresponding specialization of the $p$-adic $L$-function, and in the critical range this is simply the algebraic part of the classical $L$-value.

  Our first application is based on a wonderful idea of Bertolini, Darmon and Rotger, which is to specialise the Hida family $\bfg$ at a weight 1 modular form, corresponding to a 2-dimensional Artin representation.

  \begin{lettercorollary}[Theorem \ref{thm:sha}]
   Let $E / \QQ$ be an elliptic curve without complex multiplication, and $\rho$ a 2-dimensional odd irreducible Artin representation of $G_{\QQ}$ with splitting field $F$. Let  $p$ be prime at which $E$ is ordinary and which satisfies some further technical conditions. Then
   \[
    L(E, \rho, 1) \ne 0\Rightarrow E(F)^{\rho}\mbox{ and } \Sha_{p^{\infty}}(E/F)^{\rho}
    \mbox{ are finite.}
   \]
  \end{lettercorollary}

  The implication ``$L(E, \rho, 1) \ne 0\Rightarrow E(F)^{\rho}$ is finite'' was obtained already by Bertolini, Darmon and Rotger in \cite{BDR-BeilinsonFlach2}; the method of Euler systems allows us to extend this to obtain finiteness of the $p$-part of $\Sha$.

  As a second application, we use the Euler system machinery to study the Iwasawa theory of our Galois representation $M_{\Lp}(f \otimes g)^*$ over the $p$-adic cyclotomic tower. The results can be summarized as follows:

  \begin{lettercorollary}[Theorem \ref{thm:SelFbound} and \ref{thm:SelFbound-finite}] Under some technical hypotheses, we obtain one divisibility in the Iwasawa--Greenberg main conjecture for the Galois representation $M_{\Lp}(f \otimes g)^*$ over $\QQ(\mu_{p^\infty})$: the characteristic ideal of a suitable dual Selmer group divides the $p$-adic $L$-function.
  \end{lettercorollary}

  Further and much more detailed results can be found in section \ref{sect:selmerbound}.

 \subsection{Outline of the proof of Theorem \ref{lthm:iwasawaelt}}
  \label{sect:introthmA}

   In this introduction, we suppose (for simplicity) that $N_f = N_g = N$. Recall that the motivic Rankin--Eisenstein classes, which live in the $(f, g)$-isotypical part of the motivic cohomology of the product $Y_1(N) \times Y_1(N)$, are defined using the pushforward of Beilinson's motivic Eisenstein class on a single modular curve $Y_1(N)$, along the diagonal inclusion
   \[ \Delta: Y_1(N) \into Y_1(N) \times Y_1(N).\]
   Hence, in order to $p$-adically interpolate the \'etale versions of the Rankin--Eisenstein classes, we shall begin by solving the simpler problem of interpolating the classes on $Y_1(N)$ given by the \'etale realisation of Beilinson's Eisenstein class. We denote these classes by
  \[ \Eis^k_{\et, N} \in H^1_{\et}\left( Y_1(N)_{\ZZ[1/Np]}, \TSym^{k}(\sH_\Zp)(1)\right).\]
  Here $\sH_{\Zp}$ is the \'etale $\Zp$-sheaf on $Y_1(N)$ given by the Tate module of the universal elliptic curve $\cE / Y_1(N)$, and $\TSym^{k}\sH_\Zp$ is the sheaf of symmetric tensors of degree $k$ over $\sH_\Zp$, which is isomorphic after inverting $k!$ to the $k$-th symmetric power (we recall the definition in \S \ref{sect:multlin} below).

  The interpolation of these classes is carried out using the formalism of Lambda-adic sheaves introduced in \cite{Kings-Eisenstein}. We consider the sheaf of Iwasawa modules $\Lambda(\sH_\Zp\langle t_N \rangle)$ associated to $\sH_{\Zp}$ and its canonical order $N$ section $t_N$, which is equipped with moment maps
  \[
   \mom^{k}: \Lambda(\sH_\Zp\langle t_N \rangle) \to \TSym^{k}\sH_\Zp
  \]
  for all integers $k \ge 0$. One of the main results of \cite{Kings-Eisenstein} is that there is a class, the \emph{Eisenstein--Iwasawa class},
  \[ \cEI_{N} \in  H^1_{\et}\Big(Y_1(N)_{\ZZ[1/Np]}, \Lambda(\sH_{\Zp}\langle t_N \rangle)(1)\Big), \]
  for any integer $c > 1$ coprime to $6pN$, such that
  \[
   \mom^{k}(\cEI_{N}) = \left(c^2 - c^{-k} \langle c \rangle\right) \Eis^k_{\et, N}
  \]
  for all $k \ge 0$. We remark that this interpolation property depends on a very careful study of the \'etale realisation of the elliptic polylogarithm. Section \ref{sect:EisIwasawa} of this paper is devoted to recalling the construction of these classes $\cEI_{N}$ from \cite{Kings-Eisenstein}, and proving two (relatively straightforward) norm-compatibility relations describing the pushforward of $\cEI_{N}$ along degeneracy maps between modular curves of different levels.

  We then use this class $\cEI_{N}$ on $Y_1(N)$ in order to construct the Beilin\-son--Flach class on $Y_1(N) \times Y_1(N)$, as follows. A first approximation would be to use the comultiplication $\Lambda(\sH_\Zp \langle t_N \rangle)\to \Lambda(\sH_\Zp  \langle t_N \rangle)\otimes \Lambda(\sH_\Zp \langle t_N \rangle)$ and pushforward along the diagonal embedding $Y_1(N) \into Y_1(N)^2$, mimicking the construction of the Rankin--Eisenstein classes; this gives a map
  \[ H^1_{\et}\Big(Y_1(N)_{\ZZ[1/Np]},\Lambda(\sH_{\Zp} \langle t_N \rangle)(1)\Big) \to H^3_{\et}\Big(Y_1(N)^2_{\ZZ[1/Np]}, \Lambda(\sH_\Zp \langle t_N \rangle)\boxtimes \Lambda(\sH_\Zp \langle t_N \rangle)(2)\Big),\]
  and applying this map to $\cEI_{N}$ gives a class which interpolates the \'etale Rankin--Eisenstein classes $\Eis^{[f, g, j]}_{\et}$ for $j = 0$ and all $f, g$ of level $N$ and weights $\ge 2$.

  However, this will always give classes defined over $\QQ$ (or indeed over $\ZZ[1/Np]$), so in order to obtain classes defined over cyclotomic fields, and thus to interpolate analytically in the $j$ variable, this is not sufficient. Hence we make a slight but crucial modification of this construction, following an idea introduced in \cite{LLZ14}: we work on a higher level modular curve $Y(M, N)$, where $M \mid N$, and compose the diagonal embedding with a suitable Hecke correspondence. This defines what we call
  the \emph{Rankin--Iwasawa class} (see Definition \ref{def:RIclass})
  \[
   \cRI_{M, N, a}^{[0]}\in H^{3}_\et\left(Y(M, N)^{2}_{\ZZ[1/MNp]},\Lambda(\sH_\Zp \langle t_N \rangle)^{\boxtimes 2}(2)\right),
  \]
  for each $a \in (\ZZ / M \ZZ)^\times$. More generally, for any integer $j \ge 0$ a variation of this construction gives a class
  \[
   \cRI_{M, N, a}^{[j]}\in H^{3}_\et\left(Y(M, N)^{2}_{\ZZ[1/MNp]}, \Lambda(\sH_{\Zp}\langle t_N \rangle)^{[j, j]} (2-j)\right),
  \]
  where $\Lambda(\sH_{\Zp}\langle t_N \rangle)^{[j, j]} = \left( \Lambda(\sH_\Zp\langle t_N \rangle) \otimes \TSym^j \sH_{\Zp}\right){}^{\boxtimes 2}$. The Rankin--Iwa\-sawa class $\cRI_{M, N, a}^{[j]}$ can be used to recover the Rankin--Eisenstein classes $\Eis^{[f, g, j]}_{\et}$ for all $f, g$ of level $N$ and weights $\ge j + 2$. Moreover, there is a natural map
  \[ Y(m, mN)^2 \to Y_1(N)^2 \times \Spec \ZZ\left[\mu_m, \tfrac{1}{mN}\right], \]
  and after pushing forward along this map, and projecting to the ordinary part, the Rankin--Iwasawa classes for different $j$ become compatible under cyclotomic twists (Theorem \ref{thm:BFeltsinterp1}). This defines the Beilinson--Flach class
  \[
   \cBF_{m,N,a} \in (e_{\ord}', e_{\ord}') \cdot H^{3}_\et\left(Y_1(N)^{2}_{\ZZ\left[\mu_m, 1/mNp\right]},
   \Lambda(\sH_\Zp \langle t_N \rangle)^{\boxtimes 2} \otimes \Lambda_\Gamma(-\j)\right).
  \]
  (The projection to the ordinary part is required in order to obtain analytic variation in the cyclotomic variable $j$, an observation which also goes back to \cite{LLZ14}.)

  Finally, to project to the Hida families one proceeds as follows. We use results of Ohta to show that the $\Lambda$-adic representations $M(\bff)^*$, $M(\bfg)^*$ can be realised as quotients of the \'etale cohomology groups
  \[ e_{\ord}'\cdot H^{1}_\et\left(Y_1(N)_{\overline{\QQ}}, \Lambda(\sH_\Zp \langle t_N \rangle)(1)\right)
  \]
  for any $N$ divisible by the $p$ and by the tame levels of the two families. Then one uses the Hochschild--Serre spectral sequence and the K\"unneth formula to get a map
  \begin{multline*}
   H^{3}_\et\left(Y_1(N)^{2}_{\ZZ\left[\mu_m, 1/mNp\right]}, \Lambda(\sH_\Zp \langle t_N \rangle)^{\boxtimes 2}\otimes \Lambda_\Gamma(-\j)\right) \\ \to
   H^{1}_\et\left(\ZZ\left[\mu_m, \tfrac{1}{mNp}\right], H^{1}_\et(Y_1(N)_{\overline{\QQ}}, \Lambda(\sH_\Zp \langle t_N \rangle))^{\otimes 2}\otimes \Lambda_\Gamma(-\j)\right).
  \end{multline*}
  After projection to the ordinary part one obtains the Beilinson--Flach classes for $\bff$ and $\bfg$,
   \[ \cBF_{m}^{\bff, \bfg} \in H^1_{\et}\left(\ZZ\left[\mu_m, 1/mNp\right], M(\bff)^* \otimes M(\bfg)^*\otimes \Lambda_\Gamma(-\j)\right). \]
  It is essentially clear from the construction that these classes interpolate the Rankin--Eisenstein classes, which proves Theorem \ref{lthm:iwasawaelt}.

 \subsection{Outline of the proof of Theorem \ref{lthm:explicitrecip}}
  \label{sect:introthmB}

  The essential strategy of the proof of Theorem \ref{lthm:explicitrecip} is to ``analytically continue'' the relation to $p$-adic $L$-values given by the syntomic regulator computations of \cite{KLZ1a} along the 3-parameter family constructed in Theorem A.

  Let us fix two Hida families $\bff$, $\bfg$. For simplicity, we assume in this introduction that $\bff$ and $\bfg$ are non-Eisenstein modulo $p$, that the Hecke algebras associated to $\bff$ and $\bfg$ are unramified over $\Lambda = \Zp[[T]]$, and that the prime-to-$p$ part of the Nebentypus of $\bfg$ is trivial; for the full statements, see the main body of the paper.

  For every pair of classical specialisations $f, g$ of $\bff, \bfg$ respectively, with $f, g$ newforms of levels coprime to $p$ and weights $k + 2, k' + 2 \ge 2$, and each $j$ such that $0 \le j \le \min(k, k')$, we have the Rankin--Eisenstein class $\Eis^{[f, g, j]}_{\et} \in H^1(\QQ, M_{\Lp}(f \otimes g)^*(-j))$. The localisation of this class at $p$ lies in the Bloch--Kato ``finite'' subspace
  \[ H^1_{\mathrm{f}}(\Qp, M_{\Lp}(f \otimes g)^*(-j)) \subseteq H^1(\Qp, M_{\Lp}(f \otimes g)^*(-j)), \]
  and the Bloch--Kato logarithm map of $p$-adic Hodge theory identifies this $H^1_{\mathrm{f}}$ with the dual of a certain subspace of the de Rham cohomology of $Y_1(N)^2$. The eigenforms $f, g$ determine a vector $\eta_f \otimes \omega_g$ in this de Rham cohomology space, and the main result of \cite{KLZ1a} is a formula of the form
  \[
   \left\langle \log\left(\Eis^{[f, g, j]}_{\et}\right), \eta_f \otimes \omega_g \right\rangle = (\star) \cdot L_p(f, g, 1 + j),
  \]
  where $(\star)$ is an explicit ratio of Euler factors.

  As $f, g$ vary in the families $\bff, \bfg$, and $j$ varies over the integers, we have $p$-adic interpolations of all the objects appearing in the above formula. The interpolation of $\Eis^{[f, g, j]}_{\et}$ is provided by the Beilinson--Flach class $\cBF^{\bff, \bfg}_1$. The Bloch--Kato logarithm maps $\log(\dots)$ can be interpolated using Perrin-Riou's ``big logarithm'' map $\cL$ (using an extension of Perrion-Riou's construction due to the second and third authors). Hida's $p$-adic Rankin--Selberg $L$-function $L_p(f, g, 1 + j)$ extends to a $p$-adic analytic function of all three variables by construction.

  The most difficult terms to deal with are the de Rham cohomology classes $\eta_f$ and $\omega_g$, since their definition involves the $p$-adic Eichler--Shimura isomorphism relating de Rham and \'etale cohomology. To interpolate these as $f$ and $g$ vary in the families $\bff, \bfg$, we use the $\Lambda$-adic Eichler--Shimura isomorphism of Ohta to construct interpolating classes $\eta_{\bff}$ and $\omega_{\bfg}$. Unfortunately, the interpolating property of Ohta's construction is not quite strong enough for our purposes, so a substantial part of this paper (\S \ref{sect:ohtacompat}) is devoted to proving an additional interpolating property of $\eta_{\bff}$ and $\omega_{\bfg}$, by a somewhat indirect method involving Kato's explicit reciprocity law and the variation of Kato's Euler system in Hida families. (We also have a second, more direct proof of this compatibility using Faltings' Hodge--Tate decomposition for modular forms, which we plan to treat in a subsequent paper.)

  Once all these preparations are in place, the proof of the explicit reciprocity law is virtually trivial. We know that both the $p$-adic $L$-function, and the value of the pairing
  \[ \left\langle \cL\left({}_c \BF^{\bff, \bfg}_1\right), \eta_\bff \otimes \omega_\bfg\right\rangle, \]
  are $p$-adic analytic functions of the three variables $(k, k', j)$; and the main result of \cite{KLZ1a} shows that these two analytic functions agree at all triples of integers $(k, k', j)$ satisfying the inequality $0 \le j \le \min(k, k')$. Since this set of triples is Zariski-dense, the two functions must agree everywhere, which is Theorem \ref{lthm:explicitrecip}.

  \subsubsection*{Comparison with the approach of \cite{BDR-BeilinsonFlach2}} As mentioned above, a ``1-variable'' analogue of this explicit reciprocity law (with $f$ fixed and $j = 0$, and only $g$ varying in a family $\bfg$) has been proved by Bertolini et al.\ in \cite{BDR-BeilinsonFlach2}. Their approach also uses syntomic cohomology to obtain the result for many specialisations of the family, and analytic continuation to obtain the result everywhere; but there is a significant difference between their proof and ours, in that they analytically continue from specialisations of weight 2 and high $p$-power level, rather than high weight and prime-to-$p$ level as in our approach. Thus their strategy requires a delicate study of the special fibres of the modular curves $X_1(Np^r)$ in characteristic $p$, which is not needed in our approach; and our method is also amenable to generalisations to non-ordinary Coleman families, as we shall show in a subsequent paper \cite{loefflerzerbes16}.

 \subsection*{Acknowledgements}

  This paper grew out of a collaboration begun at the workshop ``Applications of Iwasawa Algebras'' at the Banff International Research Station, Canada, in March 2013. We are very grateful to BIRS for their hospitality, and to the organisers of the workshop for the invitation. Much of the final draft was prepared while the second and third authors were visiting the Mathematical Sciences Research Institute in Berkeley, California, for the programme ``New Geometric Methods in Automorphic Forms'', and it is again a pleasure to thank MSRI for their support and the organisers of the programme for inviting us to participate.

  During the preparation of this paper, we benefitted from conversations with a number of people, notably Fabrizio Andreatta, Massimo Bertolini, Hansheng Diao, Henri Darmon, Samit Dasgupta, Adrian Iovita, Victor Rotger, Karl Rubin and Chris Skinner. We would particularly like to thank Adrian Iovita for making us aware of the work of Delbourgo \cite{delbourgo08} which inspired the proof of Theorem \ref{thm:ohtacompat}. The authors are also very grateful to the referee for his or her extremely careful reading of the text and many helpful suggestions, which resulted in an overall improvement of the paper.


\section{Setup and notation}

 \subsection{Cohomology theories}
  \label{sect:cohomology}

  \subsubsection{\'Etale cohomology}

   In this paper we work with continuous \'etale cohomology as defined by Jannsen \cite{jannsen88}. More specifically, for a pro-system $\sF\coloneqq(\sF_r)_{r\ge 1}$ of \'etale sheaves on a scheme $S$,   indexed by integers $r\ge 1$, we let $H^i_\et(S,\sF)$ be the $i$-th derived functor of $\sF \mapsto\varprojlim_r H^0_{\et}(S,\sF_r)$. We note for later use that if $H^{i-1}_{\et}(S,\sF_r)$ is finite for all $r$, then by \cite[Lemma 1.15, Equation (3.1)]{jannsen88} one has
   \begin{equation}
    \label{torsors-as-limits}
    H^i_{\et}(S,\sF)\cong \varprojlim_r H^i_{\et}(S,\sF_r).
   \end{equation}

   This, in particular, includes the case of pro-systems $(\sF_r)_{r\ge 1}$ where each $\sF_r$ is constructible, and $S$ is of finite type over one of the following rings:
   \begin{itemize}
    \item an algebraically closed field of characteristic 0;
    \item a local field of characteristic 0;
    \item a ring of $S$-integers $\cO_{K, S}$, where $K$ is a number field and $S$ is a finite set of places of $K$, including all places that divide the order of $\sF_r$ for any $r$.
   \end{itemize}
   This covers all the cases we shall use in this paper. (In practice, all our $(\sF_r)$ will be inverse systems of finite $p$-torsion sheaves for a prime $p$, so in the third case we need only assume that $S$ contains all places dividing $p$.)


  \subsubsection{Pushforward maps}
   \label{sect:adjunction}

   Let $X$, $Y$ be schemes, both smooth of finite type over some base $S$ as above, and $\sF$, $\mathscr{G}$ constructible \'etale sheaves (or pro-systems of such sheaves) on $X$ and $Y$ respectively. Then we define a ``\emph{pushforward morphism $(X, \sF) \to (Y, \mathscr{G})$}'' to be the data of a morphism of $S$-schemes $f: X \to Y$, and a pair $\phi$ of morphisms
   \[ \phi_\flat: f_! \sF \to \mathscr{G}, \quad  \phi_\sharp: \sF \to f^! \mathscr{G}\]
   of \'etale sheaves on $Y$ (resp.\ $X$) which correspond to each other under the adjunction $f_! \leftrightarrows f^!$. In general this only makes sense at the level of derived categories, but we shall only use this construction when $f$ is finite \'etale, in which case $f_!$ and $f^!$ agree with the usual direct and inverse image \cite[Exp. XVIII, Prop. 3.1.8]{SGA4}. Thus we obtain maps
   \[ (f, \phi)_*: H^i_{\et}(X, \sF) \to H^{i}_{\et}(Y, \mathscr{G}), \]
   which can be expressed either as
   \begin{gather*}
    H^i_{\et}(X, \sF) \rTo^\cong H^{i}_{\et}(Y, f_*\sF) \rTo^{\phi_\sharp} H^{i}_{\et}(Y, \mathscr{G})
    \quad\text{or}\quad \\
    H^i_{\et}(X, \sF) \rTo^{\phi_\flat} H^i_{\et}(X, f^* \mathscr{G}) \rTo^\cong H^i_{\et}(Y, f_* f^* \mathscr{G})\rTo^{\operatorname{tr}_f} H^i_\et(Y, \mathscr{G}).
   \end{gather*}
   If the pair $\phi = (\phi_\sharp, \phi_\flat)$ is clear from context we shall omit it from the notation and write simply $f_*$.

   We will also need to consider the case where $f: X \into Y$ is a closed immersion, in which case $f^! \mathscr{G}$ is isomorphic to $f^* \mathscr{G}(-c)[-2c]$ where $c$ is the codimension of $X$ in $Y$, by the relative purity theorem \cite[Exp. XVI \S3]{SGA4}; so we obtain pushforward maps
   \[ H^i_{\et}(X, f^* \mathscr{G}) \to H^{i + 2c}_{\et}(Y, \mathscr{G}(c)). \]

  \subsubsection{De Rham cohomology}

   We shall also work with algebraic de Rham cohomology (for varieties over fields of characteristic 0), with coefficients in locally free sheaves equipped with a filtration and an integrable connection $\nabla$.

   In the case of a $p$-adic base field $K$, and constant coefficient sheaves, we will use frequently the Faltings--Tsuji comparison isomorphism, which is a canonical isomorphism of graded $K$-vector spaces
   \[
    \comp_{\dR}: H^i_{\dR}(X / K) \cong \DD_{\dR}\left( H^i_{\et}(X_{\overline{K}}, \Qp) \right)
   \]
   which is natural in $X$.


 \subsection{Multilinear algebra}
  \label{sect:multlin}

  If $H$ is an abelian group, we define the modules $\TSym^k H$, $k \ge 0$, of symmetric tensors with values in $H$ following \cite[\S 2.2]{Kings-Eisenstein}. By definition, $\TSym^k H$ is the submodule of $\frS_k$-invariant elements in the $k$-fold tensor product $H \otimes \dots \otimes H$ (while the more familiar $\Sym^k H$ is the module of $\frS_k$-coinvariants).

  The direct sum $\bigoplus_{k \ge 0} \TSym^k H$ is equipped with a ring structure via symmetrisation of the naive tensor product, so for $h \in H$ we have
  \begin{equation}
   \label{TSymmult}
   h^{\otimes m} \cdot h^{\otimes n} = \frac{(m + n)!}{m! n!} h^{\otimes(m + n)}.
  \end{equation}

  \begin{remark}
   There is a natural ring homomorphism $\Sym^\bullet H \to \TSym^\bullet H$, which becomes an isomorphism in degrees up to $k$ after inverting $k!$. However, we will be interested in the case where $H$ is a $\Zp$-module, for a fixed $p$, and $k$ varying in a $p$-adic family, so we cannot use this fact without losing control of the denominators involved; so we shall need to distinguish carefully between $\TSym$ and $\Sym$.
  \end{remark}

  Note that in general $\TSym^k$ does not commute with base change and hence does not sheafify well. In the cases where we consider $\TSym^k (H)$, $H$ is always a free module over the relevant coefficient ring so that this functor coincides with $\Gamma^k(H)$, the $k$-th divided power of $H$. This functor sheafifies (on an arbitrary site), so that the above definitions and constructions carry over to sheaves of abelian groups.

  In particular, for $X$ a regular $\ZZ[1/p]$-scheme, and $\sF$ a locally constant \'etale sheaf of $(\ZZ / p^n \ZZ)$-modules on $X$, we can define \'etale sheaves $\TSym^k \sF$ for any $k \ge 0$. Similarly, if $X$ is a variety over a characteristic 0 field and $\sF$ is a locally free sheaf on $X$, we can make sense of $\TSym^k \sF$ as a locally free sheaf on $X$, and if $\sF$ is equipped with a filtration and a connection these naturally give rise to analogous structures on $\TSym^k \sF$. Thus $\TSym^k(-)$ makes sense on the coefficient categories for both \'etale and de Rham cohomology.


 \subsection{Modular curves}
  \label{sect:modularcurves}

  We recall some notations for modular curves, following \cite[\S\S 1--2]{Kato-p-adic}. For integers $N, M\ge 1$ with $M + N\ge 5$ we define $Y(M,N)$ to be the $\ZZ[1/MN]$-scheme representing the functor
  \[
   S\mapsto\{
   \mbox{isomorphism classes }(E, e_1, e_2)\}
  \]
  where $S$ is a $\ZZ[1/MN]$-scheme, $E/S$ is an elliptic curve, $e_1,e_2\in E(S)$ and $\beta:(\ZZ/M\ZZ)\times(\ZZ/N\ZZ)\to E$, $(m,n)\mapsto (me_1+ne_2)$ an injection. When considering these curves we will always assume that $M \mid N$; then there is a left action of the group
  \[ \left\{ g \in \GL_2(\ZZ / N \ZZ) : g \equiv \tbt * * 0 * \bmod \tfrac{N}{M} \right\} \]
  on the curve $Y(M, N)$, cf.\ \cite[\S 2.1]{LLZ14}. We shall write $Y_1(N)$ for $Y(1, N)$.

  In order to define Hecke operators, we will also need the modular curves $Y(M,N(A))$ and $Y(M(A),N)$, for $A \ge 1$, which were introduced
  by Kato \cite[\S 2.8]{Kato-p-adic}. These are $\ZZ[1/AMN]$-schemes; the scheme $Y(M, N(A))$ represents the functor
  \[ S \mapsto \{\mbox{isomorphism classes}(E,e_1,e_2, C)\}\]
  where $(E, e_1, e_2) \in Y(M, N)(S)$ and $C$ is a cyclic subgroup of order $AN$ such that $C$ contains $e_2$ and is complementary to $e_1$ (i.e.\ the map $\ZZ/M\ZZ\times C\to E$, $(x,y)\mapsto x e_1 + y$ is injective). Similarly, $Y(M(A), N)$ classifies $(E, e_1, e_2, C)$ where $C$ is a cyclic subgroup scheme of order $AM$ containing $e_1$ and complementary to $e_2$.

  We use the same analytic uniformisation of $Y(M, N)(\CC)$ as in \cite[1.8]{Kato-p-adic}. Let
  \[ \Gamma(M, N) = \left\{ g \in \SL_2(\ZZ): g = 1 \bmod \stbt{M}{M}{N}{N}\right\};\]
  then we have
  \begin{align*}
   (\ZZ/M\ZZ)^* \times \Gamma(M, N)\backslash \HH\cong Y(M, N)(\CC),&&
   (a,\tau)\mapsto \left(\tfrac{\CC}{\ZZ\tau+\ZZ},\tfrac{a\tau}{M},\tfrac{1}{N}\right)
  \end{align*}
  where $\HH$ is the upper half plane. There are similar uniformisations of the curves $Y(M(A), N)$ and $Y(M, N(A))$.

  Let $\Tate(q)$ be the Tate curve over $\ZZ((q))$, with its canonical differential
  $\omega_{\mathrm{can}}$. Let $\zeta_N\coloneqq e^{2\pi i/N}$ and $q_M \coloneqq q^{1/M}$. Then we define a point of $Y(M, N)$ over $\ZZ[\tfrac1N, \zeta_N]((q^{1/M}))$ by
  \[ \infty \coloneqq \left( \Tate(q), q_M, \zeta_N \right).\]
  This is compatible with the Fourier series in the complex-analytic
  theory if one makes the usual identifications $q_M = e^{2\pi i \tau / M}$, $\zeta_N = e^{2\pi i / N}$. Note that even for $M = 1$, the uniformiser $q = q_1$ at the cusp $\infty$ is only defined over $\ZZ[1/N, \zeta_N]$.
%

  All the modular curves $Y$ we consider correspond to \emph{representable} moduli problems, and are hence equipped with universal elliptic curves $\pi: \cE \to Y$. We use this to construct coefficient sheaves on $Y$. In the \'etale case, after inverting $p$ if necessary we define
  \[ \sH_{\Zp} = \left(R^1 \pi_* \Zp\right)^\vee = R^1 \pi_* \Zp(1), \]
  which is a lisse \'etale $\Zp$-sheaf of rank 2 on $Y[1/p]$, and can be identified with the relative Tate module $T_p(\cE)$. We write $\sH_{\Qp}$ and $\sH_r$ ($r \ge 1$) for the corresponding sheaves with $\Qp$ or $\ZZ / p^r \ZZ$ coefficients. In the de Rham setting, after base-extension to $\QQ$ we have a line bundle $\sH_{\dR}$, which is equipped with its Hodge filtration and Gauss--Manin connection. Applying the multilinear algebra theory of \S \ref{sect:multlin} gives us $\Zp$-sheaves $\TSym^k \sH_{\Zp}$ for each $k \ge 0$, and similarly for $\sH_r, \sH_{\Qp}, \sH_{\dR}$.

  After base-changing to $\Qp$, the de Rham and \'etale cohomology groups are related by a comparison isomorphism: there is a canonical isomorphism
  \begin{equation}
   \label{eq:comparisoniso}
   \comp_{\dR}: H^i_{\dR}\left(Y_{\Qp}, \TSym^k \sH_{\dR}\right)
   \rTo^\cong \DD_{\dR}\left(H^i_{\et}\left(Y_{\overline{\QQ}_p},
   \TSym^k \sH_{\Qp}\right)\right).
  \end{equation}
  For $k = 0$ this is simply the Faltings--Tsuji comparison map of \S \ref{sect:cohomology} applied to $Y$. We extend this to $k > 0$ by identifying both sides with direct summands of the cohomology of the variety $\cE^k$. See \cite[Remark 3.2.4]{KLZ1a}.


 \subsection{Degeneracy maps and Hecke operators}
  \label{sect:degeneracy}

  \begin{definition}
   For $M, N, A$ integers with $M + N \ge 5$, we consider the following maps:
   \begin{enumerate}
    \item The maps $\pr_1$ and $\pr_2: Y(M, NA) \to Y(M, N)$ are defined by
    \[ \pr_1(E, e_1, e_2) = (E, e_1, Ae_2),\quad \pr_2(E, e_1, e_2) = (E / \langle Ae_2\rangle, e_1 \bmod Ae_2, e_2 \bmod Ae_2). \]
    \item The maps $\hpr_1$ and $\hpr_2: Y(MA, N) \to Y(M, N)$ are defined by
    \[ \hpr_1(E, e_1, e_2) = (E, Ae_1, e_2),\quad \hpr_2(E, e_1, e_2) = (E / \langle Ae_1\rangle, e_1 \bmod Ae_1, e_2 \bmod Ae_1). \]
   \end{enumerate}
  \end{definition}

  Note that $\pr_1$ and $\hpr_1$ correspond to the identity map on $(\ZZ/M\ZZ)^* \times \HH$ under the complex uniformisation, while $\pr_2$ and $\hpr_2$ correspond to $(x, z) \mapsto (x, Az)$ and $(x, A^{-1}z)$ respectively.

  \begin{definition}
   We write $\pr$ and $\pr'$ for the natural degeneracy maps
   \[Y(M, AN) \rTo^{\pr'} Y(M, N(A)),\qquad Y(M, N(A)) \rTo^{\pr} Y(M, N) \]
   whose composition is $\pr_1$, and similarly $\hpr$ and $\hpr'$.
  \end{definition}

  More subtly, there is an isomorphism
  \[\begin{aligned}
   \varphi_A:Y(M,N(A))& \rTo^\cong Y(M(A),N)\\
    (E,e_1,e_2,C)&\rMapsto (E',e_1',e_2',C')
  \end{aligned}\]
  with $E'\coloneqq E/NC$, $e_1'$ the image of $e_1$, $e_2'$ is the image of
  $[A]^{-1}(e_2)\cap C$ in $E'$ and $C'$ is the image of $[A]^{-1}\ZZ e_1$
  in $E'$. In the other direction, we have a similarly-defined map $\varphi_{A^{-1}}: Y(M(A),N) \to Y(M, N(A))$. These maps $\varphi_A$ and $\varphi_{A^{-1}}$ correspond to multiplication by $A$ (resp.\ $A^{-1}$) on $\HH$, and we have
  \[ \pr_2 = \hpr \circ \varphi_A \mathop\circ \pr',\quad \hpr_2 = \pr \mathop\circ \varphi_A^{-1} \circ \hpr'.\]

  Letting $\cE_1$ and $\cE_2$ denote the universal elliptic curves over $Y(M, N(A))$ and $Y(M(A), N)$ respectively, there are canonical isogenies
  \[ \lambda: E_1 \to \varphi_A^*(\cE_2),\quad \lambda': E_2 \to \varphi_{A^{-1}}^*(\cE_1) \]
  which both have cyclic kernels of order $A$, and which are dual to each other (that is, the isogeny $\varphi_A^*(\cE_2) \to E_1$ dual to $\lambda$ is the pullback of $\lambda'$ via $\varphi_A$, and vice versa). Hence the compositions
  \[ \Phi_A^* \coloneqq \lambda^* \circ \varphi_A^* \quad\text{and}\quad (\Phi_{A^{-1}})_* \coloneqq (\varphi_A)_* \circ (\lambda')_* \]
  agree as morphisms $H^i_{\et}\Big(Y(M(\ell), N), \TSym^k \sH_{\Zp}(j)\Big) \to H^i_{\et}\Big(Y(M, N(\ell)), \TSym^k \sH_{\Zp}(j)\Big)$, for any $i, k \ge 0$ and $j \in \ZZ$.

  \begin{definition}
   For $\ell$ prime, and any $i, j, k$, we define the Hecke operator $T_\ell'$ (for $\ell \nmid MN$) or $U_\ell'$ (for $\ell \mid MN$) acting on $H^i_{\et}(Y(M, N), \TSym^k\sH_{\Zp}(j))$ as the composite
   \[ T_\ell' = (\pr)_* \circ (\Phi_\ell)^* \circ (\hpr)^* = (\pr)_* \circ (\Phi_{\ell^{-1}})_* \circ (\hpr)^*. \]
  \end{definition}

  We also have Hecke operators $T_\ell = (\hpr)_* \circ (\Phi_{\ell^{-1}})^* \circ (\pr)^* =  (\hpr)_* \circ (\Phi_{\ell})_* \circ (\pr)^*$, which are the transposes of the $T_\ell'$ with respect to Poincar\'e duality; but we shall not use these so heavily in the present paper. (Note, however, that it is the $T_\ell$ rather than the $T_\ell'$ that correspond to the familiar formulae for the action on $q$-expansions.)

  We will also need the following observation:

  \begin{lemma}
   \label{lemma:isogeny}
   Let $\varphi:E\rightarrow E'$ be an isogeny between elliptic curves, and denote by $\langle\quad,\quad \rangle_{E[p^r]}$ and $\langle\quad,\quad\rangle_{E'[p^r]}$ the Weil pairings on the $p^r$-torsion points of $E$ and $E'$, respectively. If $P,Q\in E[p^r]$, then
   \[ \langle \varphi(P),\varphi(Q)\rangle_{E'[p^r]}=\left(\langle P,Q\rangle_{E[p^r]}\right)^{\deg(\varphi)}.\]
  \end{lemma}

  Hence the maps on $\bigwedge^2 \sH_{\Zp} \cong \Zp(1)$ induced by $(\Phi_A)_*$ and $(\Phi_{A^{-1}})_*$ are both equal to multiplication by $A$.

  We will need the following compatibility between pushforward maps and Hecke operators:

  \begin{proposition}
   \label{prop:pridentity1}
   As morphisms
   \[ H^1(Y(M, Np), \TSym^k \sH_{\Zp}(j)) \to H^1(Y(M, N), \TSym^k \sH_{\Zp}(j)),\]
   for any $k \ge 0$ and $j\in\ZZ$, we have
   \begin{align*}
    (\pr_2)_* \circ U_p' &= p^{k+1} (\pr_1)_*, \\
    (\pr_1)_* \circ U_p' &= T_p' \circ (\pr_1)_* - {\stbt p 0 0 {p^{-1}}}^* \circ (\pr_2)_*.
   \end{align*}
  \end{proposition}

  \begin{proof}
   Explicit calculation.
  \end{proof}

 \subsection{Atkin--Lehner operators}

  We will also need to consider Atkin--Lehner operators. We first give the definitions in classical terms, working with $2 \times 2$ matrices. Let $N \ge 1$ and let $\Gamma = \Gamma_1(N)$, or more generally any subgroup of the form $\Gamma_1(R(S)) = \Gamma_1(R) \cap \Gamma_0(RS)$ with $RS = N$. We let $Y$ be the corresponding modular curve, so $Y(\CC) = \Gamma \backslash \HH$.

  \begin{notation}\mbox{~}
   \begin{itemize}
    \item  We shall write $\mathcal{G}$ for the quotient $N_{\GL_2^+(\QQ)}(\Gamma) / \Gamma$.

    \item We use the notation $Q \parallel N$, for an integer $Q \ge 1$, to mean that $Q \mid N$ and $(Q, \tfrac{N}{Q}) = 1$.

    \item If $Q \parallel N$ and $x \in (\ZZ / Q\ZZ)^\times$, then we write $\langle x \rangle_Q$ for the class in $\mathcal{G}$ of any element of $\SL_2(\ZZ)$ of the form $\stbt a b {Nc} d$ with $d = x \bmod Q$ and $d = 1 \bmod \tfrac{N}{Q}$.
   \end{itemize}
  \end{notation}

  \begin{definition}
   For $N \ge 1$ and $Q \parallel N$, and we define $W_Q$ to be the class in $\mathcal{G}$ of any matrix $\stbt {Qx}{y}{Nz}{Qw}$, where $x, y,z,w$ are integers such that $Qxw - \tfrac{N}{Q}yz = 1$, $Qx = 1 \bmod \tfrac N Q$ and $y = -1 \bmod Q$.
  \end{definition}

  One verifies easily that in the group $\mathcal{G}$ one has the relations
  \begin{itemize}
   \item $W_Q^2 = \stbt Q00Q \cdot \langle Q \rangle_{N/Q} \cdot \langle -1 \rangle_Q$.

   \item \( \langle d \rangle_Q \cdot \langle d' \rangle_{N/Q} \cdot W_Q = W_Q \cdot \langle d^{-1} \rangle_Q \cdot \langle d' \rangle_{N/Q}\) for any $d \in (\ZZ / Q\ZZ)^\times$ and $d' \in (\ZZ / \tfrac{N}{Q}\ZZ)^\times$.

   \item If $Q$ and $Q'$ are integers such that $Q \parallel N$ and $Q' \parallel \tfrac{N}{Q}$, then we have
   $W_{QQ'} = \langle Q'\rangle_Q \cdot W_Q \cdot W_{Q'}$ (cf.\ \cite[Prop.\ 1.4]{atkinli78}).
  \end{itemize}

  \begin{remark}
   \label{rmk:conventions}
   Note that our conventions differ somewhat from \cite{atkinli78}, where the convention chosen is $y = 1 \bmod Q$ and $x = 1 \bmod \tfrac N Q$. Thus the matrix $W_Q^{\operatorname{AL}}$ considered by Atkin and Li is $W_Q \cdot \langle -1 \rangle_Q \cdot \langle Q^{-1} \rangle_{N/Q}$ in our notation.
  \end{remark}

  The action of $\GL_2^+(\QQ)$ on the upper half-plane $\HH$ descends to an action of the quotient group $\mathcal{G}$ on $Y(\CC) = \Gamma \backslash \HH$. We can extend this action to the universal elliptic curve $\cE / Y$ (where this is defined) via the identification
  \[ \cE(\CC) = \Gamma_1(N) \backslash \left( \HH \times \CC / \sim \right)\]
  where $\sim$ is the equivalence relation given by $(\tau, z) \sim (\tau, z + m + n\tau)$ for all $m, n \in \ZZ$ (cf.\ \cite[\S 1.5.9]{FukayaKato-sharifi-conj}). The submonoid  $\mathcal{G}^+$ of $\mathcal{G}$ consisting of matrices of integer determinant acts on $\HH \times \CC / \sim$ via
  \[ \stbt a b c d \cdot (\tau, z) = \left( \frac{a\tau + b}{c \tau + d}, \frac{(ad - bc) z}{c\tau + d}\right), \]
  and this induces an action of the Atkin--Lehner operators $W_Q$ on $\cE(\CC)$, acting as a cyclic isogeny of degree $Q$ on the fibres.

  These operators have an algebraic interpretation in terms of moduli spaces. For simplicity we restrict to $\Gamma = \Gamma_1(N)$ here; the more general case of modular curves of the form $Y_1(R(S))$ with $RS = N$ may be deduced by passage to the quotient. For each $Q \parallel N$, we may clearly identify $Y_1(N)$ with the moduli space of triples $(E, P_{N/Q}, P_{Q})$ where $P_{N/Q}$ and $P_Q$ have exact order $N/Q$ and $Q$ respectively.

  \begin{definition}
   The Atkin--Lehner map $W_Q$ is the automorphism of the scheme
   \( Y_1(N) \times_{\ZZ\left[1/N\right]} \ZZ\left[\tfrac 1 N, \zeta_Q\right] \)
   given by
   \[ (E, P_{N/Q}, P_Q) \mapsto \left(E/\langle P_Q \rangle, P_{N/Q} \bmod \langle P_Q \rangle, P_Q' \right),\]
   where $P_Q' \in E[Q] / \langle P_Q \rangle$ is the unique class such that $\langle P_Q, P_Q'\rangle_{E[Q]} = \zeta_Q$.
  \end{definition}

  This extends in a natural way to the universal elliptic curve $\cE$, and on base-extension to $\CC$ it coincides with the complex-analytic description given above.


 \subsection{Modular forms}
  \label{sect:modformnotation}

  For $N \ge 1$, we write $M_k(N, \CC)$ for the space of modular forms of weight $k$ and level $\Gamma_1(N)$; and we write $M_k(N, \ZZ)$ for the subspace consisting of modular forms whose $q$-expansions have integer coefficients. More generally, we write $M_k(N, R) = M_k(N, \ZZ) \otimes R$ for any commutative ring $R$. We write $S_k \subseteq M_k$ for the cusp forms.

  For any $k \ge 0$ there is a canonical isomorphism
  \[ M_{k + 2}(N, \CC) \cong \Fil^1 H^1_{\dR}\left(Y_1(N)_{\CC}, \Sym^k \sH_{\dR}^\vee\right), \]
  which maps $f \in M_{k + 2}(N, \CC)$ to the class of the $(\Sym^k \sH_{\dR}^\vee)$-valued differential (with logarithmic growth at the cusps) given by
  \[ \omega_f \coloneqq (2\pi i)^{k + 1} f(\tau) (\mathrm{d}z)^k \mathrm{d}\tau.\]

  \begin{remark}
   With our conventions, $M_{k + 2}(N, \QQ)$ does \emph{not} map to de Rham cohomology of $Y_1(N)_{\QQ}$, because in our model the cusp $\infty$ is not defined over $\QQ$. Rather, the de Rham cohomology of $Y_1(N)_{\QQ}$ corresponds to the elements of $M_{k + 2}(N, \QQ(\zeta_N))$ which satisfy the Galois-equivariance property
   \[ f^\sigma = \langle \chi(\sigma) \rangle f \]
   for all $\sigma \in \Gal(\QQ(\zeta_N) /\QQ)$, where $\chi$ denotes the mod $N$ cyclotomic character. The Atkin--Lehner operator $W_N$ interchanges this space with $M_{k + 2}(N, \QQ)$.
  \end{remark}


 \subsection{Rankin L-functions}
  \label{sect:Rankin-L-fct}

  Let $f$, $g$ be cuspidal eigenforms of weights $r, r' \ge 1$, levels $N_f, N_g$ and characters $\varepsilon_f,\varepsilon_g$. We define the Rankin $L$-function
  \[ L(f, g, s) \coloneqq L_{(N_f N_g)} (\varepsilon_f \varepsilon_g, 2s + 2 - r - r') \cdot \sum_{n \ge 1} a_n(f) a_n(g) n^{-s}, \]
  where $L_{(N_f N_g)} (\varepsilon_f \varepsilon_g,s)$ denotes the Dirichlet $L$-function with the Euler factors at the primes dividing $N_f N_g$ removed. This Dirichlet series differs by finitely many Euler factors from the $L$-function of the automorphic representation $\pi_f \otimes \pi_g$ of $\GL_2 \times \GL_2$ associated to $f$ and $g$. In particular, it has mero\-morphic continuation to all of $\CC$. It is holomorphic on $\CC$ unless $\langle \bar{f}, g \rangle \ne 0$, in which case it has a pole at $s = r$.

  More generally, for a primitive Dirichlet character $\chi$ of conductor $N_\chi$ we define
  \[ L(f, g, \chi, s) \coloneqq L_{(N_f N_g N_\chi)} (\chi^2 \varepsilon_f \varepsilon_g, 2s + 2 - r - r') \sum_{\mathclap{\substack{n \ge 1 \\ (n, N_\chi) = 1}}} \chi(n) a_n(f) a_n(g) n^{-s}.\]

  \begin{remark}
   If $f$ and $g$ are normalised newforms and the three integers $N_f, N_g, N_\chi$ are pairwise coprime, then $L(f, g, s) = L(\pi_f \otimes \pi_g \otimes \chi, s)$.
  \end{remark}

  \begin{theorem}[{Shimura, see \cite[Theorem 4]{Shimura-periods}}]
   If $f$ and $g$ are normalised newforms, with $q$-expansion coefficients in a number field $L$, and $\chi$ takes values in $L$, then for integer values of $s$ in the range $r' \le s \le r-1$, the ratio
   \[
    \frac{L(f, g, \chi, s) }{(2\pi i)^{2s + 1 - r'}  G(\chi)^2 G(\varepsilon_f) G(\varepsilon_g) i^{1-r} \langle f, f \rangle}
   \]
   lies in $L$, and depends Galois-equivariantly on $(f, g, \chi)$.
  \end{theorem}

  Here the Gauss sum of a character $\chi$ is defined by
  \[  G(\chi) \coloneqq \sum_{a \in (\ZZ / C \ZZ)^\times} \chi(a) e^{2\pi i a/ C} \]
  where $C$ is the conductor of $\chi$; and $\langle f, f \rangle_{N_f}$ is the norm of $f$ with respect to the Petersson inner product defined by
   \[ \langle f_1, f_2 \rangle_N \coloneqq \int_{\Gamma_1(N) \backslash \HH} \overline{f_1(\tau)} f_2(\tau) y^{r - 2}\, \mathrm{d}x\, \mathrm{d}y, \]
   where $\tau = x + iy$. The proof of this statement uses the Rankin--Selberg integral formula
  \[
   L(f, g, s) =  \frac{N^{r + r' - 2s - 2} \pi^{2s + 1 - r'} (-i)^{r - r'} 2^{2s + r - r'}}{\Gamma(s)\Gamma(s - r' + 1)}
      \big\langle f^*, g E^{(r - r')}_{1/N}(\tau, s - r + 1)\big\rangle_N
  \]
  where $N \ge 1$ is some integer divisible by $N_f$ and $N_g$ and with the same prime factors as $N_f N_g$, and $E^{(r - r')}_{1/N}(\tau, s - r + 1)$ is a certain real-analytic Eisenstein series (cf.\ \cite[Definition 4.2.1]{LLZ14}), whose values for $s$ in this range are nearly-holomorphic modular forms with $q$-expansions in $\QQ(\zeta_N)$.

  \begin{notation}
   If $p$ is a prime not dividing $N_f$, let $\alpha_f$ and $\beta_f$ be the roots of the ``Hecke polynomial'' $X^2 - a_p(f) X + p^{k + 1} \varepsilon_f(p)$ (and similarly for $g$).
  \end{notation}

  \begin{theorem}[Hida, Panchishkin]
   \label{thm:hida}
   Let $p \ge 5$ be prime with $p \nmid N_f N_g$, and let $\frP$ be a prime of $L$ above $p$ at which $f$ is ordinary. Suppose $\alpha_f$ is the unit root of the Hecke polynomial. Then there is a $p$-adic $L$-function
   \[ L_p(f, g) \in \QQ(\mu_N) \otimes_{\QQ} \Lp \otimes_{\Zp} \Zp[[\ZZ_p^\times]] \]
   with the following interpolation property: for $s$ an integer in the range $r' \le s \le r-1$, and $\chi$ a Dirichlet character of $p$-power conductor, we have
   \begin{align*}
     L_p(f, g, s + \chi) &= \frac{\cE(f, g, s + \chi)}{\cE(f) \cE^*(f)} \cdot \frac{\Gamma(s)\Gamma(s - r' + 1)} {\pi^{2s + 1 - r'} (-i)^{r - r'} 2^{2s + r - r'} \langle f, f \rangle_{N_f}} \cdot L(f, g, \chi^{-1}, s),
   \end{align*}
   where the Euler factors are defined by
   \begin{gather*}
    \cE(f) = \left( 1 - \frac{\beta_f}{p \alpha_f}\right), \qquad \cE^*(f) = \left( 1 - \frac{\beta_f}{\alpha_f}\right),\\
    \cE(f, g, s + \chi) =
    \begin{cases}
     \left( 1 - \frac{p^{s-1}}{\alpha_f \alpha_g}\right) \left( 1 - \frac{p^{s-1}}{\alpha_f \beta_g}\right) \left( 1 - \frac{\beta_f \alpha_g}{p^s}\right) \left( 1 - \frac{\beta_f \beta_g}{p^s}\right) & \text{if $\chi$ is trivial},\\[2mm]
     G(\chi)^2 \cdot \left( \frac{ p^{2s-2} }{\alpha_f^2 \alpha_g \beta_g}\right)^t & \text{if $\chi$ has conductor $p^t > 1$.}
    \end{cases}
   \end{gather*}
  \end{theorem}

  Here we write ``$s + \chi$'' for the character of $\ZZ_p^\times$ defined by $z \mapsto z^s \chi(z)$. In the above statement we are taking $f$ and $g$ to be fixed, but in fact Hida has shown that $L_p(f, g, s)$ varies analytically as $f$ and $g$ vary through Hida families; see  Theorem \ref{thm:hida2} below for a precise statement.

  \begin{remark}\label{remark:DpversusLp}\mbox{~}
   \begin{enumerate}

    \item The $L$-function $L_p(f, g, s)$ above is $N^{r + r' - 2s - 2} \mathcal{D}_p(\breve f, \breve g, 1/N, s)$ in the notation of \cite[\S 5]{LLZ14}, where $\breve f$ and $\breve g$ are the pullbacks of $f, g$ to level $N$. We include the power of $N$ in the definition because it makes $L_p(f, g, s)$ independent of the choice of $N$.

    \item The interpolating property of $L_p(f, g, s)$ only makes sense if $r > r'$, but one can define $L_p(f, g, s)$ for any $f, g$ using interpolation in a Hida family.

    \item The complex $L$-function $L(f, g, s)$ is symmetric in $f$ and $g$, i.e.\ we have $L(f, g, s) = L(g, f, s)$; but this is not true of $L_p(f, g, s)$.

    \item One can check from Shimura's theorem that the quotient $\frac{L_p(f, g, s)}{G(\varepsilon_f) G(\varepsilon_g)}$ lies in $\Lp \otimes_{\Zp} \Zp[[\Gamma]]$, and depends Galois-equivariantly on $f$ and $g$.

    \item The construction of $L_p(f, g, s)$ has recently been extended to the non-ordinary case by Urban \cite{Urban-nearly-overconvergent}, who has constructed a three-para\-meter $p$-adic $L$-function with $f$, $g$ varying over the Coleman--Mazur eigencurve; but we shall only consider the case of ordinary $f, g$ in this paper.
   \end{enumerate}
  \end{remark}

  For applications to the Iwasawa main conjecture, we shall need the following non-vanishing result:

  \begin{proposition}
   \label{prop:shahidi}
   If $r - r' \ge 2$, then the $p$-adic $L$-function $L_p(f, g)$ is not a zero divisor in the ring $\QQ(\mu_N) \otimes_{\QQ} \Lp \otimes_{\Zp} \Zp[[\ZZ_p^\times]]$.
  \end{proposition}

  \begin{proof}
   Let us first assume $r - r' \ge 3$. Then the Euler product for the $L$-function $L(\pi_f \otimes \pi_g, s)$ converges for $\Re(s) > \frac{r + r'}{2}$, and in this range, no term in this product is zero; hence the $L$-value does not vanish. In particular, it is non-vanishing at $s = r-1$. The same holds if $\pi_f \times \pi_g$ is replaced by $\pi_f \otimes \pi_g \otimes \chi^{-1}$ for any Dirichlet character $\chi$ of $p$-power conductor.

   If $\chi$ is ramified at $p$, the ratio $L_p(f, g)(r - 1 + \chi) / L(\pi_f \otimes \pi_g \otimes \chi^{-1}, r-1)$ is a product of factorials, Gauss sums, powers of non-zero algebraic numbers, and rational functions in the quantities $\chi(q)$ for $q \mid N_f N_g$. For all but finitely many characters $\chi$ of $p$-power conductor, these factors are non-zero; hence $L_p(f, g)$ is non-vanishing at at least one character in each component of $\Spec \Zp[[\ZZ_p^\times]]$, so it is not a zero-divisor.

   When $r - r' = 2$, then the point $s = r-1$ lies on the abcissa of convergence, so the Euler product does not necessarily converge there; however, we can deduce the non-vanishing of $L(\pi_f \otimes \pi_g \otimes \chi^{-1}, r-1)$ from a general non-vanishing theorem due to Shahidi \cite[Theorem 5.2]{Shahidi-certain}, and the argument proceeds as before. Compare \cite[Theorem 4.4.1]{LLZ14}.
  \end{proof}


 \subsection{Galois representations}

  Let $f$ be a normalised cuspidal Hecke eigenform of some weight $k+2 \ge 2$ and level $N_f$, and let $L$ be a number field containing the $q$-expansion coefficients of $f$. Note that we do not necessarily require that $f$ be a newform.

  \begin{definition}
   For each prime $\frP \mid p$ of $L$, we write $M_{\Lp}(f)$ for the maximal subspace of
   \[ H^1_{\et, c}\left(Y_1(N_f)_{\overline{\QQ}}, \Sym^k \sH_{\Qp}^\vee\right) \otimes_\Qp \Lp \]
   on which the Hecke operators $T_\ell$, for primes $\ell \nmid N_f$, and $U_\ell$, for primes $\ell \mid N_f$, act as multiplication by $a_\ell(f)$.
  \end{definition}

  This is a 2-dimensional $\Lp$-vector space with a continuous action of $\Gal(\overline{\QQ} / \QQ)$, unramified outside $S \cup \{\infty\}$, where $S$ is the finite set of primes dividing $p N_f$. (Equivalently, $M_{\Lp}(f)$ is an \'etale $\Qp$-sheaf on $\Spec \ZZ[1/S]$.)

  Dually, we write $M_{\Lp}(f)^*$ for the maximal \emph{quotient} of the non-compact\-ly supported cohomology $H^1_{\et}\left(Y_1(N_f)_{\overline{\QQ}}, \TSym^k(\sH_{\Qp})(1) \right) \otimes_\Qp \Lp$ on which the dual Hecke operators $T_\ell'$ and $U_\ell'$ act as $a_\ell(f)$. We write $\pr_{f}$ for the projection onto this quotient. The twist by 1 implies that the Poincar\'e duality pairing
  \[ M_{\Lp}(f) \times M_{\Lp}(f)^* \to \Lp \]
  is well-defined (and perfect), justifying the notation. If $f$ is a newform, then its conjugate $f^*$ is also a newform, and the natural map $M_{\Lp}(f^*)(1) \to M_{\Lp}(f)^*$ is an isomorphism of $\Lp$-vector spaces, although we shall rarely use this.

  We write $\cO_{\frP}$ for the ring of integers of $\Lp$, and we write $M_{\cO_{\frP}}(f)^*$ for the $\cO_{\frP}$-lattice in $M_{\Lp}(f)^*$ generated by the image of the integral \'etale cohomology $H^1_{\et}\left(Y_1(N_f)_{\overline{\QQ}}, \TSym^k(\sH_{\Zp})(1)\right) \otimes_\Zp \cO_{\frP}$.

  \begin{remark}
   For $f$ of weight 2, the representations $M_{\Lp}(f)$, $M_{\Lp}(f)^*$ and $M_{\cO_{\frP}}(f)^*$ appear in \cite[\S 6.3]{LLZ14} under the names $V_{\Lp}(f)$, $V_{\Lp}(f)^*$ and $T_{\cO_{\frP}}(f)^*$. We have adopted different notations here to emphasise that these coincide with the $\Lp$-realisations of the Grothendieck motive $M(f)$ over $L$ attached to $f$ by Scholl \cite{Scholl-motives}.
  \end{remark}

  We define similarly a 2-dimensional $L$-vector space of de Rham cohomology
  \[ M_{\dR, L}(f) \subseteq H^1_{\dR, c}\left(Y_1(N_f)_{\QQ}, \Sym^k \sH_{\dR}^\vee\right) \otimes_\QQ L\]
  and its dual $M_{\dR, L}(f)^*$. Writing $M_{\dR, \Lp}(f)$ for the base-extension to $\Lp$, the comparison isomorphism \eqref{eq:comparisoniso} restricts to an isomorphism
  \[  M_{\dR, \Lp}(f) \cong \DD_{\dR}(M_{\Lp}(f) ) \]
  and similarly for the dual.

  If $f$, $g$ are two eigenforms (of some levels $N_f, N_g$ and weights $k+2, k' + 2 \ge 2$) with coefficients in $L$, we write $M_{\Lp}(f \otimes g)$ for the tensor product $M_{\Lp}(f) \otimes_{\Lp} M_{\Lp}(g)$, and similarly for the dual $M_{\Lp}(f \otimes g)^*$. We define similarly de Rham cohomology groups $M_{\dR, L}(f \otimes g)$ etc.

  Via the K\"unneth formula, we may regard $M_{\Lp}(f \otimes g)^*$ as a quotient of the \'etale cohomology of $Y_1(N_f) \times Y_1(N_g)$. Moreover, if $N$ is any common multiple of $N_f$ and $N_g$, there is a natural degeneracy morphism $Y_1(N)^2 \to Y_1(N_f) \times Y_1(N_g)$. Pushforward along this degeneracy morphism defines a projection map
  \begin{equation}
   \label{eq:prfg}
   \pr_{f, g}: H^2_{\et}\left(Y_1(N)^2_{\overline\QQ}, \TSym^k \sH_{\Qp} \boxtimes \TSym^{k'} \sH_{\Qp} (2)\right) \otimes \Lp \to M_{\Lp}(f \otimes g)^*.
  \end{equation}


\section{Eisenstein and Rankin--Eisenstein classes}


 In this section we recall some of the results of \cite{KLZ1a} concerning the \'etale Eisenstein classes on $Y_1(N)$, and the \'etale  Rankin--Eisenstein classes on the product $Y_1(N) \times Y_1(N)$.


 \subsection{Eisenstein classes}
  \label{sect:motivicEis}

  Let $N \geq 4$ and $b \in \ZZ / N\ZZ$ be nonzero. For $k\geq 0$, denote by
  \[ \Eis^k_{\mot, b, N} \in H^1_{\mot}(Y_1(N), \TSym^k \sH_{\QQ}(1))\]
  the motivic Eisenstein class as defined in \cite[\S 6.4]{BeilinsonLevin}, with the normalisation from \cite[Theorem 4.1.1]{KLZ1a}); it satisfies the residue formula
  \[ \operatorname{res}_\infty(\Eis^k_{\mot, b, N})=-N^k\zeta(-1-k).\]
  We are mostly interested in the case $b = 1$, and we write this class simply as $\Eis^k_{\mot, N}$; however, we shall occasionally need to consider general values of $b$ in order to state and prove our norm-compatibility relations. 

  \begin{remark}
   See \cite{KLZ1a} for the definition of the motivic cohomology group $H^1_{\mot}(Y_1(N), \TSym^k \sH_{\QQ}(1))$. For $k = 0$, it is isomorphic to $\cO(Y_1(N))^\times \otimes \QQ$, and the Eisenstein class $\Eis^k_{\mot, b, N}$ is simply the Siegel unit $g_{0, b/N}$, in the notation of \cite{Kato-p-adic}.
  \end{remark}

  We define Eisenstein classes in \'etale cohomology as the images of the motivic Eisenstein classes under the regulator map, as in \cite[\S 4.2]{KLZ1a}; this gives an \emph{\'etale Eisenstein class}
  \[ \Eis^k_{\et, b, N} \in H^1_{\et}\left(Y_1(N)_{\ZZ[1/Np]}, \TSym^k \sH_{\Qp}(1)\right). \]


 \subsection{Cohomology of product varieties and the Clebsch--Gordan map}
  \label{sect:clebschgordan}

  Let $\cE \to S$ be an elliptic curve over a base $S$, and suppose it is a $T$-scheme for some other scheme $T$. Assume $p$ is invertible on $T$. We can then define a lisse \'etale $\Qp$-sheaf on $S \times_T S$ by
  \[ \TSym^{[k, k']} \sH_{\Qp} \coloneqq \pi_1^* \left(\TSym^k \sH_{\Qp}\right) \otimes_{\Qp} \pi_2^* \left(\TSym^{k'} \sH_{\Qp}\right),\]
  where $\pi_1$ and $\pi_2$ are the first and second projections $S \times_T S \to S$.

  We write $\Delta$ for the diagonal inclusion $S \into S \times_T S$. Then
  \[\Delta^*( \TSym^{[k, k']} \sH_{\Qp} ) = \TSym^k \sH_{\Qp} \otimes \TSym^{k'} \sH_{\Qp}\]
  as sheaves on $S$; thus, if $S$ is smooth of relative dimension $d$ over $T$, we have a pushforward map
  \[ \Delta_*: H^i_{\et}(S, \TSym^k \sH_{\Qp} \otimes \TSym^{k'} \sH_{\Qp}(j)) \to H^{i + 2d}_{\et}(S \times_T S, \TSym^{[k, k']}\sH_{\Qp}(j + d)).\]

  Now let $k,k',j$ be integers satisfying
  \begin{equation}
   \label{eq:inequalities}
    k \ge 0, \quad k' \ge 0, \quad 0 \le j \le \min(k, k').
  \end{equation}

  Then there is a map  of sheaves on $S$ (the \emph{Clebsch--Gordan map})
  \[
   CG^{[k, k', j]} :  \TSym^{k + k' - 2j}\sH_\Qp
   \to \TSym^{k }\sH_\Qp \otimes \TSym^{k'}\sH_\Qp(-j).
  \]
  defined as in \cite[\S 5.1]{KLZ1a}. Composing with the pushforward map one has
  \[ \Delta_* \circ CG^{[k, k', j]}: H^1_{\et}\left(Y_1(N)[1/p], \TSym^{k + k' -2j}\sH_{\Qp}(1)\right) \to H^3_{\et}\left(Y_1(N)^2, \TSym^{[k, k']}\sH_{\Qp}(2-j)\right).\]
  One can also carry out the same construction with coefficients in $\Zp$, or in $\ZZ / p^r \ZZ$.


 \subsection{Rankin--Eisenstein classes}\label{section:REclass}

  We now come to the case which interests us: we consider the scheme $S = Y_1(N)$ over $T = \Spec \ZZ[1/Np]$.

  \begin{definition}
   \label{def:Rankin-Eisenstein-class}
   For $k, k', j$ satisfying the inequalities \eqref{eq:inequalities}, we define the \emph{\'etale Rankin--Eisenstein class} by
   \[
    \Eis^{[k, k', j]}_{\et, b, N} \coloneqq
    (\Delta_* \circ CG^{[k, k', j]})\left(\Eis^{k + k' - 2j}_{\et, b, N}\right)\\
    \in H^3_{\et}\left(Y_1(N)^2, \TSym^{[k, k']}\sH_{\Qp}(2-j)\right).
   \]
   (As before, if $b = 1$ we shall write this class simply as $\Eis^{[k, k', j]}_{\et, N}$.)
  \end{definition}

   The Hochschild--Serre spectral sequence (and the vanishing of $H^3_{\et}$ for affine surfaces over an algebraically closed field) allows us to regard $\Eis^{[k, k', j]}_{\et, b, N}$ as an element of the group
  \[ H^1\left(\ZZ[1/Np], H^2_{\et}\left( Y_1(N)^2_{\overline{\QQ}}, \TSym^{[k, k']}\sH_{\Qp}(2-j)\right)\right).\]

  \begin{definition}
   For $f, g$ eigenforms of weights $(k + 2, k' + 2)$ and levels dividing $N$, we set
   \[ \Eis^{[f, g, j]}_{\et, b, N} = \pr_{f, g}\left( \Eis^{[k, k', j]}_{\et, b, N} \right) \in H^1\left(\ZZ[1/Np], M_{\Lp}(f \otimes g)^*(-j)\right) \]
   where $\pr_{f, g}$ is as in \eqref{eq:prfg} above. (If $k = k' = j = 0$ and $b = 1$ this agrees with the class denoted $\mathbf{z}^{(f, g, N)}_1$ in \cite[Definition 6.4.4]{LLZ14}.)
  \end{definition}

  We record for later use a key local property of these Galois cohomology classes. It is clear that they are unramified at all primes not dividing $Np$; but they also satisfy a more subtle condition at the prime $p$. Recall that for a de Rham representation $V$ of $G_{\Qp}$, the Bloch--Kato subspace $H^1_\mathrm{g}(\Qp, V) \subseteq H^1(\Qp, V)$ is the kernel of the map $H^1(\Qp, V) \to H^1(\Qp, V \otimes \mathbf{B}_{\dR})$, where $\mathbf{B}_{\dR}$ is Fontaine's period ring.

  \begin{proposition}
   \label{prop:BFeltisgeom}
   The localisation of the Rankin--Eisenstein class at $p$, considered as an element of the space
   \[ H^1\left(\Qp, H^2_{\et}\left( Y_1(N)^2_{\overline{\QQ}}, \TSym^{[k, k']}\sH_{\Qp}(2-j)\right)\right), \]
   lies in the Bloch--Kato subspace $H^1_{\mathrm{g}}$.
  \end{proposition}

  \begin{proof}
   This follows from a very general theorem of  Nekov\v{a}\'r and Nizio{\l} \cite[Theorem B]{nekovarniziol16}, which implies that the images of motivic cohomology classes for varieties over $p$-adic fields automatically land in $H^1_\mathrm{g}$.
  \end{proof}

  We assume for the rest of this section that $p \nmid N$, and that $f$ is a newform. Then a much more precise description of the localisation of the Rankin--Eisenstein class at $p$ is given by one of the main results of \cite{KLZ1a}, which we now recall. It follows from \cite[Proposition 5.4.1]{KLZ1a} that the localisation of $\Eis^{[f, g, j]}_{\et, b, N}$ at $p$ lies in the image of the Bloch--Kato exponential map (the subspace $H^1_{\mathrm{e}}$), so we can consider
  \[
   \log\left(\Eis^{[f, g, j]}_{\et, b, N}\right)
   \in \frac{M_{\dR, \Lp}(f\otimes g)^*(-j)}{\Fil^0},
  \]
  where we use the Faltings--Tsuji comparison isomorphism $\comp_{\dR}$ of Equation \eqref{eq:comparisoniso} to give an identification of filtered $\varphi$-modules
  \[ \DD_{\dR}\left(M_{\Lp}(f\otimes g)^*\right) \cong M_{\dR, \Lp}(f\otimes g)^*.\]

  As in \S 6.1 of \emph{op.cit.} we have canonical vectors
  \[ \omega_g \in \Fil^1 M_{\dR, \Lp}(g) \otimes_{\QQ} \QQ(\mu_N) \]
  and
  \[ \eta_f^\alpha \in M_{\dR, \Lp}(f) \otimes_{\QQ} \QQ(\mu_N), \]
  the latter depending on a choice of root $\alpha_f$ of the Hecke polynomial of $f$. By definition, $\omega_g$ is the class of the differential form defined by $g$, as in \S \ref{sect:modformnotation} above; and $\eta_f^\alpha$ is the unique class which lies in the $\alpha_f$-eigenspace for the Frobenius endomorphism and pairs to 1 with $\omega_{f^*}$, where $f^*$ is the eigenform conjugate to $f$. The tensor product \( \eta_f^{\alpha} \otimes \omega_g\) thus lies in $\Fil^1 M_{\dR, \Lp}(f \otimes g) \otimes_{\QQ} \QQ(\mu_N)$.

  \begin{theorem}
   \label{thm:syntomicreg}
   Suppose that $f$ is ordinary, and let $\alpha_f$ be the unit root of its Hecke polynomial. Suppose also that $\cE(f, g, 1 + j) \ne 0$. Then we have
   \[
    \left\langle \log\left(\Eis^{[f, g, j]}_{\et, b, N}\right), \eta_f^{\alpha} \otimes \omega_g\right\rangle
    =  (-1)^{k' - j + 1} (k')! \binom{k}{j} \frac{\cE(f) \cE^*(f)}{\cE(f, g, 1 + j)}
    L_p(f, g, 1 + j),
   \]
   where $L_p(f,g,s)$ is Hida's $p$-adic $L$-function and the factors $\cE(f)$, $\cE^*(f)$ and $\cE(f, g, 1 + j)$ are as defined in Theorem \ref{thm:hida}.
  \end{theorem}

  \begin{proof}
   In Theorem 6.5.9 of \cite{KLZ1a} we showed an identical formula with the \'etale Eisenstein class replaced by the Eisenstein class in syntomic cohomology. However, the compatibility between syntomic and \'etale cohomology via the Bloch--Kato exponential (Proposition 5.4.1 of \emph{op.cit.}) shows that this is equivalent to the formula above.
  \end{proof}


\section{Eisenstein--Iwasawa classes}
 \label{sect:EisIwasawa}

 In this section, we define certain cohomology classes (``Eisenstein--Iwasawa classes'')
 \[ \cEI_{b,N} \in  H^1_{\et}\left(Y_1(N), \Lambda(\sH_{\Zp}\langle t_N \rangle)(1)\right), \]
 which can be regarded as ``$p$-adic interpolations'' of the \'etale Eisenstein classes described in \S \ref{sect:motivicEis} above. Here $\Lambda(\sH_{\Zp}\langle t_N \rangle)$ is a sheaf of Iwasawa modules whose definition we recall below. These classes appeared (although not under this name) in an earlier paper of the first author \cite{Kings-Eisenstein}, and we recall below one of the main results of that paper, which asserts that the image of $\cEI_{b,N}$ under the $k$-th moment map, for any $k \ge 0$, coincides with Beilinson's weight $k$ \'etale Eisenstein class. We also prove two distribution relations describing how the classes $\cEI_{b,N}$ behave under pushforward maps, which will be used in the construction of the Euler system in the following sections.

 \subsection{Definition of Eisenstein--Iwasawa classes}

  In this subsection we review the definition and the
  properties of the Eisen\-stein--Iwasawa classes. The starting point of the construction is the
  following result, which is Proposition 1.3 of \cite{Kato-p-adic}:

  \begin{theorem}[Kato]
   \label{elliptic-units-thm}
   Let $\pi:\cE\to S$ be an elliptic curve and $c>1$ be an integer prime to $6$.
   Then there is a unique element ${}_c\theta_\cE\in\cO(\cE\setminus \cE[c])^\times$
   such that:
   \begin{enumerate}
    \item $\operatorname{Div}({}_c\theta_\cE)=c^2(0)-\cE[c]$,
    \item For each isogeny $\varphi:\cE \to \cE'$ with $\deg\varphi$ prime to $c$
    one has $\varphi_*({}_c\theta_\cE)={}_c\theta_{\cE'}$.
    \item The ${}_c\theta_\cE$ are compatible with base change.
    \item If $d$ is another integer coprime to $6$, then
    \[
    ({}_d\theta_\cE)^{c^2}[c]^*({}_d\theta_\cE)^{-1}=
    ({}_c\theta_\cE)^{d^2}[d]^*({}_c\theta_\cE)^{-1}.
    \]
   \end{enumerate}
  \end{theorem}

  Now fix a prime number $p$, and assume $p$ is invertible on $S$. For $r \ge 0$, let $\cE_r \coloneqq \cE$, considered as a covering of $\cE$ via $[p^r]:\cE_r\to \cE$, and consider the pro-system of \'etale lisse sheaves on $\cE$ given by
  \[
   \sL\coloneqq \Big([p^r]_*(\ZZ/p^r\ZZ)\Big)_{r\ge 1}
  \]
  where the transition maps
  $[p^{r+1}]_*(\ZZ/p^{r+1}\ZZ) \to [p^r]_*(\ZZ/p^{r}\ZZ)$
  are the composition of the trace map with the reduction modulo $p^r$.
  The Leray spectral sequence provides us with an isomorphism
  \[
   H^1_{\et}(\cE\setminus\cE[c],[p^r]_*(\ZZ/p^{r}\ZZ)(1))\cong
   H^1_{\et}(\cE_r\setminus\cE_r[p^rc],\ZZ/p^{r}\ZZ(1))
  \]
  and if we combine this with \eqref{torsors-as-limits} we get
  \[
   H^1_{\et}(\cE\setminus \cE[c],\sL(1)) \cong
   \varprojlim_r H^1_{\et}(\cE_r\setminus \cE_r[p^rc],\ZZ/p^r\ZZ(1)).
  \]

  Denote by
  \[
   \partial_r: \cO(\cE_r\setminus\cE_r[p^rc])^\times\to
   H^1_{\et}(\cE_r\setminus \cE[p^rc],\ZZ/p^r\ZZ(1))
  \]
  the Kummer map, i.e., the connecting homomorphism for the exact sequence
  \[ 0\to \mu_{p^r}\to \Gm\to \Gm\to 0.\]
  We assume henceforth that $p \nmid c$. Then Theorem \ref{elliptic-units-thm} implies that the elements
  $\partial_r({_c\theta_{\cE_r}})$ are compatible with the trace maps.

  \begin{definition}
   \label{Theta-defn}
   For $c > 1$ coprime to $6p$, let
   \[
    {}_c \Theta_\cE \coloneqq\varprojlim_r \partial_r({_c\theta_{\cE_r}})\in
    H^1_{\et}(\cE\setminus \cE[c],\sL(1)).
   \]
  \end{definition}

  The elements
  ${}_c \Theta_\cE$ inherit all of the important properties of ${}_c\theta_\cE$. To formulate them precisely, observe that
  for each isogeny $\varphi:\cE\to \cE'$ with degree coprime to $c$ one has
  a morphism
  \[
   \varphi_*:H^1_{\et}(\cE\setminus \cE[c],\sL(1))\to
   H^1_{\et}(\cE'\setminus \cE'[c],\sL'(1))
  \]
  defined to
  be the inverse limit of the natural trace maps
  \[
   \varphi_*:H^1_{\et}(\cE_r\setminus \cE_r[p^rc],\ZZ/p^r\ZZ(1))\to
   H^1_{\et}(\cE'_r\setminus \cE'_r[p^rc],\ZZ/p^r\ZZ(1)).
  \]

  \begin{proposition}
   \label{Theta-norm-compatibility}
   The elements ${}_c \Theta_\cE$ satisfy the following compatibilities:
   \begin{enumerate}
   \item Let $\varphi:\cE\to \cE'$ be an isogeny of degree coprime to $c$. Then
      \[
       \varphi_*({}_c \Theta_\cE)={}_c\Theta_{\cE'}.
      \]
      In particular, if $a$ is an integer coprime to $c$ one has
      $[a]_*({}_c\Theta_\cE)={}_c\Theta_\cE$.
   \item If $f: T \to S$ is a morphism, $\cE_T \coloneqq \cE\times_ST$, and
      $\tilde f:\cE_T \to \cE$ is the base change morphism, one has
      \[
      \tilde f^*({}_c\Theta_\cE)={}_c\Theta_{\cE_T}.
      \]
   \item In $ H^1_{\et}(\cE\setminus \cE[cd],\sL(1))$ one has the equality
      \[
      d^2{}_c\Theta_\cE-[d]^*({}_c\Theta_\cE)=
      c^2{}_d\Theta_\cE-[c]^*({}_d\Theta_\cE)
      \]
   for any integer $d$ coprime to $6p$.
   \end{enumerate}
  \end{proposition}

  \begin{proof}
   The compatibility with isogenies follows from the commutative diagram
   \[
    \begin{diagram}[labelstyle=\scriptstyle,small]
     \cO(\cE_r\setminus\cE_r[p^rc])^* &\rTo^{\ \ \partial_r\ \ }&
     H^1_{\et}(\cE_r\setminus \cE_r[p^rc],\ZZ/p^r\ZZ(1))\\
     \dTo<{\varphi_*} && \dTo>{\varphi_*}\\
     \cO(\cE'_r\setminus\cE'_r[p^rc])^* & \rTo^{\partial_r}&
     H^1_{\et}(\cE'_r\setminus \cE'_r[p^rc],\ZZ/p^r\ZZ(1))
    \end{diagram}
   \]
   and the isogeny-compatibility relation $\varphi_*({}_c\theta_\cE)={}_c\theta_{\cE'}$. The
   compatibility with base change follows from
   \[
   \begin{diagram}[labelstyle=\scriptstyle,small]
    \cO(\cE_r\setminus\cE_r[p^rc])^* &\rTo^{\ \ \partial_r\ \ }&
    H^1_{\et}(\cE_r\setminus \cE_r[p^rc],\ZZ/p^r\ZZ(1))\\
    \dTo<{\tilde f^*} & & \dTo>{\tilde f^*} \\
    \cO(\cE_{T,r}\setminus\cE_{T,r}[p^rc])^* & \rTo^{\partial_r}&
    H^1_{\et}(\cE_{T,r}\setminus \cE_{T,r}[p^rc],\ZZ/p^r\ZZ(1))
   \end{diagram}
  \]
  and Theorem \ref{elliptic-units-thm}. The final statement is immediate from the corresponding compatibility of ${}_c \theta_\cE$ and ${}_d \theta_\cE$.
  \end{proof}

  \begin{definition}
   Let $\iota_D:D\into \cE$ be a subscheme finite \'etale over $S$ and
   write $p_D\coloneqq\pi\circ\iota_D:D\to S$. Define
   $\cE[p^r]\langle D\rangle$ by the Cartesian diagram
   \[
    \begin{diagram}[small]
     \cE[p^r]\langle D\rangle & \rInto & \cE\\
     \dTo<{p_{r,D}} && \dTo>{[p^r]}\\
     D & \rInto^{\iota_D} &\cE
    \end{diagram}
   \]
   and let
   \[
    \Lambda_r(\sH_r\langle D\rangle)\coloneqq p_{D*}\iota_D^*[p^r]_*\ZZ/p^r\ZZ
    \cong p_{D*}p_{r,D,*}\ZZ/p^r\ZZ.
   \]
  \end{definition}

  Then $p_{D*}\iota_D^*\sL$ is the sheaf defined  by the
  pro-system $(\Lambda_r(\sH_r\langle D\rangle))_{r\ge 1}$. We denote it by $\Lambda(\sH_{\Zp}\langle D\rangle)$. In the special case where $D=S$ and $\iota=t:S\to \cE$
  is a section, we write
  \[
   \Lambda_r(\sH_r\langle t\rangle)=t^*[p^r]_*\ZZ/p^r\ZZ\mbox{ and }
   \Lambda(\sH_{\Zp}\langle t\rangle)=t^*\sL,
  \]
  which are the sheaves defined and studied in \cite{Kings-Eisenstein}.
  These sheaves
  can and should be viewed as sheaves of modules under the sheaf of
  Iwasawa algebras $\Lambda(\sH_{\Zp})\coloneqq\Lambda(\sH_{\Zp}\langle 0\rangle)$. For more
  details on this we refer again to \cite{Kings-Eisenstein}.

  In the special case
  where $D$ splits over $S$ into a disjoint union of copies
  of $S$, we get
  \begin{equation}\label{Iwasawa-direct-sum}
   \Lambda_r(\sH_r\langle D\rangle)\cong \bigoplus_{t\in D(S)}
   \Lambda_r(\sH_r\langle t\rangle).
  \end{equation}
  For any isogeny $\varphi:\cE\to \cE'$ and subschemes $\iota_D:D\to \cE$, $\iota_{D'}:D'\to \cE'$ with $\varphi(D)\subset D'$ the trace map with respect to $\varphi$ induces a map
  \[
  \varphi_*:\Lambda_r(\sH_r\langle D\rangle)\to \Lambda_r(\sH_r'\langle D'\rangle).
  \]
  In the special case where $D$ is split and $D'$ is a section, the map
  \[
  \varphi_*:\bigoplus_{t\in D(S),\varphi(t)=t'}
       \Lambda_r(\sH_r\langle t\rangle)\to \Lambda_r(\sH_r'\langle t'\rangle)
  \]
  is just the sum of the trace maps $\varphi_*:\Lambda_r(\sH_r\langle t\rangle)\to
  \Lambda_r(\sH_r'\langle t'\rangle)$.

  To define the \emph{Eisenstein--Iwasawa} classes note that one has an isomorphism
  \[
   H^1_{\et}(D,\iota_D^*\sL(1))\cong H^1_{\et}(S,\Lambda(\sH_{\Zp}\langle D\rangle)(1)).
  \]

  \begin{definition}
   Let $\cE$ be an elliptic curve, $\iota_D:D\to \cE\setminus\cE[c]$ be
   a subscheme finite \'etale over $S$ and $p_D=\pi\circ\iota_D:D\to S$ the structure map.
   The \emph{Eisen\-stein--Iwasawa classes} are the classes
   \[
    \cEI_D\coloneqq\iota_D^*({}_c \Theta_\cE)\in H^1_{\et}(S,\Lambda(\sH_{\Zp}\langle D\rangle)(1)).
   \]
   In the case where $D$ corresponds to a section $t$, we simply write $\cEI_t$.
  \end{definition}


 \subsection{Properties of the Eisenstein--Iwasawa classes}

  The Eisenstein--Iwasawa classes share the properties of ${}_c \Theta_\cE$. In particular,
  they behave well under base-change and norm maps.

  \begin{proposition}
   \label{prop:eisenstein-iwasawa-properties}
   Let $\pi:\cE\to S$ be an elliptic curve, and $\iota_D:D \to \cE \setminus \cE[c]$ be a subscheme finite \'etale over $S$.
   \begin{enumerate}
    \item Let $f: T \to S$ be a morphism, and define $\cE' = \cE \times_S T$ and similarly $D'$. Then
    \[
     f^*\left(\cEI_D\right)=\cEI_{D'}.
    \]
    \item Let $\varphi:\cE\to \cE'$ be an isogeny of degree coprime to $c$, $D'\subset \cE'$ finite \'etale over $S$, and $D=\varphi^{-1}(D')$. Then
    \[
     \varphi_*\left(\cEI_D\right)=\cEI_{D'}.
    \]
    In particular, if $D'$ corresponds to a section $t'$ and $D=\varphi^{-1}(t')$ splits over $S$ into a disjoint union of copies of $S$, one has
    \[
    \cEI_{t'}=\sum_{t\in \cE(S),\varphi(t)=t'}\varphi_*(\cEI_{t}).
    \]
    \item If $c, d > 1$ are both coprime to $6p$ and $D \subset \cE \setminus \cE[cd]$, then the class
    \[
     d^2 \cEI_{D} - ([d]_*)^{-1} \cEI_{[d] D}
     \in H^1_{\et}\left(S, \Lambda(\sH_{\Zp}\langle D \rangle)(1)\right)
    \]
    is symmetric in $c$ and $d$.
   \end{enumerate}
  \end{proposition}

  \begin{proof} \
   \begin{enumerate}
    \item As before, let $\tilde f:\cE'\coloneqq\cE\times_ST\to \cE$ be the base change map. Then by Proposition \ref{Theta-norm-compatibility} we have $\tilde f^*({}_c\Theta_\cE)={}_c\Theta_{\cE'}$ and hence $\cEI_{D'}=(\iota_{D'}^* \circ \tilde f^*)({}_c\Theta_\cE)=(f^* \circ \iota_D^*)({}_c\Theta_{\cE})=f^*\cEI_D$.
    \item As
    \[
     \begin{diagram}[labelstyle=\scriptstyle,small]
      D & \rTo^{\iota_D} & \cE\\
      \dTo<\varphi & & \dTo>\varphi\\
      D' & \rTo^{\iota_{D'}} & \cE'
     \end{diagram}
    \]
    is Cartesian, we have $\varphi_*\cEI_D=\varphi_*\iota_D^* {}_c\Theta_{\cE}=\iota_{D'}^*\varphi_*{}_c\Theta_{\cE}=\iota_{D'}^*{}_c\Theta_{\cE'} =\cEI_{D'}$ by Proposition \ref{Theta-norm-compatibility}.
    \item The multiplication map $[d]$ gives an isomorphism $\cE[p^r]\langle D\rangle\cong \cE[p^r]\langle [d] \circ D\rangle$. This induces an isomorphism $[d]_*:\Lambda_r(\sH_r\langle D\rangle)\cong \Lambda_r(\sH_r\langle [d]\circ D\rangle)$ with inverse $[d]^*$. The commutative diagram
    \[
     \begin{diagram}[labelstyle=\scriptstyle,small]
      H^1_{\et}(\cE\setminus \cE[c],\sL(1)) & \rTo^{[d]^*} & H^1_{\et}(\cE\setminus \cE[cd],\sL(1))\\
      \dTo<{(\iota_{[d]D})^*} & & \dTo>{\iota_{D}^*}\\
      H^1_{\et}\left(S,\Lambda(\sH_{\Zp}\langle [d]D\rangle)(1)\right) & \rTo^{[d]_*}&
      H^1_{\et}\left(S,\Lambda(\sH_{\Zp}\langle D\rangle)(1)\right)
     \end{diagram}
    \]
    shows that $\iota_D^*[d]^*({}_c\Theta_\cE)=[d]^*(\iota_{[d] D})^*({}_c\Theta_\cE)$. Hence, from property \ref{Theta-norm-compatibility}(3), the expression
    \[
     \iota_D^*\left( d^2 {}_c \Theta_{\cE} - [d]^*  {}_c \Theta_{\cE}\right)
     = d^2 \cEI_{D} - [d]^* \cEI_{[d] D}
    \]
    is symmetric in $c$ and $d$ as required.\qedhere
   \end{enumerate}
  \end{proof}

  We now consider a particular special case, which will be used below to prove the Euler system norm relations. Let $\cE / S$ be an elliptic curve, $c > 1$ coprime to $6p$, $t: S \into \cE \setminus \cE[c]$ an order $N$ section, and $\cE'$ a second elliptic curve over $S$ equipped with an isogeny $\lambda: \cE' \to \cE$ whose degree is invertible on $S$ and coprime to $c$. Then the subscheme $\lambda^{-1} t \subset \cE' \setminus \cE'[c]$ is finite \'etale over $S$.

  We define $S'$ to be the fibre product of $t: S \to \cE$ with the isogeny $\lambda$; so $S'$ is a variety equipped with a finite \'etale covering map $\pi: S' \to S$ and a closed embedding $t' : S' \to \cE'$ such that $\lambda \circ t'  = t \circ \pi$, and $(S', \pi, t')$ is universal among such data. We interpret $t'$ as a section of $\cE' \times_{S} S'$ in the natural way; then for each $r \ge 1$ we have the equality
  \[ (\cE' \times_S S')[p^r]\langle t' \rangle = \cE'[p^r]\langle \lambda^{-1} t \rangle.\]

  Hence we have an equality of pro-\'etale sheaves on $S$
  \[
   \pi_*\left(\Lambda(\sH'_{\Zp}\langle t' \rangle)\right)
   = \Lambda(\sH'_{\Zp}\langle \lambda^{-1} t \rangle),
  \]
  and it is clear that $\pi_*\left( \cEI_{t'}\right) = \cEI_{\lambda^{-1} t}$. The isogeny $\lambda$ gives a map $\lambda_*: \Lambda(\sH'_{\Zp}\langle \lambda^{-1} t \rangle) \to \Lambda(\sH_{\Zp}\langle t\rangle)$, and by part (2) of the preceding proposition, we have $\lambda_* \left(\cEI_{\lambda^{-1} t}\right) = \cEI_{t}$. Combining these two statements, we see that
  \begin{equation}
   \label{eq:fibreprodcompat}
   \lambda_* \pi_* \left(\cEI_{t'}\right) = \cEI_{t}.
  \end{equation}


 \subsection{Modular curves and pushforward relations}

  We are particularly interested in the case where $S$ is the modular curve $Y_1(N)$ for some $N \ge 4$ (viewed as a scheme over $\ZZ[1/Np]$), $\cE$ is the universal elliptic curve over $Y_1(N)$, and $t = t_N$ the canonical order $N$ section. For $c > 1$ coprime to $6Np$ and $b \in \ZZ / N\ZZ \setminus \{0\}$ we write
  \[
   \cEI_{b, N} \coloneqq \cEI_{b t_N}
   \in H^1_{\et}\left(Y_1(N), \Lambda(\sH_{\Zp}\langle b t_N \rangle)(1)\right).
  \]
  As with the motivic classes of the previous section, we shall abbreviate $\cEI_{1, N}$ simply as $\cEI_N$.

  \begin{remark}
   Note that since $N$ is invertible on $Y_1(N)$ and $(c, N) = 1$, the image of $b t_N$ is automatically contained in $\cE \setminus \cE[c]$.
  \end{remark}

  More generally, for any $M, N$ with $M + N \ge 5$, we may define classes $\cEI_{b, N}$ on $Y(M, N)$ in the same way. If $N \ge 4$ then the class $\cEI_{b, N}$ is the pullback of the corresponding class on $Y_1(N)$, but the latter does not exist for $N \le 3$ as the moduli problem corresponding to $Y_1(N)$ is not representable.

  We now study the compatibility of the Eisenstein--Iwasawa classes under pushforward maps between modular curves. For the remainder of this subsection, $M, N$ will be integers $\ge 1$ with $M + N \ge 5$ and $M \mid N$, and $\ell$ will be any prime. Note that we allow $\ell = p$. We have a natural degeneracy map $\pr_1: Y(M, \ell N) \to Y(M, N)$, and $(\pr_1)^*(t_N) = \ell \cdot t_{N\ell}$, so the $\ell$-multiplication gives a map
  \[ [\ell]_*: (\pr_1)_* \left(\Lambda(\sH_{\Zp}\langle b t_{N\ell} \rangle)\right) \to \Lambda(\sH_{\Zp}\langle b t_{N} \rangle)\]
  of sheaves on $Y(M, N)$.

  \begin{definition}
   \label{def:lambdapr1}
   We consider $(\pr_1, [\ell]_*)$ as a pushforward map
   \[ \left( Y(M, \ell N), \Lambda(\sH_{\Zp}\langle b t_{N\ell}\rangle)\right) \to \left( Y(M, N), \Lambda(\sH_{\Zp}\langle b t_{N}\rangle) \right) \]
   in the sense of \S\ref{sect:adjunction} (and we denote this map simply by $\pr_1$).
  \end{definition}

  \begin{theorem}
   \label{thm:pushforward}
   Let $M, N \ge 1$ with $M \mid N$ and $M + N \ge 5$, and let $\ell$ be a prime. Then for any $b \in (\ZZ / N\ell\ZZ)^\times$, the map $(\pr_1)_*$ sends $\cEI_{b, N\ell}$ to
   \[
    \begin{cases}
     \cEI_{b, N} & \text{if $\ell \mid N$,} \\
     \cEI_{b, N} - [\ell]_* \cEI_{\ell^{-1} b, N} & \text{if $\ell \nmid N$,}
    \end{cases}
   \]
   where in the latter case ``$\ell^{-1}$'' signifies the inverse of $\ell$ modulo $N$.
  \end{theorem}

  \begin{proof}
   We will deduce the theorem from the isogeny-compatibility formula of equation \eqref{eq:fibreprodcompat} applied with $\lambda$ equal to the multiplication-by-$\ell$ isogeny $[\ell] : \cE \to \cE$, where $\cE$ is the universal elliptic curve over $S = Y(M, N)$.

   If $\ell \mid N$, then the triple $(Y(M, N\ell), \pr_1, t_{N\ell})$ evidently satisfies the same universal property as the covering $(S', \pi, t')$ defined in the previous section (since any $[\ell]$-preimage of a point of exact order $N$ has exact order $N\ell$). Thus Equation \eqref{eq:fibreprodcompat} in this case is exactly the statement that $(\pr_1)_* \left(\cEI_{b, N\ell}\right) = \cEI_{b, N}$.

   In the case $\ell \nmid N$, we must be slightly more careful, since $S'$ classifies arbitrary preimages of $t_N$, while $Y(M, N\ell)$ classifies only those having exact order $N\ell$. Hence we have $S' = Y(M,N) \sqcup Y(M,N\ell)$, with the restriction of $t'$ to $Y(M, N)$ being the order $N$ section $\ell^{-1} t_N$. Thus Equation \eqref{eq:fibreprodcompat} becomes
   \[
    (\pr_1)_*\left( \cEI_{b, N\ell}\right) + [\ell]_*\left( \cEI_{\ell^{-1} b, N}\right) = \cEI_{b, N}
   \]
   as required.
  \end{proof}

  We also give a second pushforward relation refining the above. As in \S\ref{sect:degeneracy} above, we factor $\pr_1$ as the composite of the natural degeneracy maps
  \[ Y(M, N\ell) \rTo^{\pr'} Y(M, N(\ell)) \rTo^{\pr} Y(M, N). \]

  The image of the section $t_{N\ell}$ under $\pr'$ has the following description. Recall that we have an isomorphism $\varphi_{\ell}: Y(M, N(\ell)) \to Y(M(\ell), N)$, and the cyclic $\ell$-isogeny $\lambda: \cE \to \cE'$ of elliptic curves over $Y(M, N(\ell))$, where $\cE' = \varphi_\ell^*(\cE)$; then it follows easily from the definitions that we have
  \[ \pr'(t_{N\ell}) \subseteq \lambda^{-1} t'_N, \]
  where $t'_N = \varphi_\ell^*(t_N)$ is the standard order $N$ section of $\cE'$. On the other hand, the dual isogeny $\hat\lambda: \cE' \to \cE$ maps $t_N'$ to $t_N$, which is naturally the pullback of a section over $Y(M, N)$ via $\pr$. Thus we have pushforward maps
  \[
   \begin{array}{ccccc}
   \pr' \coloneqq (\pr', \lambda_*) &:& \Big(Y(M, N\ell), \Lambda(\sH_{\Zp}\langle b t_{N\ell}\rangle)\Big) &\to& \Big(Y(M, N(\ell)), \varphi_\ell^* \Lambda(\sH_{\Zp}\langle b t_N\rangle)\Big)\\
   \pr \coloneqq (\pr, \hat\lambda_*)&: &\Big(Y(M, N(\ell)), \varphi_\ell^* \Lambda(\sH_{\Zp}\langle b t_N\rangle)\Big) &\to& \Big(Y(M, N), \Lambda(\sH_{\Zp}\langle b t_{N}\rangle)\Big)
  \end{array}\]
  whose composite is $\pr_1$.

  \begin{theorem}
   \label{thm:pushforward2} \
   \begin{enumerate}
    \item Let $b \in (\ZZ / N\ell\ZZ)^\times$. As elements of $H^1(Y(M, N(\ell)), \varphi_\ell^* \Lambda(\sH_{\Zp}\langle b t_N \rangle))$, we have
    \begin{equation*}
     (\pr')_*(\cEI_{b, N\ell}) =
     \begin{cases}
      \varphi_\ell^*(\cEI_{b, N}) & \text{if $\ell \mid N$},\\
      \varphi_\ell^*(\cEI_{b, N}) - \lambda_*(\cEI_{\ell^{-1}b, N}) & \text{if $\ell \nmid N$.}
     \end{cases}
    \end{equation*}
    \item Let $b \in (\ZZ / N\ZZ)^\times$. Then we have
    \begin{equation*}
     \pr_* \left(\varphi_\ell^* \cEI_{b, N}\right) =
     \begin{cases}
      \cEI_{b, N} & \text{if $\ell \mid N$}\\
      \cEI_{b, N} + \ell [\ell]_* \cEI_{\ell^{-1}b, N} & \text{if $\ell \nmid N$}
     \end{cases}
    \end{equation*}
   \end{enumerate}
  \end{theorem}

  \begin{proof}
   Let us first prove (1). We shall deduce this from Equation \eqref{eq:fibreprodcompat} applied to the isogeny $\lambda: \cE \to \cE'$ over $S = Y(M, N(\ell))$. If $\ell \mid N$, then the covering of $Y(M, N(\ell))$ classifying points of $\cE$ such that $\lambda(s) = b t_N'$ is exactly $Y(M, N\ell)$ with the canonical section $bt_{N\ell}$, so Equation \eqref{eq:fibreprodcompat} tells us that
   \[ (\pr')_* \left(\cEI_{b, N\ell}\right) = \varphi_\ell^*\left(\cEI_{b, N}\right) \]
   as claimed.

   If $\ell \nmid N$, then this fibre product is slightly larger than $Y(M, N\ell)$, since not all preimages of $bt_N$ under $\lambda$ have exact order $\ell N$. Exactly as in the proof of Theorem \ref{thm:pushforward}, we find that the required fibre product is the disjoint union of $Y(M, N\ell)$ and a copy of $Y(M, N(\ell))$ with the section $\ell^{-1} b t_N$, and the same argument as before gives
   \[ (\pr')_* \left(\cEI_{b, N\ell}\right) + \lambda_* \left(\cEI_{\ell^{-1}b, N}\right)  = \varphi_\ell^*\left(\cEI_{b, N}\right). \]

   We now deduce part (2) by comparing the above with Theorem \ref{thm:pushforward} (after choosing an arbitrary lifting of $b$ to $(\ZZ / \ell N \ZZ)^\times$). This gives the result immediately in the case $\ell \mid N$. For $\ell \nmid N$, we note that the image of $\lambda_* \cEI_{\ell^{-1}b, N}$ on under $\hat\lambda_*$ is just $[\ell] \cEI_{\ell^{-1} b, N}$, which is the pullback via $\pr$ of its namesake on $Y(M, N)$; so applying $\pr_*$ to it simply multiplies it by the degree of the map $\pr$, which is $\ell + 1$. Thus
   \begin{align*}
    \pr_*\left( \varphi_\ell^* \cEI_{b, N}\right) &= (\pr_1)_* \left(\cEI_{b, \ell N}\right) + \pr_* \left(\lambda_* \cEI_{\ell^{-1}b, N}\right)\\
    &= \left(\cEI_{b,N} - [\ell]_* \cEI_{\ell^{-1}b, N}\right) + (\ell + 1)[\ell]_*  \cEI_{\ell^{-1}b, N}\\
    &= \cEI_{b, N} + \ell [\ell]_* \cEI_{\ell^{-1}b, N}.\qedhere
   \end{align*}
  \end{proof}

  \begin{remark}
   Note that $\pr_* \varphi_\ell^*(\cEI_{b, N})$ is the image of $\cEI_{b, N}$ under the Hecke operator $T_\ell'$ (if $\ell \nmid N$) or $U_\ell'$ (if $\ell \mid N$), so we can interpret Theorem \ref{thm:pushforward2}(2) as the statement that for $\ell \mid N$ we have $U_\ell' \left(\cEI_{b,N}\right) = \cEI_{b, N}$, and for $\ell \nmid N$ we have $T_\ell'\left(\cEI_{b, N}\right) = \cEI_{b, N} + \ell [\ell]_* \cEI_{\ell^{-1}b, N}$.
  \end{remark}

  \begin{corollary}
   \label{cor:pr2}
   If we consider $\pr_2$ as a pushforward map $\left(Y(M, N\ell), \sH_{\Zp}\langle t_{N\ell}\rangle\right) \to \left(Y(M, N), \sH_{\Zp}\langle t_N \rangle\right)$ using the $\ell$-isogeny $\lambda: \cE \to \pr_2^* \cE$ (which maps $t_{N\ell}$ to $t_N$), then we have
   \[ (\pr_2)_* \left( \cEI_{b, N \ell}\right) =
    \begin{cases}
     \ell \cEI_{b, N} & \text{if $\ell \mid N$,}\\
     \ell \cEI_{b, N} - \ell [\ell]_* \cEI_{\ell^{-1} b, N} & \text{if $\ell \nmid N$}.
    \end{cases}
   \]
  \end{corollary}

  \begin{proof}
   This follows easily from part (1) of the previous theorem.
  \end{proof}

 \subsection{Moment maps and the relation to Eisenstein classes}

  The sheaves of algebras $\Lambda(\sH_{\Zp})$ are sheafifications of Iwasawa algebras, and can be handled in much the same way. In particular, one has moment maps, corresponding to the natural maps of sheaves of sets $\sH_r \to \TSym^k\sH_r$, $x \mapsto x^{[k]}$:

  \begin{proposition}[\cite{Kings-Eisenstein} 2.5.2, 2.5.3]
   Let $\Lambda_r(\sH_r)\coloneqq\Lambda_r(\sH_r\langle 0\rangle)$. Then there are moment maps for all $r\ge 1$,
   \[ \mom^k_r:\Lambda_r(\sH_r)\to \TSym^k\sH_r, \]
   which assemble into a morphism of pro-sheaves
   \[ \mom^k:\Lambda(\sH_{\Zp}) \to \TSym^k\sH_{\Zp}.\]
  \end{proposition}

  We recall some functoriality properties of the maps $\mom^k_{r}$.

  \begin{lemma}
   \label{mom-functoriality} Let $\pi:\cE\to S$ be an elliptic curve.
   \begin{enumerate}
    \item\label{mom-functoriality-basechange} (Base-change compatibility) Let $f:T\to S$
    be a map and $\cE'$
    be the pullback of $\cE$. Then for each $r \ge 1$
    there is a commutative diagram of sheaves on $T$
    \[
     \begin{diagram}[labelstyle=\scriptstyle,small]
      f^*\left(\Lambda_r(\sH_r)\right)
      &\rTo^{f^*(\mom_r^k)} & f^*(\TSym^k \sH_r)\\
     \dTo<\cong & & \dTo>\cong\\
      \Lambda_r(\sH_r') & \rTo^{\mom_r^k} & \TSym^k \sH_r'.
     \end{diagram}
    \]
    \item\label{mom-functoriality-pushfwd} (Pushforward via isogenies) Let $\varphi:\cE\to \cE'$ be an isogeny.
    Denote by  $\varphi_*:\sH_r\to\sH_r'$ the corresponding trace map. Then for each $r \ge 1$
    there is a commutative diagram of sheaves
    \[
     \begin{diagram}[labelstyle=\scriptstyle,small]
      \Lambda_r(\sH_r) & \rTo^{\mom_{r}^k} & \TSym^k \sH_r\\
      \dTo<{\varphi_*} && \dTo>{\TSym^k\varphi_*}\\
      \Lambda_r(\sH_r') & \rTo^{\mom_{r}^k} &\TSym^k \sH_r'.
     \end{diagram}
    \]
    In the case where $\varphi=[A]$ is the $A$-multiplication, the
    map $\TSym^k[A]_*$ is multiplication by $A^k$.
   \end{enumerate}
  \end{lemma}

  \begin{proof}
   These compatibilities are clear from the construction of the moment map, cf.\ \cite[Prop.\ 2.2.2]{Kings-Eisenstein}.
  \end{proof}

  \begin{notation}
   If $t$ is an $N$-torsion section, we denote the composite
   \[ \Lambda(\sH_{\Zp}\langle t \rangle) \rTo^{[N]_*} \Lambda(\sH_{\Zp}) \rTo^{\mom^k} \TSym^k \sH_{\Zp} \]
   by $\mom^k_{\langle t \rangle, N}$. (We will omit the subscripts if $N$ and $t$ are clear from context.)
  \end{notation}

  The following theorem, which is a slight restatement of one of the main results of \cite{Kings-Eisenstein}, is fundamental for the entire paper:

  \begin{theorem}
   \label{thm:moments-motivic}
   As elements of $H^1\left(Y_1(N), \TSym^k \sH_{\Qp}(1)\right)$, we have
   \[ \mom^k_{\langle b t_N\rangle, N} (\cEI_{b, N}) = c^2 \Eis^k_{\et, b, N} - c^{-k} \Eis^k_{\et, c b, N},\]
   where the classes on the right-hand side are the \'etale Eisenstein classes of \S \ref{sect:motivicEis} above.
  \end{theorem}

  \begin{proof}
   See \cite[Theorem 4.7.1]{Kings-Eisenstein}.
  \end{proof}

  \begin{remark}
   Note that the statement in \emph{op.cit.} includes a factor of $-N$ that does not appear here, which is the motivation for our slightly different normalisation for the Eisenstein class in the present paper compared to \cite{Kings-Eisenstein}.
  \end{remark}

  \begin{remark}
   Note that the moment maps commute with $(\pr_1)_*$: more precisely, we have a commutative diagram
   \begin{diagram}[labelstyle=\scriptstyle,small]
    H^1\left(Y_1(N\ell), \Lambda(\sH_{\Zp}\langle b t_{N\ell}\rangle)(1)\right) & \rTo^{\mom^k_{\langle b t_{N\ell}\rangle, N\ell}} &  H^1\left(Y_1(N\ell), \TSym^k \sH_{\Zp}(1)\right)\\
    \dTo<{(\pr_1)_*} & & \dTo>{(\pr_1)_*} \\
    H^1\left(Y_1(N), \Lambda(\sH_{\Zp}\langle b t_{N} \rangle)(1)\right) & \rTo^{\mom^k_{\langle b t_{N}\rangle, N}} &  H^1\left(Y_1(N), \TSym^k \sH_{\Zp}(1)\right).
   \end{diagram}
   Thus one can immediately deduce a pushforward compatibility for the \'etale Eisenstein classes from Theorems \ref{thm:pushforward} and \ref{thm:moments-motivic}. However, the analogous diagram for $\pr_2$ does not commute; instead, we have
   \[ (\pr_2)_* \circ \mom^k_{\langle b t_{N\ell}\rangle, N\ell} = \ell^k \mom^k_{\langle b t_{N}\rangle, N} \circ (\pr_2)_*, \]
   so the analogue of Corollary \ref{cor:pr2} for the \'etale Eisenstein classes includes an additional factor of $\ell^k$.

   These pushforward relations for the \'etale Eisenstein classes can also be obtained as a consequence of corresponding statements for the motivic Eisenstein classes $\Eis^k_{\mot, b, N}$, although we shall not use this here. (The case $k = 0$ of this motivic compatibility is \cite[Theorem 2.2.4]{LLZ14}. The general case has been treated by Scholl \cite[\S A.2]{Scholl-Kato-Euler-system}, although Scholl's normalisations are a little different from ours.)
  \end{remark}

 \subsection{Relation to Ohta's twisting map}

  We now describe a relation between the above moment maps and a construction of Ohta (cf.\ \cite{Ohta-ordinary}); this is also closely related to the twisting map considered by Kato (cf.\ \cite[\S 8.4.3]{Kato-p-adic}).

  \begin{theorem}\label{thm:Ohta-twisting}
   Let $M$ be an integer dividing $N$. Suppose $p \mid N$, and let $t_N$ be the canonical order $N$ section of the universal elliptic curve $\cE$ over $Y(M, N)$.
   \begin{enumerate}
    \item There is an isomorphism
    \[
     H^1_{\et}\left(Y(M, N), \Lambda(\sH_{\Zp}\langle t_N \rangle)(1)\right) \cong \varprojlim_{r \ge 0} H^1_{\et}(Y(M, Np^r), \Zp(1)),
    \]
    where the inverse limit is with respect to the pushforward maps $(\pr_1)_*$; and this isomorphism maps the Eisenstein--Iwasawa class $\cEI_{t_N}$ to $\left(\partial({}_c g_{0, 1/Np^r})\right)_{r \ge 0}$, where ${}_c g_{0, 1/Np^r} \in \cO(Y_1(Np^r))^\times$ is the Kato--Siegel unit.
    \item The morphism
    \[ \mom^k_{\langle t_N \rangle, N}: H^1_{\et}(Y(M, N), \Lambda(\sH_{\Zp}\langle t_N \rangle)(1)) \to H^1_{\et}(Y(M,N), \TSym^k \sH_{\Zp}(1))\]
    coincides with the morphism
    \begin{align*}
     \varprojlim_{r \ge 0} H^1_{\et}(Y(M, Np^r), \Zp(1)) &\cong \varprojlim_{r \ge 0} H^1_{\et}(Y(M, Np^r), \ZZ/p^r \ZZ(1))\\
     &\to\varprojlim_{r \ge 0} H^1_{\et}(Y(M, Np^r), \TSym^k \sH_r(1))\\
     &\to H^1_{\et}(Y(M, N), \TSym^k \sH_{\Zp}(1)),
    \end{align*}
    where the second map is given by cup-product with $(N \cdot t_{Np^r})^{\otimes k} \in H^0(Y(M, Np^r), \TSym^k \sH_r)$.
   \end{enumerate}
   In particular, the image of the inverse system
   \[ \left( \partial({}_c g_{0, 1/Np^r})\right)_{r \ge 0} \]
   is $\mom^k_{\langle t_N \rangle, N}\left(\cEI_{1, N}\right) = c^2 \Eis^k_{\et, 1, N} - c^{-k} \Eis^k_{\et, c, N}$.
  \end{theorem}

  \begin{remark}
   Compare \cite[\S 1.3]{ohta95}; Ohta uses the notation $S^k(\Zp)$ for what we would call $\TSym^k(\ZZ_p^2)$, considered as a left $\GL_2(\Zp)$-module via the multiplication action of $\GL_2(\Zp)$ on column vectors.
  \end{remark}

  \begin{proof}
   It suffices to consider the case $M = 1$, since the case of general $M \mid N$ follows by pullback.

   We claim that there is an isomorphism of varieties
   \[ Y_1(Np^r) \cong \cE[p^r]\langle t_N \rangle\]
   which intertwines the map $p_{r, t}: \cE[p^r]\langle t_N \rangle \to Y_1(N)$ and the canonical projection $\pr_1: Y_1(Np^r) \to Y_1(N)$. To prove this claim, we use the moduli-space interpretation of $Y_1(N)$: a point of $Y_1(N)$ is given by a pair $(E, P)$ where $P$ has exact order $N$. Similarly, a point of $Y_1(Np^r)$ is $(E, Q)$ where $Q$ has order $Np^r$; and by definition a point of $\cE[p^r]\langle t \rangle$ over $(E, P) \in Y_1(N)$ is given by a point $Q$ such that $p^r Q = P$. So we may define our isomorphism by mapping the point $\big( (E, P), Q\big)$ of $\cE[p^r]\langle t_N \rangle$ to $(E, Q) \in Y_1(Np^r)$. The reverse bijection is given by $(E, Q) \mapsto \big( (E, p^r Q), Q\big)$. Thus we have
   \[ H^1_{\et}(\cE[p^r]\langle t_N \rangle, \Zp(1)) \cong H^1_{\et}(Y_1(Np^r), \Zp(1)) \]
   for all $r \ge 0$, and passing to the inverse limit over $r$ gives the required isomorphism. Moreover, the inclusion $\cE[p^r]\langle t_N \rangle \subseteq \cE$ corresponds to the canonical section $t_{Np^r}$ over $Y_1(Np^r)$, so the Siegel unit ${}_c g_{0, 1/Np^r}$ on $Y_1(Np^r)$ is just the restriction of ${}_c \theta_{\cE} \in \cO(\cE \setminus \cE[c]^\times)$ to $\cE[p^r]\langle t_N \rangle$. Applying the Kummer map to each side gives $\cEI_{1, N} = (\partial({}_c g_{0, 1/Np^r}))_{r \ge 0}$.

   We now prove (2). We know that the moment map coincides with the Soul\'e twisting map \cite[\S 2.6]{Kings-Eisenstein}. Thus it suffices to check that the section $\tau_{r, t_N} \in H^0(\cE[p^r]\langle t_N \rangle, p_{r, t_N}^* \sH_r)$ defined in (2.5.1) of \emph{op.cit.} corresponds under the above isomorphism to $N \cdot t_{Np^r}$, which is clear by construction.
  \end{proof}

  \begin{remark}
   This statement is, of course, not true for $p \nmid N$ without some minor modifications, since ${}_c g_{0, 1/N}$ is not the image under the norm map of ${}_c g_{0, 1/Np}$ if $p \nmid N$.

   It is worth noting that the moment map $\varprojlim_{r} H^1_{\et}(Y(M, Np^r), \Zp(1)) \to \varprojlim_{r} H^1_{\et}(Y(M, Np^r), \TSym^k \sH_{\Zp}(1))$ commutes with the Hecke operators $T_\ell'$ for $\ell \nmid N$, $U_\ell'$ for $\ell \mid N$, and $\stbt a001$ for $a \in (\ZZ / M\ZZ)^\times$, but intertwines $\stbt 1 0 0 b$, for $b \in \varprojlim_r (\ZZ / Np^r \ZZ)^\times$, with $b^{-k} \stbt 100b$.
  \end{remark}


\section{Rankin--Iwasawa classes and norm relations}


 In this section, we shall define classes $\cRI^{[j]}_{M, N, a}$, which are \'etale cohomology classes on the products $Y(M, N)^2$ (for $M \mid N$) with coefficients in a $\Lambda$-adic sheaf. The role of these classes is to interpolate the \'etale Eisenstein classes $\Eis^{[k, k', j]}_{\et, 1, N}$ (for a fixed integer $j \ge 0$, and varying $k, k' \ge j$). The construction is somewhat messy for general $j$, but very much simpler when $j = 0$, so the reader may wish to keep the case $j = 0$ in mind on a first reading.

 \subsection{An Iwasawa-theoretic Clebsch--Gordan map}
  \label{sect:iwaclebschgordan}

  We now define a morphism on the sheaves $\Lambda(\sH_{\Zp})$, whose images under the moment maps will correspond to the \'etale Clebsch--Gordan maps defined in \S \ref{sect:clebschgordan} considered above.

  Recall that for an elliptic curve $\cE / S$ and a section $t \in \cE[S]$, we have defined
  \[ \Lambda_r(\sH_{r}\langle t \rangle) = (p_{r, t})_* (\ZZ/p^r \ZZ),\]
  where $p_{r, t}$ is the structure map $\cE[p^r]\langle t \rangle \to S$. There is a morphism of sheaves
  \[\Lambda_r(\sH_{r}\langle t \rangle) \to \Lambda_r(\sH_{r}\langle t \rangle) \otimes \Lambda_r(\sH_{r}\langle t \rangle)
  \]
  given by the diagonal inclusion of $\cE[p^r]\langle t \rangle$ into $\cE[p^r]\langle t \rangle\times_S \cE[p^r]\langle t \rangle$. These morphisms are compatible as $r$ varies and assemble into a morphism
  \begin{equation}
   \label{eq:diagonal}
   \Lambda(\sH_{\Zp}\langle t \rangle) \to \Lambda(\sH_{\Zp}\langle t \rangle) \htimes \Lambda(\sH_{\Zp}\langle t \rangle).
  \end{equation}

  We also have a morphism of sheaves $\Zp \to (\TSym^j \sH_{\Zp})^{\otimes 2}(-j)$ for any $j \ge 0$, which is the special case $k = k' = j$ of the Clebsch--Gordan map.

  \begin{definition}
   For $j \ge 0$, let us write
   \[ \Lambda(\sH_{\Zp}\langle t \rangle)^{[j]} = \Lambda(\sH_{\Zp}\langle t \rangle) \otimes \TSym^j \sH_{\Zp}.\]
   We define a morphism
   \[ CG^{[j]} : \Lambda(\sH_{\Zp}\langle t \rangle) \to \left(\Lambda(\sH_{\Zp}\langle t \rangle)^{[j]} \htimes \Lambda(\sH_{\Zp})^{[j]}\right)(-j) \]
   as the tensor product of the two morphisms we have just defined.
  \end{definition}

  For integers $k\ge j$ we can define a moment map
  \[ \mom^{k-j} \cdot \id: \Lambda(\sH_{\Zp}\langle t \rangle)^{[j]} \to \TSym^k \sH_{\Zp}\]
  as the composition
  \[
   \Lambda(\sH_{\Zp}\langle t \rangle)\otimes\TSym^j \sH_{\Zp} \rTo^{\mom^{k-j} \otimes \id}
   \TSym^{k-j} \sH_{\Zp} \otimes \TSym^j \sH_{\Zp} \rTo^\times \TSym^k \sH_{\Zp}
  \]
  where ``$\id$'' denotes the identity on $\TSym^j \sH$, and the second arrow is the product in the ring $\TSym^\bullet \sH_{\Zp}$ (the symmetrisation of the naive tensor product).

  \begin{proposition}
   \label{prop:moment-compat}
   For integers $0 \le j \le k,k'$ there is a commutative diagram of pro-sheaves on $S$
   \begin{diagram}
    \Lambda(\sH_{\Zp}\langle t \rangle) & \rTo^{\scriptstyle CG^{[j]}} &\left(\Lambda(\sH_{\Zp}\langle t \rangle)^{[j]} \htimes \Lambda(\sH_{\Zp}\langle t \rangle)^{[j]}\right)(-j) \\
    \dTo^{\scriptstyle \mom^{k + k' - 2j}} & & \dTo_{\scriptstyle (\mom^{k-j} \cdot \id) \otimes (\mom^{k'-j} \cdot \id)}\\
    \TSym^{k+k'-2j} \sH_{\Zp} & \rTo^{\scriptstyle CG^{[k,k',j]}}& \left(\TSym^k \sH_{\Zp} \otimes \TSym^{k'} \sH_{\Zp}\right)(-j).
   \end{diagram}
  \end{proposition}

  \begin{proof}
   Clear from the construction of the maps $CG^{[j]}$ and $CG^{[k, k', j]}$.
  \end{proof}


  Let us temporarily write $Y$ for $Y(M, N)[1/p]$ and $Y^2$ for its self-product over $\ZZ\left[ \frac{1}{N}, \mu_M \right]$.

  \begin{notation}
   Given sheaves $\mathcal{A}, \mathcal{B}$ on $Y$ we write $\mathcal{A} \boxtimes \mathcal{B}$ for the sheaf on $Y^2$ given by $\pi_1^* \mathcal{A} \otimes \pi_2^* \mathcal{B}$, where $\pi_1, \pi_2$ are the first and second projections from $Y^2 \to Y$.
  \end{notation}

  To shorten the notation, we write $\Lambda(\sH_{\Zp}\langle t_N \rangle)^{[j, j]}$ for the sheaf $\Lambda(\sH_{\Zp}\langle t_N \rangle)^{[j]} \boxtimes \Lambda(\sH_{\Zp}\langle t_N \rangle)^{[j]}$ on $Y^2$, where $t_N$ is the canonical order $N$ section.

  Since $\Delta$ has relative dimension 1, we obtain pushforward maps
  \[
   \Delta_*: H^1_{\et}\left( Y, \Lambda(\sH_{\Zp}\langle t_N \rangle)^{[j]} \otimes \Lambda(\sH_{\Zp}\langle t_N \rangle)^{[j]}(1 - j)\right) \to H^3_{\et}\left( Y^2, \Lambda(\sH_{\Zp}\langle t_N \rangle)^{[j,j]}(2 - j)\right).
  \]

  We also have an action of $\ZZ/ M\ZZ$ on $Y(M, N)$ via
  \[a \cdot (E, e_1, e_2) = \left(E, e_1 + a\tfrac{N}{M}e_2, e_2\right).\]

  \begin{notation}
   Let $u_a$ be the automorphism of $Y^2$ that is the identity in the first factor and the action of $a$ in the second factor.
  \end{notation}

  The sheaf $\Lambda(\sH_{\Zp}\langle t_N \rangle)^{[j, j]}$ is canonically isomorphic to its pullback by $u_a$, so $u_a$ acts on its cohomology. This leads to the following definition:

  \begin{definition}
   \label{def:RIclass}
   We define the \emph{Rankin--Iwasawa class}
   \[ \cRI_{M, N, a}^{[j]} = ((u_a)_* \circ \Delta_* \circ CG^{[j]})(\cEI_{1, N}) \in H^3_{\et}\left( Y^2, \Lambda(\sH_{\Zp}\langle t_N \rangle)^{[j, j]}(2-j)\right).\]
  \end{definition}


 \subsection{First properties of the Rankin--Iwasawa class}

  \begin{notation}
   \label{not:dstar}
   We use the following notations. We assume that $M|N$.
   \begin{enumerate}
    \item For $d \in \ZZ$, let $[d]_*$ denote the morphism of sheaves on $Y(M, N)^2$,
    \[ \Lambda(\sH_{\Zp}\langle t_N \rangle)^{[j]} \to \Lambda(\sH_{\Zp}\langle d t_N \rangle)^{[j]} \]
    given by the tensor product of pushforward by the $d$-multiplication on the factor $\Lambda(\sH_{\Zp})$, and the \emph{trivial} map on $\TSym^j \sH_{\Zp}$.
    \item For $x \in (\ZZ / N\ZZ)^*$, let $\langle x \rangle$ denote the automorphism of $Y(M, N)$ over $\ZZ[1/N, \mu_M]$ given by $(E, e_1, e_2) \to (E, x^{-1} e_1, x e_2)$; and let $\sigma_x$, for $x \in (\ZZ / M\ZZ)^*$, be the automorphism $(E, e_1, e_2) \to (E, x e_1, e_2)$.
    \item Denote the automorphism $(\sigma_x, \sigma_x)$ of $Y(M, N)^2$ simply as $\sigma_x$.
   \end{enumerate}
  \end{notation}

  \begin{remark}
   The utility of the (slightly curious) definition of $[d]_*$ is that it interacts well with the Clebsch--Gordan map: we have $CG^{[j]} \circ [d]_* =    ([d]_*, [d]_*) \circ CG^{[j]}$, as is clear from the construction of the map $CG^{[j]}$.
  \end{remark}

  \begin{proposition}
   \label{prop:RIfirstproperties}
   The elements $\cRI_{M, N, a}^{[j]}$ have the following properties:
   \begin{enumerate}
    \item We have
    \[ \rho^* \big(\cRI_{M, N, a}^{[j]}\big) = (-1)^j \cRI_{M, N, -a}^{[j]},\]
    where $\rho$ is the involution of $Y^2$ which interchanges the two factors.
    \item For $c, d > 1$ coprime to $6Np$, the element
    \[ \left[d^2 - \big([d]_*^{-1} \langle d \rangle,[d]_*^{-1} \langle d \rangle\big) \sigma_d^2 \right] \cRI_{M, N, a}^{[j]}\]
    is symmetric in $c$ and $d$.
    \item\label{item:RIfirstproperties-moments} For any integers $(k, k')$ such that $(k, k', j)$ satisfies the inequality $0 \le j \le \min(k, k')$ of \eqref{eq:inequalities}, we have
    \[
     \left((\mom^{k-j} \cdot \id) \boxtimes (\mom^{k'-j} \cdot \id)\right) \left( \cRI_{M, N, a}^{[j]} \right) =
     \left[c^2 - c^{2j-k-k'}\big(\langle c \rangle,\langle c \rangle\big) \sigma_c^2 \right] (u_a)_* \left(\Eis^{[k, k', j]}_{\et, 1, N}\right).
    \]
    In particular, the image of $\cRI_{M, N, a}^{[j]}$ under this moment map is the image of a motivic cohomology class, for all such $k, k'$.
    \item We have
    \[ \sigma_b  \cdot \cRI_{M, N, a}^{[j]} = \cRI_{M, N, b^{-1} a}^{[j]}\]
    for any $b \in (\ZZ / M\ZZ)^\times$.
   \end{enumerate}
  \end{proposition}

  \begin{proof}
   The proofs of these statements are exactly the same as in the case of Siegel units, which is Proposition 2.6.2 of \cite{LLZ14}.
  \end{proof}


 \subsection{The first norm relation}

  We now give a norm relation for the classes $\cRI^{[j]}_{M, N, a}$ as $N$ varies, generalizing Theorem 3.1.1 of \cite{LLZ14}. As in Definition \ref{def:lambdapr1} above, we consider $\pr_1$ as a map
  \[ \left(Y(M, N\ell), \Lambda(\sH_{\Zp}\langle t_{N\ell}\rangle)^{[j]}\right) \to \left(Y(M, N), \Lambda(\sH_{\Zp}\langle t_{N}\rangle)^{[j]}\right)\]
  by composing with the map $[\ell]_*$ of Notation \ref{not:dstar}.

  \begin{theorem}
   \label{thm:norm1}
   Let $M, N$ be integers with $M \mid N$ and $M + N \ge 5$, and $\ell$ a prime. Then
   \[
     (\pr_1 \times \pr_1)_* \left( \cRI^{[j]}_{M, \ell N, a}\right) = \begin{cases}
     \cRI^{[j]}_{M, N, a}, & \text{if $\ell \mid N$,}\\
     \left[ 1 - \big([\ell]_*\langle \ell^{-1} \rangle, [\ell]_*\langle \ell^{-1} \rangle\big) \sigma_\ell^{-2}\right]\cRI^{[j]}_{M, N, a}, & \text{if $\ell \nmid N$.}
    \end{cases}
   \]
  \end{theorem}

  \begin{proof}
   This follows immediately from Theorem \ref{thm:pushforward} and the commutativity of the diagram
   \[
    \begin{diagram}[small]
    Y(M, N\ell) & \rTo & Y(M, N\ell)^2 \\
    \dTo & & \dTo \\
    Y(M, N) & \rTo & Y(M, N)^2.
   \end{diagram}
  \qedhere\]
  \end{proof}


 \subsection{The second norm relation}

  Our next result is a version of Theorem 3.3.1 of \cite{LLZ14}. We fix integers $M, N$ and a prime $\ell$ with $M + N \ge 5$ and $M\ell \mid N$. Recall the degeneracy maps $\hpr_1$ and $\hpr_2: Y(M\ell, N) \to Y(M, N)$ introduced in \S \ref{sect:degeneracy} above.

  \begin{theorem}
   \label{thm:norm2}
   Suppose $a \in \ZZ / M\ell \ZZ$ is not divisible by $\ell$. Then we have
   \[ (\hpr_2 \times \hpr_2)_*\left( \cRI_{M\ell, N, a}^{[j]}\right) =
    \begin{cases}
     (U_\ell', U_\ell') \cdot \cRI_{M, N, a}^{[j]} & \text{if $\ell \mid M$},\\
     \left[ (U_\ell', U_\ell') - \ell^j \sigma_\ell\right] \cdot \cRI_{M, N, a}^{[j]} & \text{if $\ell \nmid M$.}
    \end{cases}
   \]
  \end{theorem}

  Before embarking on the proof, we need some preparatory lemmas. For $a \in \ZZ / M\ZZ$, we write $\iota_{M, N, a}$ for the map $u_a \circ \Delta: Y(M, N) \to Y(M, N)^2$, and similarly $\iota_{M\ell, N, a}$ for $a \in \ZZ / M\ell \ZZ$.

  \begin{lemma}
   \label{lemma:cartesian1}
   Let $a \in \ZZ / M\ell\ZZ$ be not divisible by $\ell$. Then the composition
   \[\iota_{M(\ell), N, a}:  Y(M\ell, N)\rTo^{\iota_{M\ell, N, a}} Y(M\ell, N)^2 \rTo^{\hpr' \times \hpr'} Y(M(\ell), N)^2,\]
   where the second arrow is the natural degeneracy map, is a closed embedding. If moreover $\ell \mid M$, then the diagram
   \[
    \begin{diagram}[small,labelstyle=\scriptstyle]
     Y(M\ell, N) & \rTo^{\iota_{M(\ell), N, a}} & Y(M(\ell), N)^2 \\
     \dTo<{\hpr_1} & & \dTo>{\hpr \times \hpr}\\
     Y(M, N) &\rTo^{\iota_{M, N, a}} & Y(M, N)^2
    \end{diagram}
   \]
   is Cartesian, where the vertical maps are the natural projections.
  \end{lemma}

  \begin{proof}
   We show first that $\iota_{M(\ell), N, a}$ is a closed embedding. Its image is clearly closed, so it suffices to show that it is injective. This we may check on $\CC$-points.

   So it suffices to show that the preimage of $U(M(\ell), N) \times U(M(\ell), N)$ under the map $\GL_2(\mathbf{A}_\QQ) \to (\GL_2 \times_{\GL_1} \GL_2)(\mathbf{A}_\QQ)$ given by $x \mapsto \left( x, \stbt 1 a 0 1 x \stbt 1 a 0 1^{-1} \right)$ is $U(M\ell, N)$. This is a completely elementary calculation: if $x = \stbt rstu \in U(M(\ell), N)$, then $\stbt 1 a 0 1 x \stbt 1 a 0 1^{-1}$ is congruent to $\tbt 1{a(1-r)}01$ modulo $\tbt M{M\ell}NN$; and $a \notin \ell \widehat\ZZ$, so if this is to lie in $U(M(\ell), N)$, then we must have $r = 1 \bmod M\ell$, i.e.\ $x \in U(M\ell, N)$.

   Let us now show the ensuing square is Cartesian. Since both horizontal arrows are closed immersions and the vertical ones are surjective, it suffices to show that the vertical maps have the same degree. However, since $\ell \mid M$ the degree of $Y(M\ell, N)$ over $Y(M, N)$ is $\ell^2$, which is also the degree of $Y(M(\ell), N)^2$ over $Y(M, N)^2$.
  \end{proof}

  If $\ell \nmid M$ then we need to use a slightly modified version of the above statement. Let $\tilde a$ be the unique lifting of $a \in \ZZ / M\ZZ$ to an element of $\ZZ / \ell M \ZZ$ divisible by $\ell$.

  \begin{notation}
   Denote by $\gamma$ the map $Y(M(\ell), N) \to Y(M(\ell), N)^2$ given by $u_{\tilde a} \circ \Delta$ (which is an embedding, by the same matrix calculation as before).
  \end{notation}

  \begin{lemma}
   In the above setting, the following diagram is Cartesian:
   \[
    \begin{diagram}[small]
     Y(M\ell, N) \sqcup Y(M(\ell), N) & \rTo^{(\iota_{M(\ell), N, a}, \gamma)} & Y(M(\ell), N)^2 \\
     \dTo & & \dTo \\
     Y(M, N) &\rTo^{\iota_{M, N, a}} & Y(M, N)^2
    \end{diagram}
   \]
   where the vertical arrows are the natural projection maps.
  \end{lemma}

  \begin{proof}[Proof of Theorem \ref{thm:norm2}]
   We factor the map $\hpr_2$ as the composite
   \[ Y(M\ell, N) \rTo^{\hpr'} Y(M(\ell), N) \rTo^{\varphi_{\ell^{-1}}}_\cong Y(M, N(\ell)) \rTo^{\pr} Y(M, N),\]
   and for brevity we write $\tilde\pr = \pr \mathop\circ \varphi_{\ell^{-1}}$.

   Consider first the case $\ell \mid M$. The Cartesianness of the diagram of Lemma \ref{lemma:cartesian1}, together with the commutativity of pushforward and pullback in Cartesian diagrams (cf.\ Remark 2.4.6 of \cite{LLZ14}) now implies that
   \[ (\hpr' \times \hpr')_* \left(\cRI_{M\ell, N, a}^{[j]}\right) = (\hpr \times \hpr)^*\left(\cRI_{M, N, a}^{[j]}\right)\]
   as elements of $H^3_{\et}(Y(M(\ell), N), \Lambda^{[j, j]}(2-j))$. Applying the map $(\tilde\pr \times \tilde\pr)_*$ to both sides of this formula, and noting that $\tilde\pr_* \circ \hpr^* = (\pr)_* \circ (\varphi_{\ell^{-1}} )_* \circ \hpr^*$ is the definition of the Hecke operator $U_\ell'$, we obtain the result.

   The case $\ell \nmid M$ is similar, although slightly more elaborate. The same argument as before tells us that
   \[ (U_\ell', U_\ell') \cdot \cRI_{M, N, a}^{[j]} = (\hpr_1 \times \hpr_1)_* \left(\cRI_{M\ell, N, a}^{[j]}\right) + A\]
   where
   \[ A \coloneqq (\tilde\pr \times \tilde\pr)_* \left(\gamma_* \circ CG^{[j]}\right)\big(\cEI_{1, N}\big).
   \]
   There is a commutative diagram
   \[
    \begin{diagram}[small]
     Y(M(\ell), N) &\rTo^{\gamma} & Y(M(\ell), N)^2 \\
     \dTo<{\tilde\pr}&&\dTo>{\tilde\pr \times \tilde\pr}\\
     Y(M, N) & \rTo^{\iota_{M, N, \ell^{-1}a}}& Y(M, N)^2,
    \end{diagram}
   \]
   so we have
   \[ A = \left( (\iota_{M, N, \ell^{-1}a})_* \circ \tilde\pr_* \circ CG^{[j]}\right) \left(\cEI_{1, N}\right).\]
   Since $\tilde\pr_*\left(\cEI_{1, N}\right) = \cEI_{1, N}$ by Theorem \ref{thm:pushforward2}(2), it suffices to show that
   \[ (\tilde\pr)_* \circ CG^{[j]} = \ell^j CG^{[j]} \circ (\tilde\pr)_*.\]
   Recall that $CG^{[j]}$ is defined using the $j$-th tensor power of the identification $\det \sH_{\Zp} \cong \Zp(1)$ given by the Weil pairing. Now, by definition, the map $\tilde\pr_* =\pr_* \circ (\varphi_{\ell^{-1}})_*$ acts on $\sH_{\Zp}$ via $\lambda'_*$, where $\lambda'$ is the isogeny $\cE \to (\varphi_{\ell^{-1}})^* \cE$ defined in Section \ref{sect:degeneracy}. Since $\lambda'$ has degree $\ell$, we deduce from Lemma \ref{lemma:isogeny} that
   \[  (\varphi_{\ell^{-1}})_*\circ CG^{[j]} = \ell^j CG^{[j]} \circ (\varphi_{\ell^{-1}})_*,\]
   as required.
  \end{proof}


 \subsection{The third norm relation}

  The last relation we shall need is the following. Recall that $\pr_2$ denotes the ``twisted'' degeneracy map $Y(M, N\ell) \to Y(M, N)$, corresponding to $z \mapsto \ell z$ on the upper half-plane, and we extend this to a map on our coefficient sheaves using the isogeny $\lambda: \cE \to \varphi_\ell^*(\cE)$ of elliptic curves over $Y(M, N(\ell))$. 

  \begin{theorem}
   \label{thm:norm3}
  The map $(\pr_1 \times \pr_2)_*$ sends $\cRI^{[j]}_{M, \ell N, a}$ to
   \[
    \begin{cases}
      (U_\ell', 1) \cdot \cRI^{[j]}_{M, N, \ell a} & \text{if $\ell \mid N$}, \\
      \left[ (T_\ell', 1) \sigma_\ell^{-1} -
      ([\ell]_* \langle \ell^{-1} \rangle, T_\ell') \sigma_\ell^{-2} \right] \cdot \cRI^{[j]}_{M, N, a} & \text{if $\ell \nmid N$.}
    \end{cases}
   \]
  \end{theorem}

  The proof of this statement closely follows that of Lemma A.2.1 of \cite{LLZ2}, and we leave it to the reader to make the necessary modifications for the $\Lambda$-adic case.

  \begin{corollary}
   \label{cor:norm3b}
   We have
   \[
    (\pr_2 \times \pr_1)_* \left(\cRI^{[j]}_{M, \ell N, a}\right)
    = \begin{cases}
      (1, U_\ell') \cdot \cRI^{[j]}_{M, N, \ell a} & \text{if $\ell \mid N$}, \\
      \left[ (1, T_\ell') \sigma_\ell^{-1} -
      ( T_\ell',[\ell]_* \langle \ell^{-1} \rangle) \sigma_\ell^{-2} \right] \cdot \cRI^{[j]}_{M, N, a} & \text{if $\ell \nmid N$.}
    \end{cases}\]
   and
   \[ (\pr_2 \times \pr_2)_* \cRI^{[j]}_{M, \ell N, a} =
    \begin{cases}
     \ell^{j+1} \cRI^{[j]}_{M, N, \ell a} & \text{if $\ell \mid N$,}\\
     \ell^{j+1} \sigma_\ell^{-1} \left(1 - ([\ell]_*\langle \ell^{-1} \rangle, [\ell]_* \langle \ell^{-1} \rangle) \sigma_\ell^{-2})\right) \cRI^{[j]}_{M, N, a} &\text{if $\ell \nmid N$.}
    \end{cases}
   \]
  \end{corollary}

  \begin{proof}
   The first statement follows from the previous theorem by symmetry. The second follows by writing
   \begin{align*}
    (\pr_2 \times \pr_2)_*\left(\cRI^{[j]}_{M, N\ell, a}\right) &= (\pr_2 \times \pr_2)_*(\pr_1 \times \pr_1)_*\left(\cRI^{[j]}_{M, N\ell^2, a}\right)\\
    &=  (\pr_2 \times \pr_1)_*(\pr_1 \times \pr_2)_* \left(\cRI^{[j]}_{M, N\ell^2, a}\right)\\
    &=  (\pr_2 \times \pr_1)_* \left[ (U_\ell', 1) \cRI^{[j]}_{M, N\ell, \ell a}\right].
   \end{align*}
   On the sheaf $\Lambda(\sH_{\Zp}\langle t_{N\ell}\rangle)^{[j]}$ we have the relation
   \[ \pr_2 \circ U_\ell' = \ell^{j + 1} \pr_1,\]
   so this gives
   \[ (\pr_2 \times \pr_2)_*\left(\cRI^{[j]}_{M, N\ell, \ell a}\right) = \ell^{j + 1} (\pr_1 \times \pr_1)_*\left(\cRI^{[j]}_{M, N\ell, \ell a}\right)\]
   and we are done by Theorem \ref{thm:norm1}.
  \end{proof}


 \subsection{The Euler system distribution relation}

  From the three basic norm relations above -- Theorems \ref{thm:norm1}, \ref{thm:norm2} and \ref{thm:norm3} -- we can derive all the other relations we shall need between Rankin--Iwasawa classes as corollaries, using only elementary relations in the Hecke algebra.

  The first of these corollaries is the following relation, which will be the key to the Euler system arguments of \S \ref{sect:selmerbound}:

  \begin{proposition}
   \label{prop:cyclonorm}
   Let $\ell$ be a prime with $\ell \nmid Np$, and let $a \in \ZZ / \ell M \ZZ$ with $\ell \nmid a$. Then for any $c > 1$ coprime to $6\ell M N p$, pushforward along the map
   \[ Y(\ell M, \ell N)^2 \rTo^{\hpr_2 \times \hpr_2} Y(M, \ell N)^2 \rTo^{\pr_1 \times \pr_1} Y(M, N)^2, \]
   maps the class $\cRI^{[j]}_{\ell M, \ell N, a}$ to the following class:
   \begin{multline*}
    \Big( -\ell^j \sigma_\ell + (T_\ell', T_\ell') +
     \left((\ell + 1)\ell^j(\langle \ell \rangle^{-1}[\ell]_*, \langle \ell \rangle^{-1} [\ell]_*) - (\langle \ell \rangle^{-1} [\ell]_*,  T_{\ell}'^2) - (T_{\ell}'^2, \langle \ell \rangle^{-1} [\ell]_*)\right) \sigma_\ell^{-1} \\
       + (\langle \ell^{-1} \rangle [\ell]_* T_\ell',\langle \ell^{-1} \rangle [\ell]_* T_\ell') \sigma_\ell^{-2} - \ell^{1 + j} ([\ell^2]_* \langle \ell^{-2}\rangle, [\ell^2]_* \langle \ell^{-2}\rangle)\sigma_\ell^{-3}\Big) \cRI^{[j]}_{M, N, a}.
   \end{multline*}
  \end{proposition}

  \begin{proof}
   By the second norm relation (Theorem \ref{thm:norm2}), pushforward along $\hpr_2 \times \hpr_2: Y(\ell M, \ell N)^2 \to Y(M, \ell N)^2$ maps the class $\cRI^{[j]}_{\ell M, \ell N, a}$ to
   \[ ((U_\ell', U_\ell') - \ell^j \sigma_\ell) \cRI^{[j]}_{M, \ell N, a}.\]
   So we must compute the pushforward of this element along the natural degeneracy map $\pr_1 \times \pr_1: Y(M, \ell N)^2 \to Y(M, N)^2$.

   With our present conventions, as maps $\left(Y(M, N\ell), \Lambda(\sH_{\Zp}\langle t_{N\ell}\rangle)^{[j]}\right) \to \left(Y(M, N), \Lambda(\sH_{\Zp}\langle t_{N}\rangle)^{[j]}\right)$ we have the relations
   \begin{subequations}
    \begin{equation}
     \label{eq:pridentities1}
     (\pr_1)_* \circ U_\ell' = T_\ell' \circ (\pr_1)_* - [\ell]_* \langle \ell^{-1} \rangle \circ (\pr_2)_*
    \end{equation}
    and
    \begin{equation}
     \label{eq:pridentities2}
     (\pr_2)_* \circ U_\ell' = \ell^{1 + j} \cdot (\pr_1)_*,
    \end{equation}
   \end{subequations}
   which are the $\Lambda$-adic versions of Proposition \ref{prop:pridentity1}. Applying the first relation to both factors of the product $Y(M, \ell N)^2$, we have
   \begin{multline*}
    \left((\pr_1 \times \pr_1)_* \circ ((U_\ell', U_\ell') - \ell^j \sigma_\ell)\middle) \middle( \cRI^{[j]}_{M, \ell N, a} \right) \\
    = \Big(((T_\ell', T_\ell') - \ell^j \sigma_\ell) (\pr_1 \times \pr_1)_*  \\- (T_\ell', [\ell]_* \langle \ell^{-1} \rangle)(\pr_1 \times \pr_2)_* - ( [\ell]_* \langle \ell^{-1} \rangle, T_\ell')(\pr_2 \times \pr_1)_*\\
    + ([\ell]_* \langle \ell^{-1}\rangle, [\ell]_* \langle \ell^{-1}\rangle)(\pr_2 \times \pr_2)_*\Big)\left(\cRI^{[j]}_{M, \ell N, a} \right).
   \end{multline*}
   We have formulae for the images of $\cRI^{[j]}_{M, \ell N, a}$ under each of the four maps
   \[\Big\{   (\pr_1 \times \pr_1)_*, (\pr_1 \times \pr_2)_*, (\pr_2 \times \pr_1)_*, (\pr_2 \times \pr_2)_*\Big\}\]
   as the image of $\cRI^{[j]}_{M, N, a}$ under a Hecke operator (by Theorem \ref{thm:norm1}, Theorem \ref{thm:norm3}, and the two cases of Corollary \ref{cor:norm3b} respectively). Combining these gives the stated formula.
  \end{proof}

  \begin{remark}
   Compare Theorem 3.4.1 of \cite{LLZ14}, which is essentially the above theorem with trivial coefficients. (In fact the above result is slightly stronger, cf.\ Remark 3.4.2 of \emph{op.cit.}.)
  \end{remark}

 \subsection{The \texorpdfstring{$\ell$}{l}-stabilisation relation}
  \label{sect:lstab}
  Our second application of the three basic norm relations is to prove a theorem relating Rankin--Eisenstein classes at levels prime to $\ell$ with those at level divisible by $\ell$. We will use this later for $\ell = p$, in order to relate Hida theory (which requires the level to be divisible by $p$) with the syntomic regulator (which requires the level to be coprime to $p$).

  \subsubsection*{The ``abstract'' \texorpdfstring{$\ell$}{l}-stabilisation relation} The following construction is adapted from that used in \cite[Theorem 4.2.8]{LLZ2}, which we learned from the work of Wiles (cf.\ p490 of \cite{Wiles-FLT}). Let $\ell$ be a prime not dividing $N$, and $j \ge 0$ an integer.

  \begin{definition}
   We define a map
   \[ (\Pr \times \Pr)_*:
   H^3_{\et}\left(Y(M, N\ell)^2, \Lambda(\sH_{\Zp}\langle t_{N\ell}\rangle)^{[j, j]}(2-j)\right) \to H^3_{\et}\left(Y(M, N)^2, \Lambda(\sH_{\Zp}\langle t_{N}\rangle)^{[j, j]}(2-j)\right)^{\oplus 4}\]
   via the formula
   \[ (\Pr \times \Pr)_* \coloneqq
    \begin{pmatrix}
    (\pr_1 \times \pr_1)_* \\
    (\pr_2 \times \pr_1)_*\\
    (\pr_1 \times \pr_2)_*\\
    (\pr_2 \times \pr_2)_*
   \end{pmatrix}.\]
  \end{definition}

  The map $(\Pr \times \Pr)_*$ commutes with the Hecke operators $T_q'$, for $q \nmid N\ell$. (It is evidently induced by a correspondence from $Y(M, N\ell)^2$ to the disjoint union of four copies of $Y(M, N)^2$, but we shall not use this interpretation directly.)

  \begin{definition}
   On the module $H^3_{\et}\left(Y(M, N)^2, \sH^{[k, k']}(2-j)\right)^{\oplus 4}$, we define endomorphisms $(A_\ell', 1)$ and $(1, A_\ell')$ via left-multiplication by the matrices
   \[ (A_\ell', 1) =
    \begin{pmatrix}
        (T'_\ell, 1) &  -([\ell]_*\langle \ell^{-1} \rangle, 1)& 0 & 0 \\
        \ell^{j + 1} & 0 & 0 & 0 \\
        0 & 0 & (T'_\ell, 1) &  -([\ell]_*\langle \ell^{-1} \rangle, 1)  \\
        0 & 0& \ell^{j + 1} & 0  \end{pmatrix}.\]
   and
   \[
    (1, A_\ell') =  \begin{pmatrix}
         (1, T'_\ell) & 0 & -(1, [\ell]_*\langle \ell^{-1} \rangle) & 0 \\
         0 & (1, T'_\ell) &  0& -(1, [\ell]_*\langle \ell^{-1} \rangle) \\
         \ell^{1 + j} & 0 & 0 & 0 \\
         0 & \ell^{1 +j} & 0 & 0 \end{pmatrix}\]
   We define $(B'_\ell, 1) = (T'_\ell, 1) - (A'_\ell, 1)$, and similarly for $(1, B_\ell')$
  \end{definition}

  Note that these four matrices all commute with each other. They are chosen in order to give the following compatibility:

  \begin{lemma}
   These operators satisfy the relations
   \[ (A_\ell', 1) \circ (\Pr \times \Pr)_* = (\Pr \times \Pr)_* \circ (U_\ell', 1) \quad\text{and}\quad(1, A_\ell') \circ (\Pr \times \Pr)_* = (\Pr \times \Pr)_* \circ (1, U_\ell').\]
  \end{lemma}

  \begin{proof}
   This is simply a restatement of the formulae \eqref{eq:pridentities1} and \eqref{eq:pridentities2} for $(\pr_1)_* \circ U_\ell'$ and $(\pr_2)_* \circ U_\ell'$.
  \end{proof}

  \begin{theorem}
   \label{thm:lstab}
   For $a \in (\ZZ / M \ZZ)^\times$, we have
   \[ (\Pr \times \Pr)_* \left( \cRI^{[j]}_{M, N\ell, a}\right) =
    \left( 1 - \frac{(A_\ell', B_\ell')}{\ell^{1 + j}}\right) \left( 1 - \frac{(B_\ell', A_\ell')}{\ell^{1 + j} }\right)  \left( 1 - \frac{(B_\ell', B_\ell')}{\ell^{1 + j} }\right)  \begin{pmatrix}  \cRI^{[j]}_{M, N, a} \\ 0 \\ 0 \\ 0 \end{pmatrix}.\]
  \end{theorem}

  \begin{proof}
   Multiplying out the various $4 \times 4$ matrices
   \footnote{This is messy but can be done easily using a computer algebra system such as Sage, working in a polynomial ring with six formal variables corresponding to the operators $(T_\ell', 1)$, $(1, T_\ell')$, $([\ell]_* \langle \ell^{-1} \rangle,1)$, $(1, [\ell]_* \langle \ell^{-1} \rangle)$, $\ell^{1 +j}$, and $\sigma_\ell^{-1}$.}
one finds that the product of the three brackets on the right-hand side is given by a matrix whose first column is
   \[ \begin{pmatrix}
    1 - ([\ell]_* \langle \ell^{-1} \rangle, [\ell]_* \langle \ell^{-1} \rangle) \sigma_\ell^{-2}\\
    (1, T'_\ell)\sigma_\ell^{-1} - (T'_\ell, [\ell]_* \langle \ell^{-1} \rangle)\sigma_\ell^{-2}\\
    (T'_\ell, 1)\sigma_\ell^{-1} - ([\ell]_* \langle \ell^{-1} \rangle, T'_\ell)\sigma_\ell^{-2} \\
    \ell^{1 + j} \sigma_\ell^{-1} \left(1 - ([\ell]_* \langle \ell^{-1} \rangle, [\ell]_* \langle \ell^{-1} \rangle)\sigma_\ell^{-2} \right)
    \end{pmatrix}.\]
    Comparing this with Theorem \ref{thm:norm1}, Theorem \ref{thm:norm3} and Corollary \ref{cor:norm3b}, these four Hecke operators are exactly the ones whose actions on $\cRI^{[j]}_{M, N, a}$ give the four components of $(\Pr \times \Pr)_* \left( \cRI^{[j]}_{M, N\ell, a}\right)$.
  \end{proof}

  \subsubsection*{Application to eigenform projections}

   We now give the motivating application of the above construction. Let $f_0, g_0$ be two Hecke eigenforms of weights $(k + 2, k' + 2)$ and levels $N_f, N_g$ (with Hecke eigenvalues in some number field $L$). We choose a prime $\frP \mid p$ of $L$. Let $S$ be a finite set of primes containing all those dividing $p N_f N_g$.

   Letting $N$ denote any integer divisible by $N_f$ and $N_g$, and with the same prime factors as $N_f N_g$, we have an \'etale Eisenstein class $\Eis^{[f_0, g_0, j]}_{\et, 1, N} \in H^1(\ZZ[1/S], M_{\Lp}(f_0 \otimes g_0)^*(-j))$ for each $0 \le j \le \min(k, k')$. It follows immediately from Theorem \ref{thm:norm1} that this is independent of $N$, so we shall denote it simply by $\Eis^{[f_0, g_0, j]}_{\et}$.

   As before, we let $\ell$ be a prime not dividing $N_f$ or $N_g$ (and we assume $\ell \in S$ without loss of generality). Extending $L$ if necessary, we choose a root $\alpha_f \in L$ of the Hecke polynomial $X^2 - a_\ell(f_0) X + \ell^{k+1} \varepsilon_f(\ell)$ of $f_0$, and we let $\beta_f = a_\ell(f) - \alpha_f$ be the complementary root. The choice of $\alpha_f$ determines a Hecke eigenform $f$ of level $\ell N_f$, with $U_\ell$-eigenvalue $\alpha_f$ (and the same Hecke eigenvalues as $f_0$ at all other primes). We also choose $\alpha_g$ and an eigenform $g$ of level $\ell N_g$ similarly. Our goal is to compare the classes
   \[ \Eis^{[f, g, j]}_{\et} \in H^1(\ZZ[1/S], M_{\Lp}(f \otimes g)^*(-j)) \quad\text{and}\quad \Eis^{[f_0, g_0, j]}_{\et} \in H^1(\ZZ[1/S], M_{\Lp}(f_0 \otimes g_0)^*(-j)).\]

   \begin{definition}
    We let $(\Pr^\alpha \times \Pr^\alpha)_*$ denote the map
    \[ H^2_{\et}\left(Y_1(\ell N)_{\overline\QQ}^2,  \sH^{[k, k']}(2-j) \right) \to H^2_{\et}\left(Y_1(N)_{\overline\QQ}^2,  \sH^{[k, k']}(2-j) \right),
    \]
    where $\sH^{[k, k']}$ denotes the coefficient sheaf $\TSym^k \sH_{\Zp} \boxtimes \TSym^{k'} \sH_{\Zp}$, defined by
    \[
     (\pr_1 \times \pr_1)_* - \frac{\beta_f}{\ell^{k + 1}} (\pr_2 \times \pr_1)_* - \frac{\beta_g}{\ell^{k' + 1}} (\pr_1 \times \pr_2)_* + \frac{\beta_f \beta_g}{\ell^{k + k' + 2}} (\pr_2 \times \pr_2)_*.
    \]
   \end{definition}

   Using Proposition \ref{prop:pridentity1} one sees that the composite $\pr_{f_0, g_0} \circ (\Pr^\alpha \times \Pr^\alpha)_*$ factors through $M_{\Lp}(f \otimes g)^*$, and defines an isomorphism
   \[ M_{\Lp}(f \otimes g)^* \to M_{\Lp}(f_0 \otimes g_0)^*, \]
   which we denote by the same symbol $(\Pr^\alpha \times \Pr^\alpha)_*$.

   \begin{theorem}
    \label{thm:lstab-eigen}
    With the above notations we have
    \[ (\Pr^\alpha \times \Pr^\alpha)_*\left( \Eis^{[f, g, j]}_{\et}\right) =
    \left(1 - \frac{\alpha_f \beta_g}{\ell^{1 + j}}\right) \left(1 - \frac{\beta_f \alpha_g}{\ell^{1 + j}}\right)\left(1 - \frac{\beta_f \beta_g}{\ell^{1 + j}}\right) \Eis^{[f_0, g_0, j]}_{\et}.\]
   \end{theorem}

   \begin{proof}
    Let us write $\mom_{f, g}$ for the composite $\pr_{f, g} \circ \left[(\mom^{k-j} \cdot \mathrm{id}) \boxtimes (\mom^{k'-j} \cdot \mathrm{id})\right]$. Then the following diagram is commutative:
    \begin{diagram}[small]
     H^3_\et(Y_1(N\ell)^2, \Lambda(\sH_{\Zp}\langle t_{N\ell}\rangle)^{[j, j]}(2-j))
     & \rTo^{(\Pr \times \Pr)_*} &
     H^3_\et(Y_1(N)^2, \Lambda(\sH_{\Zp}\langle t_{N}\rangle)^{[j, j]}(2-j))^{\oplus 4}\\
      & & \dTo>{(\mom_{f_0, g_0})^{\oplus 4}} \\
     \dTo<{\mom_{f, g}} & & H^1(\ZZ[1/S], M_{\Lp}(f_0 \otimes g_0)^*(-j))^{\oplus 4}\\
     & & \dTo>\sigma\\
     H^1(\ZZ[1/S], M_{\Lp}(f \otimes g)^*(-j))& \rTo^{(\Pr^\alpha \times \Pr^\alpha)_*} & H^1(\ZZ[1/S], M_{\Lp}(f_0 \otimes g_0)^*(-j))
    \end{diagram}
    Here the map $\sigma$ is defined by
    \[
     \left(\begin{smallmatrix} x \\ y \\ z \\ w \end{smallmatrix}\right)
     \mapsto \left(1, -\tfrac{\beta_f}{\ell^{1 + j}}, -\tfrac{\beta_g}{\ell^{1+j}}, \tfrac{\beta_f\beta_g}{\ell^{2 + 2j}}\right) \cdot \left(\begin{smallmatrix} x \\ y \\ z \\ w \end{smallmatrix}\right).
    \]
    (Note that the powers of $\ell$ in the denominators here differ slightly from those in the definition of $(\Pr^\alpha \times \Pr^\alpha)_*$, since the failure of the moment map $\mom^{k-j} \cdot \mathrm{id}$ to commute with $\pr_2$ induces a factor of $\ell^{k-j}$.)

    The $4 \times 4$ matrix of Hecke operators $(A'_\ell, 1)$ introduced above acts on the space $\left( M_{\Lp}(f_0 \otimes g_0)^*\right)^{\oplus 4}$ by the matrix
    \[
     \left(\begin{smallmatrix}
      a_\ell(f_0) & -\ell^{k-j}\varepsilon_\ell(f) & 0 & 0 \\
      \ell^{1 + j} & 0 & 0 & 0  \\
      0 & 0 & a_\ell(f_0) & -\ell^{k-j}\varepsilon_\ell(f) \\
      0 & 0 & \ell^{1 + j} & 0 \end{smallmatrix} \right)=
     \left(\begin{smallmatrix}
      \alpha_f + \beta_f & -\ell^{-1-j} \alpha_f \beta_f & 0 & 0 \\
      \ell^{1 + j} & 0 & 0 & 0  \\
      0 & 0 &  \alpha_f + \beta_f & -\ell^{-1-j} \alpha_f \beta_f \\
      0 & 0 & \ell^{1 + j} & 0 \end{smallmatrix}
     \right).
    \]
    and since $\left(1, -\tfrac{\beta_f}{\ell^{1 + j}}, -\tfrac{\beta_g}{\ell^{1+j}}, \tfrac{\beta_f\beta_g}{\ell^{2 + 2j}}\right)$ is an eigenvector for right-multiplication by this matrix, with eigenvalue $\alpha_f$, we have $\sigma\left( (A'_\ell, 1) x\right) = \alpha_f\, \sigma(x)$. Similarly,  $(B'_\ell, 1)$ corresponds via $\sigma$ to multiplication by $\beta_f$, and $(1, A'_\ell)$ and $(1, B'_\ell)$ to $\alpha_g$ and $\beta_g$.

    We now follow what the maps in the diagram do to the class $\cRI^{[j]}_{1, N\ell, 1}$. By the interpolating property of $\cRI$, the image of this class under $\mom_{f, g}$ is $(c^2 - c^{2j-k-k'}\varepsilon_f(c)^{-1}\varepsilon_g(c)^{-1})\Eis^{[f, g, j]}_{\et}$; while the image under $(\Pr \times \Pr)_*$ was computed in Theorem \ref{thm:lstab}, and this maps under $(\mom_{f_0, g_0})^{\oplus 4}$ to
    \[ (c^2 - c^{2j-k-k'}\varepsilon_f(c)^{-1}\varepsilon_g(c)^{-1})\left( 1 - \frac{(A_\ell', B_\ell')}{\ell^{1 + j}}\right) \left( 1 - \frac{(B_\ell', A_\ell')}{\ell^{1 + j} }\right)  \left( 1 - \frac{(B_\ell', B_\ell')}{\ell^{1 + j} }\right)  \left(\begin{smallmatrix} \Eis^{[f_0,g_0, j]}_{\et} \\ 0 \\ 0 \\ 0 \end{smallmatrix}\right).\]
    Applying $\sigma$ to this, and cancelling the $c$ factor (which is non-zero, since $2j-k-k' \le 0$), we find that
    \[ (\Pr^{\alpha} \times \Pr^{\alpha})_*\left( \Eis^{[f, g, j]}_{\et}\right) =
   \left( 1 - \frac{\alpha_f \beta_g}{p^{1 + j}}\right)  \left( 1 - \frac{\beta_f \alpha_g}{p^{1 + j}}\right) \left( 1 - \frac{\beta_f \beta_g}{p^{1 + j}}\right)\Eis^{[f_0, g_0, j]}_{\et}.\qedhere
   \]
   \end{proof}

  \begin{remark}
   Compare Corollary 6.7.8 of \cite{LLZ14}. Our present result is somewhat stronger even in the case $k = k' = j = 0$ considered in \emph{op.cit.}, since we do not need to impose the additional hypothesis that was Assumption 6.7.4 of \cite{LLZ14}.

   One can prove, by exactly the same method, two results refining the above, which correspond to the ``asymmetric'' norm relations of \cite[Theorem 3.5.1]{LLZ2}. The map $(\Pr^\alpha \times \Pr^\alpha)_*$ is naturally the composition of two maps
   \begin{align*}
    (\Pr^\alpha \times \id)_*&:M_{\Lp}(f \otimes g)^* \to M_{\Lp}(f_0 \otimes g)^*\\
    (\id \times \Pr^\alpha)_*&: M_{\Lp}(f_0 \otimes g)^* \to M_{\Lp}(f_0 \otimes g_0)^*,
   \end{align*}
   whose composition is $(\Pr^\alpha \times \Pr^\alpha)_*$; and we obtain the formulae
   \begin{align*}
     (\Pr^\alpha \times \id)_* \left( \Eis^{[f, g, j]}_{\et}\right) &= \left( 1 - \frac{\beta_f \alpha_g}{\ell^{1 + j}}\right) \Eis^{[f_0, g, j]}_\et, \\
    (\id \times \Pr^\alpha)_*  \left(  \Eis^{[f_0, g, j]}_\et \right) &= \left(1 - \frac{\alpha_f \beta_g}{\ell^{1 + j}}\right)\left(1 - \frac{\beta_f \beta_g}{\ell^{1 + j} }\right)\Eis^{[f_0, g_0, j]}_\et.
   \end{align*}
  \end{remark}


\section{Projection to \texorpdfstring{$Y_1(N)$}{Y1(N)} and cyclotomic twists}

 \subsection{Projection to \texorpdfstring{$Y_1(N)$}{Y1(N)}}

  Let $M, N$ be integers with $N \ge 4$, and let $\mu_M^\circ$ be the scheme of primitive $M$-th roots of unity, so that $\mu_M^\circ = \Spec \ZZ[\zeta_M]$. Then there is a canonical map
  \[ s_M: Y(M, MN) \to Y_1(N) \times \mu_M^\circ \]
  given in terms of moduli spaces by
  \[ (E, e_1, e_2) \mapsto \Big( \big(E/\langle e_1\rangle, e_2 \bmod \langle e_1 \rangle\big), \langle e_1, Ne_2\rangle_{E[M]}\Big).\]
  (This is the map denoted by $t_M$ in \cite{LLZ14}, but this notation unfortunately conflicts with the notation $t_N$ for the canonical order $N$ section inherited from \cite{Kings-Eisenstein}; so we have adopted the alternative notation $s_M$ here.) For any prime $\ell$ we have a commutative diagram
  \begin{diagram}
    Y(M\ell, MN\ell) & \rTo^{\hpr_2}&  Y(M, MN\ell) & \rTo^{\pr_1}& Y(M, MN) \\
    \dTo^{s_{M\ell}}& & & & \dTo^{s_{M}} \\
    Y_1(N) \times \mu_{M\ell}^\circ & &\rTo & & Y_1(N) \times \mu_M^\circ
   \end{diagram}
  where the bottom horizontal map is given by $\zeta \to \zeta^\ell$ on $\mu_{M\ell}^\circ$ (and is the identity on $Y_1(N)$). Moreover, the maps $s_M$ for different values of $N$ are compatible with the maps $\pr_1$.

  \begin{definition}
   Extend the map $s_M$ to a map on sheaves (which we also denote by $s_M$)
   \[
    \big(Y(M,MN),\sH_{\Zp}\langle t_{MN}\rangle \big)
    \rTo
    \big(Y_1(N)\times\mu_M^\circ, \sH_{\Zp}\langle t_N\rangle\big)
   \]
   by defining $(s_M)_\flat \in \Hom_{\mathrm{Sheaves}(Y(M,MN)_{\et})}(\sH_{\Zp}\langle t_{MN}\rangle, s_M^*\sH_{\Zp}\langle t_N\rangle)$ as the composite
   \[ \sH_{\Zp}\langle t_{MN}\rangle \rTo^{[M]_*} \sH_{\Zp} \langle M \cdot t_{MN} \rangle
   \rTo s_M^*\sH_{\Zp}\langle t_N\rangle, \]
   where the first map is the $M$-multiplication, and the latter is pushforward by the natural $M$-isogeny $\cE \to s_M^* \cE$ (which maps $M \cdot t_{MN}$ to $s_M^* t_N$).
  \end{definition}

  \begin{definition}
   The ($\Lambda$-adic) \emph{Beilinson--Flach element}
   \[ \cBF_{M, N, a}^{[j]} \in H^3_\et\left(Y_1(N)^2 \times \mu_M^\circ, \Lambda(\sH_{\Zp}\langle t_N\rangle)^{[j, j]}(2-j)\right)\]
   is defined to be the image of $\cRI^{[j]}_{M, MN, a}$ under $(s_{M} \times s_M)_*$.
  \end{definition}

  Note that for $M = 1$, $s_M$ is the identity, so we have
  \[ \cBF^{[j]}_{1, N, 1} = \cRI^{[j]}_{1, N, 1} = (\Delta_* \circ CG^{[j]})(\cEI_{1, N}).\]

 \subsection{Compatibility with cyclotomic twists}

   We now set $M = m p^r$, where $m$ is coprime to $p$ and $r \ge 1$, and we work over the $p^r$-torsion sheaves $\sH_r$ and $\Lambda_r(\sH_r\langle t \rangle)$. It is clear that $s_{M}$ induces a map on the torsion sheaves $\sH_r$.

  \begin{notation}
   We also write $s_{M}$ for the induced maps
   \[ \big(Y(M,MN),\TSym^k \sH_r\big)\rTo \big(Y_1(N)\times\mu_{M},\TSym^k\sH_r\big) \]
   and
   \[  \big(Y(M,MN), \Lambda_r(\sH_r\langle t_{p^r N}\rangle)\big)\rTo \big(Y_1(N)\times\mu_{M},\Lambda_r(\sH_r\langle t_N \rangle)\big). \]
  \end{notation}

  \begin{remark}
   Over $Y(M, M N)$ the sheaf $\sH_r$ becomes isomorphic to the constant sheaf $(\ZZ / p^r \ZZ)^{\oplus 2}$, spanned by the sections $x: (E, P, Q) \mapsto mP$ and $y: (E, P, Q) \mapsto mNQ$ (so that $y = mN  \cdot t_{MN}$). On the sections $H^0(Y(M, M N)^2,\TSym^k\sH_r\boxtimes\TSym^{k'}\sH_r)$, the map $(u_a)_* = (u_{-a})^*$ sends $x^{[i]}y^{[k-i]}\boxtimes x^{[l]}y^{[k'-l]}$ to $x^{[i]}y^{[k-i]}\boxtimes (x-ay)^{[l]}y^{[k'-l]}$.
  \end{remark}

  \begin{notation}
   To simplify the notation, in the following diagram we write $\Lambda_r$ for $\Lambda_r(\sH_r\langle t_? \rangle)$, where $?$ denotes the level of the relevant modular curve.
  \end{notation}

  \begin{theorem}\label{cyctwistdiag}
   The following diagram commutes:
   \begin{diagram}
    H^1_{\et}\big(Y(M, MN), \Lambda_r^{[0,0]}(1)\big)
    &  \rTo^{\cup (y^{[j]} \otimes y^{[j]})} & H^1_{\et}\big(Y(M,MN),\Lambda^{[j, j]}(1)\big)\\
    \dTo^{CG^{[j]}}  & & \dTo^{(-a)^j j!}\\
    H^1_{\et}\big(Y(M,MN),\Lambda_r^{[j,j]}(1 - j)\big) & & H^1_{\et}\big(Y(M,MN),\Lambda_r^{[j, j]}(1)\big)\\
    \dTo^{\Delta_*} & & \dTo^{\Delta_*} \\
     H^3_{\et}\big(Y_1(M,MN)^2,\Lambda_r^{[j,j]}(2-j)\big) & & H^3_{\et}\big(Y_1(M,MN)^2,\Lambda_r^{[j, j]}(2)\big)\\
    \dTo^{(u_a)_*} & & \dTo^{(u_a)_*} \\
     H^3_{\et}\big(Y_1(M,MN)^2,\Lambda_r^{[j, j]}(2-j)\big) & & H^3_{\et}\big(Y_1(M,MN)^2,\Lambda_r^{[j, j]}(2)\big)\\
     \dTo^{(s_{p^r})_*} & & \dTo^{(s_{p^r})_*} \\
     H^3_{\et}\big(Y_1(N)^2\times\mu_{M}^\circ,\Lambda_r^{[j, j]}(2-j)\big) & \rTo^{\cup (\zeta_{p^r})^{\otimes j}}&H^3_{\et}\big(Y_1(N)^2\times\mu_{M}^\circ,\Lambda_r^{[j, j]}(2)\big).
   \end{diagram}
  \end{theorem}

  \begin{proof}
   We start with the following observations:
   \begin{enumerate}
    \item the map $CG^{[j]}$ is defined by the cup-product with the element
    \[ \sum_{i=0}^j(-1)^i i!(j-i)!(x^{[i]}y^{[j-i]}\otimes y^{[i]}x^{[j-i]})\otimes \zeta_{p^r}^{\otimes -j} \in H^0(Y(M, MN), (\TSym^j \sH_r \otimes \TSym^j \sH_r)(-j));\]
    \item this is the pullback under $\Delta$ of the element
    \[ \sum_{i=0}^j(-1)^i i!(j-i)!(x^{[i]}y^{[j-i]}\boxtimes y^{[i]}x^{[j-i]})\otimes \zeta_{p^r}^{\otimes -j} \in H^0(Y(M, MN)^2, (\TSym^j \sH_r \boxtimes \TSym^j \sH_r)(-j)),\]
    and the cup-product satisfies the projection formula $\Delta_*(u \cup \Delta^* v) = \Delta_*(u) \cup v$;
    \item the automorphism $(u_a)_* = (u_{-a})^*$ of
    \[ H^0(Y(M, MN)^2, (\TSym^j \sH_r \boxtimes \TSym^j \sH_r)(-j))
    \]
    sends $x^{[i]}y^{[j-i]}\boxtimes y^{[i]}x^{[j-i]}$ to
    \[ x^{[i]}y^{[j-i]}\boxtimes y^{[i]}(x-ay)^{[j-i]}\]
    (and acts trivially on $\zeta_{p^r}$);
    \item $u_a$ is an automorphism, so $(u_a)_* = (u_{-a})^*$ distributes over cup products.
   \end{enumerate}

   With these preliminaries out of the way, we proceed to the proof. Let $z\in  H^1_{\et}\big(Y(M,MN),\Lambda_r^{[0,0]}(1)\big)$. Then (2) and (3) above imply that
   \[ ((u_a)_* \circ \Delta_*) \big(z\cup (y^{[j]} \otimes y^{[j]})\big) = ((u_a)_* \circ \Delta_*)(z) \cup (y^{[j]} \boxtimes y^{[j]}),\]
   so by (1) and (4) we have
   \begin{align*}
     ((u_a)_*\circ \Delta_*\circ &CG^{[j]})(z)\\
     & =\left[(u_a)_*\Delta_*(z)\right]\cup \sum_{i=0}^j \Big[(-1)^i i!(j-i)!\times x^{[i]}y^{[j-i]}\boxtimes y^{[i]}(x-ay)^{[j-i]}\Big]\otimes \zeta_{p^r}^{\otimes -j}.
   \end{align*}
   Modulo the subsheaf generated by $x \boxtimes 1$ and $1 \boxtimes x$, which is in the kernel of $(s_{M} \boxtimes s_{M})_\flat$, the only term that is nonzero is the term for $i = 0$, which is
   \[ \left[(u_a)_*\Delta_*(z)\right]\cup (-a)^j j! (y^{[j]} \boxtimes y^{[j]}) \cup \zeta_{p^r}^{\otimes -j}.\qedhere\]
  \end{proof}


 \subsection{Cyclotomic twists of the Beilinson--Flach elements}

  We shall now use the commutative diagram of the previous section to relate the Beilinson--Flach elements $\cBF_{mp^r, N, a}^{[j]}$ for general $j$ to those for $j = 0$. We first introduce the necessary maps.


  \begin{notation}
   We write $\id \otimes \mom^j$ for the map of sheaves on $Y_1(N)$ defined by
   \[ \Lambda(\sH_{\Zp}\langle t_N \rangle) \rTo^u \Lambda(\sH_{\Zp}\langle t_N \rangle) \otimes \Lambda(\sH_{\Zp}\langle t_N \rangle) \rTo^{\id \otimes \mom^j} \Lambda(\sH_{\Zp}\langle t_N \rangle) \otimes \TSym^j \sH_\Zp = \Lambda(\sH_{\Zp}\langle t_N \rangle)^{[j]},\]
   where $u$ is the map of \eqref{eq:diagonal} (induced by the diagonal embedding $\cE \into \cE \times \cE$).
  \end{notation}

  As in \S\ref{sect:iwaclebschgordan}, for any integer $k \ge j$ we also have a map $\mom_r^{k-j} \cdot \id: \Lambda_r^{[j]} \to \TSym^k \sH_r$.

  \begin{lemma}
   \label{lemma:momidentity}
   For all $0 \le j\leq k$, we have the following identity of moment maps:
   \[ (\mom^{k-j}\cdot \id)\circ (1\otimes \mom^j) =\tbinom{k}{j}\mom^k.\]
  \end{lemma}

  \begin{proof}
   We have a commutative diagram
   \begin{diagram}
     \Lambda_r &\rTo^u & \Lambda_r\otimes\Lambda_r & \rTo^{\id\otimes \mom^j_r} & \Lambda_r^{[j]}\\
     && &\rdTo^{\mom^{k-j}_r\otimes \mom^j_r} &\dTo_{\mom^{k-j}_r\otimes\id} \\
     &&& &\TSym^{k-j}\sH_r\otimes \TSym^j\sH_r  \\
   \end{diagram}
   We deduce that
   \begin{align*}
    (\mom^{k-j}_r \cdot \id) \circ (1 \otimes \mom^j_r)\circ u &=\mom^{k-j}_r \cdot \mom^j_r\\
    & = \tbinom{k}{j}\mom^k_r,
  \end{align*}
  where the last equality follows from \eqref{TSymmult}.
  \end{proof}

  \begin{notation}
   If $S$ is a $\ZZ[1/p]$-scheme, define a pro-\'etale sheaf $\Lambda_\Gamma(-\j)$ on $S$ as the inverse limit of the sheaves $(p_r)_* (\ZZ / p^r \ZZ)$ for $r \ge 1$, where $p_r$ is the map $S \times \mu_{p^r}^\circ \to S$.
  \end{notation}

  The stalk of $\Lambda_\Gamma(-\j)$ at a geometric point $\overline{x}$ is isomorphic to the Iwasawa algebra of the group $\Gamma = \Gal(\QQ(\mu_{p^\infty}) / \QQ)$, with $\Gamma$ acting via the inverse of the canonical character $\j: \Gamma \to \Lambda_{\Gamma}^\times$, hence the notation. This is a simple case of the $\Lambda$-adic sheaf theory of \cite{Kings-Eisenstein}, and it is equipped with moment maps $\Lambda_{\Gamma}(-\j) \to \Zp(-j)$ for any $j \in \ZZ$, which we write as $\mom^j_{\Gamma}$ (to distinguish them from the moment maps for the sheaves $\sH_{\Zp}$).

   Then it is well-known (see e.g. \cite[Proposition II.1.1]{colmez98}) that for any set $\Sigma$ of primes with $p \in \Sigma$, and any profinite $\Zp[G_{\QQ, \Sigma}]$-module $A$, we have a canonical isomorphism
  \[ H^1(\ZZ[1/\Sigma], A \otimes \Lambda_\Gamma(-\j)) = \varprojlim_r H^1(\ZZ[1/\Sigma, \zeta_{p^r}], A). \]


  We may sum up the computations of this section (and the preceding two) in the form of the following theorem. Let $e_{\ord}' = \lim_{n \to \infty} (U_p')^{n!}$ be the ordinary idempotent attached to $U_p'$.

  \begin{theorem}
   \label{thm:BFeltsinterp1}
   For any prime $p \ge 3$, $N \ge 4$ divisible by $p$, $m \ge 1$ coprime to $p$, and $c > 1$ coprime to $6mNp$, there is a class
   \[ \cBF_{m, N, a} \in (e_{\ord}', e_{\ord}')\, H^3\left(Y_1(N)^2 \times \mu_m^\circ, \Lambda(\sH_{\Zp}\langle t_N \rangle)^{\boxtimes 2} \otimes \Lambda_\Gamma(2-\j)\right)\]
   such that for any integers $k, k', j$ satisfying the condition \eqref{eq:inequalities} that $0 \le j \le \min(k, k')$, we have
   \begin{multline*}
    (\mom^k \boxtimes \mom^{k'} \otimes \mom^{j}_\Gamma)\left(\cBF_{m, N, a}\right) =
   \\
    \left(1 - p^j (U_p', U_p')^{-1} \sigma_p\right) \left(c^2 - c^{2j-k-k'} \sigma_c^{2} (\langle c \rangle, \langle c \rangle)\right) (e_{\ord}', e_{\ord}') (s_m \times s_m)_* \frac{(u_{a})_*\, \Eis^{[k, k', j]}_{1, mN, \et}}{ (-a)^j j! \tbinom{k}{j}\tbinom{k'}{j}}
   \end{multline*}
   where $\sigma_c$ is the arithmetic Frobenius at $c$ in $\Gal(\QQ(\mu_m) / \QQ)$.
  \end{theorem}

  Write $\cBF_{mp^r, N, a, r}^{[j]}$ for the image of this element under reduction modulo $p^r$, as an element of $H^3_\et(Y_1(N)^2 \times \mu_{mp^r}^\circ, \Lambda_r(\sH_r\langle t_N\rangle)^{[j, j]}(2-j))$.

  \begin{proof}
   The operator $(U_p', U_p')$ is invertible on the image of the idempotent $(e_{\ord}', e_{\ord}')$, and the classes
   \[ (U_p', U_p')^{-r} (e_{\ord}', e_{\ord}') \cBF_{mp^r, N, a}^{[0]}\]
   are compatible under corestriction for $r \ge 1$, by Theorem \ref{thm:norm2}. They thus define an element of the inverse limit
   \[ \varprojlim_r H^3\left(Y_1(N)^2 \times \mu_{mp^r}^\circ, \Lambda(\sH_{\Zp}\langle t_N \rangle)^{\boxtimes 2}(2)\right) = H^3\left(Y_1(N)^2 \times \mu_m^\circ, \Lambda(\sH_{\Zp}\langle t_N \rangle)^{\boxtimes 2} \otimes \Lambda_{\Gamma}(2-\j)\right),\]
   and we define $\cBF_{m, N, a}$ to be this class.

   For any $j \ge 0$, the maps $\mom^j_{\Gamma}: \Lambda_{\Gamma}(-\j) \to \Zp(-j)$ and \( \id \otimes \mom^j: \Lambda(\sH_{\Zp}\langle t_N \rangle) \to \Lambda(\sH_{\Zp}\langle t_N \rangle)^{[j]}\)
   combine into a map of sheaves $\Lambda(\sH_{\Zp}\langle t_N \rangle)^{\boxtimes 2} \times \Lambda_{\Gamma}(-\j) \to \Lambda(\sH_{\Zp}\langle t_N \rangle)^{[j, j]}(-j)$. We claim that this map sends $(-a)^jj! \cBF_{m, N, a}$ to the element
   \[\left(1 - p^j (U_p', U_p')^{-1} \sigma_p\right) (e_{\ord}', e_{\ord}') \cBF_{m, N, a}^{[j]}.\]
   Unwinding the definition of the moment maps, the image of $(-a)^j j! \cBF_{m, N, a}$ is given by the limit of the inverse system
   \[(-a)^j j! \norm_{m}^{mp^r} \left[ (U_p', U_p')^{-r} (e_{\ord}', e_{\ord}')  (\id\otimes \mom^j_r)^{\boxtimes 2}\left(\cBF_{mp^r, N, a, r}^{[0]}\right) \cup \zeta_{p^r}^{\otimes (-j)} \right] \tag{$\star$}\]
   over integers $r \ge 1$, where $\cBF^{[j]}_{mp^r, N, a, r}$ denotes the mod $p^r$ reduction of the Beilinson--Flach class. We have seen in Theorem \ref{cyctwistdiag} that
   \[ (-a)^j j!\,(\id\otimes \mom^j_r)^{\boxtimes 2}\left(\cBF_{mp^r, N, a,r}^{[0]}\right) = \cBF_{mp^r, N, a, r}^{[j]}\otimes \zeta_{p^r}^{\otimes j}.\]
   So $(\star)$ is the mod $p^r$ reduction of the element
   \[ \norm_{m}^{mp^r}\left[ (U_p', U_p')^{-r} (e_{\ord}', e_{\ord}')\cBF_{m p^r, N, a}^{[j]}\right],\]
   which is independent of $r \ge 1$ by the $\ell \mid M$ case of Theorem \ref{thm:norm2}. We conclude that $(-a)^jj! \cBF_{m, N, a}$ maps to
   \[ \norm_m^{mp} \left[(U_p', U_p')^{-1} (e_{\ord}', e_{\ord}') \cBF_{mp, N, a}^{[j]}\right].\]
   But this is just $(U_p', U_p')^{-1} (e_{\ord}', e_{\ord}') \left( (U_p', U_p') - p^j \sigma_p\right) \cBF_{m, N, a}^{[j]}$ by the $\ell \nmid M$ case of Theorem \ref{thm:norm2}, completing the proof of the claim.

   Taking $(\mom^{k-j} \cdot \id) \boxtimes (\mom^{k'-j} \cdot \id)$ of both sides, and using Lemma \ref{lemma:momidentity}, we conclude that
   \begin{multline*} (\mom^k \boxtimes \mom^{k'} \otimes \mom^j_{\Gamma})
   \left(\cBF_{m, N, a}\right) \\= \frac{\left(1 - p^j (U_p', U_p')^{-1} \sigma_p\right)}{ (-a)^j j! \tbinom{k}{j}\tbinom{k'}{j} }  (e_{\ord}', e_{\ord}') \left[(\mom^{k-j} \cdot \id) \boxtimes (\mom^{k'-j} \cdot \id)\right] (\cBF_{m, N, a}^{[j]}).
   \end{multline*}

   However, we know that
   \[ \left[(\mom^{k-j} \cdot \id) \boxtimes (\mom^{k'-j} \cdot \id)\right] (\cBF_{m, N, a}^{[j]}) = \left(c^2 - c^{2j-k-k'} \sigma_c^{2} (\langle c \rangle, \langle c \rangle)\right)\circ (s_m \times s_m)_* (u_a)_*\left(\Eis^{[k, k', j]}_{1, mN, \et}\right), \]
   by Proposition \ref{prop:RIfirstproperties}(\ref{item:RIfirstproperties-moments}),which completes the proof of the theorem.
  \end{proof}


\section{Hida theory: background}
 \label{sect:hidatheory}


 In this short section (which contains no substantial original results) we shall recall the fundamental theorems on ordinary $p$-adic families of modular forms, due to Hida, Wiles, and Ohta.

 \subsection{Notation}

  In this section, we shall need to consider numerous modules over the ring $\Lambda = \Zp[[\ZZ_p^\times]]$. It will be convenient to use the notation $x^{\k}$, for $x \in \ZZ_p^\times$, for the image of $x$ under the canonical character $\ZZ_p^\times \to \Zp[[\ZZ_p^\times]]$. We shall interpret $\k$ as a ``coordinate'' on the weight space $\operatorname{Spf} \Lambda$ parametrizing characters of $\ZZ_p^\times$, and we shall speak of ``specialising at $\k = \tau$'' for a character $\tau$ to refer to the homomorphism $\Lambda_D \to \overline{\QQ}_p$ given by evaluation at $\tau$.

  We interpret both the integers, and the Dirichlet characters of $p$-power conductor, as subsets of the characters of $\ZZ_p^\times$; and we shall write characters additively, so by ``specialising at $\k = 3 + \eta$'' (where $\eta$ is a Dirichlet character) we mean the homomorphism $z \mapsto z^3 \eta(z)$.

  Since we are interested in Rankin convolutions of pairs of modular forms, and we also have a character twist, we will work over the ring $\Lambda \htimes_{\Zp} \Lambda \htimes_{\Zp} \Lambda$. We shall write $\k$, $\k'$, and $\j$ for the canonical characters of each factor, with $\k$ and $\k'$ reserved for weights of families of modular forms, and $\j$ for cyclotomic twists. Although they are isomorphic algebras, we distinguish between them by writing the first two factors as $\Lambda_D$, the $D$ signifying diamond operators, and the third as $\Lambda_{\Gamma}$, where $\Gamma$ signifies the cyclotomic Galois group $\Gal(\QQ(\mu_{p^\infty}) / \QQ)$.

  We shall fix an algebraic closure $\overline{\QQ}$ of $\QQ$, and embeddings $\overline{\QQ} \into \CC$ and $\overline{\QQ} \into \overline{\QQ}_p$. For $m \ge 1$, let $\zeta_m \in \overline{\QQ}$ be the primitive root of unity corresponding to $e^{2\pi i / m} \in \CC$.

 \subsection{Hida families and associated Galois modules}

  We recall the following result (due to Ohta, \cite{Ohta-ordinary, Ohta-ordinaryII}) describing the ordinary parts of inverse limits of cohomology groups. We assume, for the remainder of this paper, that $p \ge 5$.

  \begin{proposition}
   \label{prop:controlthm}
   Let $N$ be coprime to $p$. Then:
   \begin{enumerate}
    \item The module
    \begin{align*}
     H^1_{\ord}(Np^\infty) \coloneqq&\,
     e_{\ord}' \cdot \varprojlim_r  H^1_{\et}\left(Y_1(Np^r)_{\overline{\QQ}},\,\Zp(1)\right)\\
     \cong&\, e_{\ord}' \cdot H^1_{\et}\left(Y_1(Np)_{\overline{\QQ}},\, \Lambda(\sH_{\Zp}\langle t_{Np}\rangle)(1)\right)\qquad \text{\textrm{(cf.\ Theorem \ref{thm:Ohta-twisting})}}
    \end{align*}
    is finitely-generated and projective over the algebra $\Lambda_D = \Zp[[\ZZ_p^\times]]$ (acting via the inverse diamond operators at $p$, so $u \in \ZZ_p^\times$ acts on the $r$-th layer in the inverse limit as the diamond operator $\langle u^{-1} \rangle_{p^r}$).

    \item The module $H^1_{\ord}(Np^\infty)$ has $\Lambda_D$-linear actions of the group $G_{\QQ, S}$, where $S$ is the set of primes dividing $Np$, and of the Hecke operators $T_n'$ for $n \ge 1$; and these actions commute with each other.

    \item The module $H^1_{\ord}(Np^\infty)$ has $\Lambda_D$-linear actions of the operators $W_Q$ for integers $Q \parallel N$ (which do \emph{not} commute with either the Hecke operators or the Galois action).

    \item We have the following ``perfect control'' theorem: if $k \ge 0$, $r \ge 1$ and $I_{k, r}$ is the ideal of $\Lambda_D$ generated by $[1 + p^r] - (1 + p^r)^k$, then the moment map $\mom^k$ induces an isomorphism of $\Zp$-modules
    \[ H^1_{\ord}(Np^\infty) / I_{k, r} \cong e_{\ord}' H^1_{\et}\left(Y_1(Np^r)_{\overline{\QQ}}, \TSym^k(\sH_{\Zp})(1)\right) \]
    compatible with the actions of $G_{\QQ, S}$ and the operators $T_n'$ and $W_Q$. It is an isomorphism of $\Lambda_D$-modules if we let $u \in \ZZ_p^\times$ act on the right-hand side as $u^k \langle u^{-1} \rangle_{p^r}$.
   \end{enumerate}
   Similar statements hold for the submodule $H^1_{\ord, \Par}(Np^\infty)$ defined with $X_1(Np^r)$ in place of $Y_1(Np^r)$.
  \end{proposition}

  \begin{remark}
   The action of Atkin--Lehner operators away from $p$ is not explicitly mentioned in Ohta's work, but it is easy to see that $W_Q$ is compatible with the maps in the inverse limit, and that it commutes with the moment maps $\mom^k$. For part (4) of the theorem, we use the fact that Ohta's ``specialisation'' maps $\mathrm{sp}_{r, k}$ coincide with our moment maps $\mom^k$, which is the content of Theorem \ref{thm:Ohta-twisting}(4).
  \end{remark}

  We write $\TT_{Np^\infty}$ for the Hecke algebra acting on $H^1_{\ord}(Np^\infty)$ (generated by the operators $T_n'$ for all $n \ge 1$), and $\TT_{Np^\infty, \Par}$ for the quotient acting on the parabolic cohomology $H^1_{\ord, \Par}(Np^\infty)$. These are finite projective $\Lambda_D$-algebras.

  \begin{theorem}[Ohta, \cite{Ohta-ordinaryII}]
   \label{thm:ordinaryfiltration}
   There are short exact sequences of $\TT_{Np^\infty}[G_{\Qp}]$-modules
   \begin{diagram}
    0 &\rTo& \sF^+ H^1_{\ord, \Par}(Np^\infty) &\rTo&  H^1_{\ord, \Par}(Np^\infty) &\rTo&  \sF^- H^1_{\ord, \Par}(Np^\infty) &\rTo& 0\\
    & & \dEq & & \dInto & & \dInto \\
    0 &\rTo& \sF^+ H^1_{\ord}(Np^\infty) &\rTo&  H^1_{\ord}(Np^\infty) &\rTo&  \sF^- H^1_{\ord}(Np^\infty) &\rTo& 0\\
   \end{diagram}
   with the following properties:
   \begin{enumerate}[(i)]

    \item All the modules $\sF^{\pm} H^1_{\ord}(Np^\infty)$ and $\sF^{\pm} H^1_{\ord, \Par}(Np^\infty)$ are projective of finite rank over $\Lambda_D$.

    \item There is a non-canonical isomorphism of $\TT_{Np^\infty}$-modules
    \[ \sF^- H^1_{\ord, \Par}(Np^\infty) \cong \Hom_{\Lambda_D}(\TT_{Np^\infty, \Par}, \Lambda_D).\]

    \item The quotient $\sF^- H^1_{\ord}(Np^\infty)$ is unramified as a $G_{\Qp}$-module, with arithmetic Frobenius acting via the Hecke operator $U_p' \in \TT_{Np^\infty}$.

    \item The group $G_{\Qp}$ acts on the submodule $\sF^+ H^1_{\ord}(Np^\infty)$ via the $\TT_{Np^\infty}$-valued character given by the product of the unramified character mapping arithmetic Frobenius to $\langle p^{-1} \rangle_N \cdot (U_p')^{-1}$ with the $(1 + \k)$-th power of the $p$-adic cyclotomic character.

    \item There is a canonical perfect pairing of $\Lambda_D$-modules
    \[  H^1_{\ord, \Par}(Np^\infty) \times  H^1_{\ord, \Par}(Np^\infty) \to \Lambda_D \]
    with respect to which the $\TT_{Np^\infty}$-action is selfadjoint, and the modules $\sF^{\pm}$ are orthogonal complements. Hence $\sF^+ H^1_{\ord, \Par}(Np^\infty) = \sF^+ H^1_{\ord}(Np^\infty)$ is free of rank 1 as a $\TT_{Np^\infty, \Par}$-module.
   \end{enumerate}
  \end{theorem}

  \begin{proof}
   This is equivalent to  Corollary 1.3.8 and Corollary 2.3.6 of \cite{Ohta-ordinaryII}; our modules $\sF^{+} H^1_{\ord}(Np^\infty)$, $\sF^{-} H^1_{\ord,\Par}(Np^\infty)$ and $\sF^{-} H^1_{\ord}(Np^\infty)$ are the modules $\mathfrak{A}_\infty^*$, $\mathfrak{B}_\infty^*$ and $\widetilde{\mathfrak{B}}_\infty^*$ in Ohta's notation.

   However, the Galois actions above are somewhat different, as we use a different model of $Y_1(N)$ from Ohta: Ohta uses the notation $Y_1(N)$ for the modular curve classifying elliptic curves with an embedding of $\mu_N$, rather than a point of order $N$. We also take cohomology with $\Zp(1)$ coefficients, rather than $\Zp$. Thus our modules coincide (as $\Lambda_D[G_{\QQ}]$-modules) with Ohta's modules twisted by the character $1 + \k + \underline{\varepsilon}_N$, where $\underline{\varepsilon}_N$ is the composite of the mod $N$ cyclotomic character with the map $d \mapsto \langle d^{-1} \rangle_N \in \TT_{Np^\infty}$. This gives the statements above. (Compare \S 1.7.16 of \cite{FukayaKato-sharifi-conj}.)
  \end{proof}

  The Hecke algebra $\mathbf{T}_{Np^\infty, \Par}$ is a finite projective $\Lambda_D$-algebra, and is thus isomorphic to the direct product of its localisations at its finitely many maximal ideals. We refer to these maximal ideals as \emph{Hida families}.

  \begin{remark}
   There are multiple conventions in the literature as to what exactly is meant by ``Hida family''. Our conventions follow those of \cite{emertonpollackweston06}, for example; but other authors use the term ``family'' for what we call a ``branch'' (see below).
  \end{remark}

  \begin{definition}
   If $\bff$ is a Hida family, we define
   \[ M(\bff)^* = H^1_{\ord}(Np^\infty)_{\bff},\]
   and we write $\Lambda_{\bff}$ for the corresponding localisation of the Hecke algebra $\mathbf{T}_{Np^\infty}$, which is a local $\Lambda_D$-algebra, finite and projective as a $\Lambda_D$-module. We write $M(\bff)^*_{\Par}$ for the image of $H^1_{\ord, \Par}(Np^\infty)$ in $M(\bff)^*$, and $\Lambda_{\bff, \Par}$ for the corresponding quotient of $\Lambda_{\bff}$.
  \end{definition}

  \begin{definition}
   \label{def:nonEis}
   We say that $\bff$ is \emph{non-Eisenstein modulo $p$} if the residual Galois representation $\rho_{\bff}: G_{\QQ, S} \to \GL_2(\FF)$ associated to $\bff$ (where $\FF$ is the residue field of $\Lambda_{\bff}$) is irreducible. If $\bff$ is non-Eisenstein, then $M(\bff)^*_{\Par}= M(\bff)^*$.

   We say $\bff$ is \emph{$p$-distinguished} if the semisimplification of $\overline{\rho}_{\bff} |_{G_{\Qp}}$ is the direct sum of two \emph{distinct} characters.
  \end{definition}

  \begin{remark}
   Note that $\bff$ is automatically $p$-distinguished if the weight of $\bff$ mod $p-1$ (i.e.\ the mod $p-1$ congruence class of the weight of any classical specialisation of $\bff$ of level prime to $p$) is not congruent to 1, by the same argument as in \cite[Proposition 4.3.6]{LLZ2}.
  \end{remark}

  \begin{theorem}[Wiles]
   \label{thm:gorenstein}
   If $\bff$ is non-Eisenstein mod $p$, and $p$-distinguished, then $\Lambda_{\bff}$ is a Gorenstein ring, and the modules $M(\bff)^*$ and $\sF^\pm M(\bff)^*$ are free over $\Lambda_{\bff}$.
  \end{theorem}

  \begin{proof}
   This follows, using the control theorem, from Wiles' results on the freeness of Hecke modules at level $Np$ \cite{Wiles-FLT}; cf.\ \cite[Theorem 4.3.4]{LLZ2} or \cite[Proposition 3.3.1]{emertonpollackweston06}.
  \end{proof}

 \subsection{Specialisations of Hida families}
  \label{sect:hidaspec}

  We define an \emph{arithmetic prime} of $\Lambda_D$ or $\Lambda_{\bff}$ to be a prime ideal of height 1 lying over the ideal $I_{k, r}$, for some $k \ge 0$ and $r \ge 1$. From the control theorem, we see that each arithmetic prime of $\Lambda_{\bff}$ above $I_{k, r}$ corresponds to an eigenform $f$ of weight $k + 2 \ge 2$ and level $Np^r$ (together with a choice of prime $\frP \mid p$ of the coefficient field $L = \QQ(f)$ at which $f$ is ordinary). In this setting we say that $f$ is a \emph{specialiation} of the family $\bff$.

  For each specialisation $f$ of $\bff$, we have a specialisation-at-$f$ isomorphism
  \[ \mathrm{sp}_f: M(\bff)^* \otimes_{\Lambda_{\bff}} \cO_{L, \frP} \rTo^\cong M_{\cO_{L, \frP}}(f)^*, \]
  which fits into a commutative diagram
  \begin{diagram}[small]
   H^1_{\ord}(N p^\infty)  & \rTo & M(\bff)^* \\
   \dTo<{\mom^k} & & \dTo>{\mathrm{sp}_f} \\
   H^1_\et\left(Y_1(Np^r)_{\overline\QQ}, \TSym^k \sH_{\Zp}(1)\right) & \rTo & M_{\cO_{L, \frP}}(f)^*.
  \end{diagram}

  We shall be primarily interested in the case where $r = 1$ and $f$ is the ordinary $p$-stabilisation of an ordinary newform $f_0$ of level $N$, with the $U_p$-eigenvalue of $f$ being the unit root $\alpha$ of the Hecke polynomial of $f_0$ at $p$. In this case, we may identify $M_{\Lp}(f)^*$ with $M_{\Lp}(f_0)^*$ via the map
  \[
   (\Pr^\alpha)_*:
   \cO_{L, \frP} \otimes_{\Zp} H^1\left(Y_1(Np)_{\overline\QQ}, \TSym^k \sH_{\Zp}(1)\right)
   \to \cO_{L, \frP} \otimes_{\Zp} H^1\left(Y_1(N)_{\overline\QQ}, \TSym^k \sH_{\Zp}(1)\right)
  \]
  given by $(\Pr^\alpha)_* \coloneqq (\pr_1)_* - \frac{\beta}{p^{k + 1}} (\pr_2)_*$, where $\beta = a_p(f) - \alpha$ is the non-unit root of the Hecke polynomial. This map clearly sends the integral lattice $M_{\cO_{L, \frP}}(f)^*$ to a sublattice of finite index in $M_{\cO_{L, \frP}}(f_0)^*$.

  \begin{proposition}
   \label{prop:pstablattice}
   Suppose at least one of the following conditions holds:
   \begin{itemize}
    \item $k > 0$;
    \item the family $\bff$ is non-Eisenstein modulo $p$;
    \item $\frac{\beta}{p} \ne \alpha \bmod \frP$.
   \end{itemize}
   Then the map $(\Pr^\alpha)_*$ is an isomorphism of $\cO_{L, \frP}$-modules
   \[ M_{\cO_{L, \frP}}(f)^* \rTo^\cong M_{\cO_{L, \frP}}(f_0)^*. \]
  \end{proposition}

  \begin{proof}
   This was proved in \cite[Proposition 4.3.6]{LLZ2} assuming $k = 0$ and $\bff$ non-Eisenstein mod $p$ (using Ihara's lemma). We give an alternative argument that covers the remaining cases.

   We consider the map
   \[ M_{\cO_{L, \frP}}(f_0)^* \to M_{\cO_{L, \frP}}(f_0)^*\]
   given by the composition of $(\Pr^\alpha)_*$ with the degeneracy map $\pr_2^*$. The composition $(\pr_1)_* \circ (\pr_2)^*$ is the Hecke operator $T_p'$, and the composition $(\pr_2)_* \circ (\pr_2)^*$ acts as multiplication by $p^k(p + 1)$. So the composition of these two maps is given by multiplication by $\alpha - \frac{\beta}{p}$. Our hypotheses imply that this is a $p$-adic unit. Hence this composition is an isomorphism, so $(\Pr^\alpha)_*$ must be surjective.
  \end{proof}

 \subsection{Lambda-adic modular forms}

  Recall from \S \ref{sect:modformnotation} above that, for any $N \ge 1$, we write $M_{k + 2}(N, \Qp)$ for the space of modular forms with $q$-expansions in $\Qp[[q]]$.

  Let $M_{k+2}'(N, \Qp)$ denote the modular forms of weight $k+2$ which are defined over $\Qp$ as classes in the de Rham cohomology of $Y_1(N)$, so that we have
  \[ M_{k + 2}'(N, \Qp) = \Fil^0 H^1_{\dR}(Y_1(N)_\Qp, \TSym^k \sH_{\dR}).\]

  Then there is a canonical isomorphism
  \[ \QQ(\mu_N) \otimes_{\QQ} M_{k + 2}'(N, \Qp) = \QQ(\mu_N) \otimes_{\QQ} M_{k + 2}(N, \Qp),\]
  as we saw in  \S \ref{sect:modformnotation}; but this does not generally descend to $\QQ$. Rather, the spaces $M_{k + 2}(N, \Qp)$ and $M_{k + 2}'(N, \Qp)$ are two distinct $\Qp$-subspaces of $\QQ(\mu_N) \otimes_{\QQ} M_{k + 2}(N, \Qp)$, which are interchanged by the Atkin--Lehner operator $W_N$.

  Now suppose $N$ is coprime to $p$, and $r \ge 1$. For any $k \ge -1$, we define (following \cite[Definition 2.2.2]{Ohta-ordinary})
  \begin{align*}
   M_{k+2}'(Np^r, \Zp) &= \left\{ f \in M_{k+2}'(Np^r, \Qp) : a_n\left(W_{Np^r}^{-1} f\right) \in \Zp\ \forall n \ge 0\right\}, \\
   \mathfrak{M}_{k+2}'(N, \Zp) &= \varprojlim_{r \ge 1}  M_{k+2}'(Np^r, \Zp).
  \end{align*}
  Here the inverse limit is with respect to the pushforward maps $(\pr_1)_*$.

  \begin{remark}
   Ohta's normalisations are slightly different from ours, but Ohta's operator $\tau_r$ is our $W_{Np^r}^{-1}$ up to signs, so the above definition is equivalent to Ohta's.
  \end{remark}

  \begin{lemma}[Ohta]
   \label{lemma:lambdaadicmf}
   Let $k \ge 0$ and let $(f_r)_{r \ge 1} \in \mathfrak{M}_{k+2}'(N, \Zp)$. Then there is a unique power series $\cF \in \Lambda_D[[q]]$ whose specialisation at $\k = k + \varepsilon$, for every finite-order character $\varepsilon$, is equal to the modular form
   \[ \sum_{\alpha \in (\ZZ / p^r \ZZ)^\times} \varepsilon(\alpha)^{-1} \cdot \left(\langle \alpha \rangle_p \circ U_p^r \circ W_{Np^r}^{-1}\right)(f_r) \]
   for any $r \ge 1$ such that $\varepsilon$ is trivial on $1 + p^r \Zp$. Moreover, the image of the resulting map
   \[ \cW_{k} : e'_{\ord} \mathfrak{M}_{k + 2}'(N, \Zp) \into \Lambda_D[[q]] \]
  is independent of $k$.
  \end{lemma}

  \begin{proof} See \cite[Theorem 2.2.3, Theorem 2.4.5]{Ohta-ordinary}. \end{proof}

  The image of $\cW_{k}$ is called the \emph{module of $\Lambda$-adic modular forms} (of tame level $N$), and we shall denote it by $M^{\ord}_{\k + 2}(N, \Lambda_D)$. (It can be interpreted as the space of ordinary Katz $p$-adic modular forms with coefficients in $\Lambda_D$ and weight $\k + 2$, where $\k$ is the universal character.) For each arithmetic prime ideal $\nu$ of $\Lambda_D$, with $\ZZ_p^\times$ acting on $\cO = \Lambda_D / \nu$ via a character $z \mapsto z^k \omega(z)$ with $k \ge 0$ and $\omega$ of finite order, the natural map $\Lambda_D[[q]] \to \cO[[q]]$ induces an isomorphism
  \[ M_{\k + 2}^{\ord}(N, \Lambda_D) \otimes_{\Lambda_D} \cO \cong e_{\ord} M_{k + 2}(Np^r, \cO)[\omega] \]
  where $(\dots)[\omega]$ indicates the $\omega$-eigenspace for the diamond operators. (Here $r$ is the smallest integer $\ge 1$ ssch that $\omega$ is trivial on $1 + p^r \Zp$.) Similar statements hold \emph{mutatis mutandis} for the space $S_{\k + 2}^{\ord}(N, \Lambda_D)$ of $\Lambda$-adic cusp forms.

 \subsection{Branches}

  For each Hida family $\bff$, the algebra $\Lambda_{\bff}$ has finitely many minimal primes, and we call these the \emph{branches} of the Hida family $\bff$. They biject with the simple direct summands of the Artinian ring $\Lambda_{\bff} \otimes_{\Lambda_D} \Frac \Lambda_D$. If $\bfa$ is a branch of $\bff$, then $\Lambda_\bff / \bfa$ is an integral domain, and its field of fractions is a finite extension of $\Frac \Zp[[1 + p \Zp]]$. We let $\Lambda_{\bfa}$ be the normalisation of $\Lambda_\bff / \bfa$ (the integral closure of $\Lambda_\bff / \bfa$ in its field of fractions), which is a normal integral domain, finite and projective as a module over $\Lambda_D$ \cite[Lemma 3.1]{Hida-RS-II}.

  \begin{definition}
   We say the branch $\bfa$ is \emph{cuspidal} if the natural map $M(\bff)^*_{\Par} \to M(\bff)^*$ becomes an isomorphism after tensoring with $\Lambda_{\bfa}$.
  \end{definition}

  A branch $\bfa$ is cuspidal if one, or equivalently every, arithmetic prime of $\Lambda_{\bff}$ above $\bfa$ corresponds to a cuspidal modular form. If the family $\bff$ is non-Eisenstein mod $p$, then every branch of $\mathbf{f}$ is cuspidal (and the above map is even an isomorphism with $\Lambda_{\bff}$-coefficients). Associated to each cuspidal branch $\bfa$ we have a $\Lambda$-adic eigenform
  \[ \mathcal{F} = \sum_{n \ge 1} T_n q^n \in S_{\k + 2}^{\ord}(N, \Lambda_D) \otimes_{\Lambda_D} \Lambda_\bfa. \]

  We say $\bfa$ is \emph{new} if it is cuspidal and one, or equivalently every, arithmetic prime of $\Lambda_{\bff}$ above $\bfa$ corresponds to an eigenform which is new away from $p$.

 \subsection{Specialisations in weight 1}

  If $\mathcal{F} \in S^{\ord}_{\k + 2}(N, \Lambda_D)$, then it is not necessarily the case that the specialization of $\mathcal{F}$ at $\k = -1$ is a weight 1 modular form. Nonetheless, one has the following fact:

  \begin{theorem}[{\cite[Theorem 3]{Wiles-ordinary}}]
   Let $g_0 \in S_1(N, \varepsilon)$ be a normalised newform of level prime to $p$, and let $g \in S_1^{\ord}(Np, \varepsilon)$ be a $p$-stabilisation of $g_0$. Then there is a $\Lambda$-adic eigenform $\mathcal{G}$ (with coefficients in some finite integral extension of $\Lambda_D$) whose specialisation in weight 1 is $g$.
  \end{theorem}

  We shall assume in the applications below that the roots of the Hecke polynomial $X^2 - a_p(g_0) X + \varepsilon_f(p)$ are distinct, and remain distinct modulo $\frP$. In this case, the form $g_0$ has two distinct ordinary $p$-stabilisations $g_\alpha$, $g_\beta$, and they belong to different Hida families, both of which are $p$-distinguished. Let $\bfg$ be the Hida family corresponding to $g_\alpha$.

  If $\bfg$ is also non-Eisenstein modulo $p$, then we conclude from Theorem \ref{thm:gorenstein} that $M(\bfg)^*$ is free of rank 2 over $\Lambda_{\bfg}$. Hence the space
  \[ M_{\cO_{L, \frP}}(g_\alpha)^* \coloneqq M(\bfg)^* \otimes_{\Lambda_{\bfg}} \cO_{L, \frP}\]
  is free of rank 2 over $\cO_{L, \frP}$. This gives a canonical realisation of the 2-dimensional odd Artin representation associated to $g_0$.

 \subsection{Congruence ideals and 3-variable Rankin L-functions}

  \begin{notation}
   If $\bff$ is a Hida family, and $\bfa$ is a new, cuspidal branch of $\bff$, we let $I_{\bfa}$ be the congruence ideal of $\bfa$, which is a fractional $\Lambda_{\bfa}$-ideal $I_\bfa \subset \Frac \Lambda_{\bfa}$ characterised by the existence of a $\Lambda_{\bfa}$-linear, Hecke-equivariant map
   \[ M_{\k + 2}^{\ord}(N, \Lambda_{\bfa}) \to I_\bfa \]
   mapping the normalised $\Lambda$-adic eigenform associated to $\bfa$ to 1.
  \end{notation}

  See \cite[\S 4]{Hida-RS-II} for further details. It follows from Theorem 4.2 of \emph{op.cit.} that the fractional ideal $I_{\bfa}$ is contained in the localisation of $\Lambda_{\bfa}$ at any arithmetic prime ideal, so elements of $I_{\bfa}$ define meromorphic functions on $\Spec \Lambda_{\bfa}$ which are regular at any arithmetic point.

  \begin{theorem}
   \label{thm:hida2}
   Let $\bff$, $\bfg$ be two Hida families (of some tame levels $N_f$, $N_g$), and let $\bfa$ be a new, cuspidal branch of $\bff$. Then there is an element
   \[ L_p(\bfa, \bfg) \in \left(I_{\bfa} \htimes \Lambda_{\bfg} \htimes \Lambda_{\Gamma}\right) \otimes_{\ZZ} \ZZ[\mu_N] \]
   where $N$ is the lowest common multiple of $N_f$ and $N_g$, with the following interpolation property: if $f$, $g$ are specialisations of $\bff$, $\bfg$ respectively, with $f$ lying on the branch $\bfa$, and $f$, $g$ are $p$-stabilisations of eigenforms $f_0, g_0$ of levels $N_f, N_g$, then the image of $L_p(\bff, \bfg)$ under $\operatorname{sp}_f \otimes \operatorname{sp}_g$ is the $p$-adic $L$-function $L_p(f_0, g_0)$ of Theorem \ref{thm:hida}.
  \end{theorem}

  It is this 3-variable $p$-adic $L$-function which will appear in our explicit reciprocity law (Theorem \ref{lthm:explicitrecip}). The interpolating property extends also to weight 1 specialisations of $\bfg$, as long as the specialised eigenform is classical.

\section{Euler systems in Hida families}

 We now use the theory developed in the earlier parts of this paper to construct cohomology classes in the tensor product of the Galois representations attached to two Hida families.

 \subsection{Beilinson--Flach classes in Hida families}

  Let $\bff, \bfg$ be Hida families of tame levels $N_f, N_g$. We write
  \[ M(\bff \otimes \bfg)^* =  M(\bff)^* \htimes_{\Zp} M(\bfg)^* \]
  which is a $\Lambda_{\bff} \htimes \Lambda_{\bfg}$-module, finite and projective over $\Lambda_D \htimes \Lambda_D$. Let $N$ be any integer coprime to $p$ and divisible by $N_f$ and $N_g$; then, via pusforward along the degeneracy map
  \[ (\pr_1 \times \pr_1)_*: Y_1(Np)^2 \to Y_1(N_f p) \times Y_1(N_g p) \]
  and the K\"unneth formula, we obtain a Galois-equivariant projection map
  \[ \pr_{\bff, \bfg}: H^2_\et\left( Y_1(Np)^2_{\overline\QQ}, \Lambda(\sH_{\Zp}\langle t_{Np}\rangle)^{\boxtimes 2}(2)\right) \to M(\bff \otimes \bfg)^*. \]
  This is the Hida-family analogue of the map $\pr_{f, g}$ of \eqref{eq:prfg} above, and it is compatible with these maps for the specialisations of $\bff$ and $\bfg$: for all specialisations $f$ of $\bff$ and $g$ of $\bfg$, we have a commutative diagram
  \begin{diagram}
   H^2_\et\left( Y_1(Np)^2_{\overline\QQ}, \Lambda(\sH_{\Zp}\langle t_{Np}\rangle)^{\boxtimes 2}(2)\right) & \rTo^{\pr_{\bff, \bfg}} & M(\bff \otimes \bfg)^* \\
   \dTo<{\mom^k \boxtimes \mom^{k'}} & & \dTo>{\mathrm{sp}_f \otimes \mathrm{sp}_g} \\
   H^2_\et\left( Y_1(Np^r)^2_{\overline\QQ}, \TSym^{[k, k']}(\sH_{\Zp})(2)\right) & \rTo^{\pr_{f,g}} & M_{\Lp}(f \otimes g)^*
  \end{diagram}
  for any $r$ large enough that $f$ and $g$ have level $Np^r$.

  \begin{definition}
   For $\bff$, $\bfg$ Hida families of tame levels $N_f, N_g$, $m\ge 1$ coprime to $p$, and $c > 1$ coprime to $6 m N_f N_g p$, we define
   \[ \cBF^{\bff, \bfg}_m \in H^1\left(\ZZ\left[\tfrac 1{m p N_f N_g}, \mu_m\right], M(\bff \otimes \bfg)^* \otimes \Lambda_\Gamma(-\j)\right)\]
   to be the image of the class $\cBF_{m, p N, 1}$ of Theorem \ref{thm:BFeltsinterp1} above, for any $N \ge 1$ divisible by $N_f, N_g$ and with the same prime factors as $N_f N_g$, under the map $\pr_{\bff, \bfg}$. We write $\cBF^{\bff, \bfg}$ for $\cBF^{\bff, \bfg}_1$.
  \end{definition}

  \begin{remark}
   The class $\cBF^{\bff, \bfg}_m$ is the image under the map $Y_1(pN)^2 \to Y_1(p N_f) \times Y_1(p N_g)$ of the class denoted by ${}_c \mathfrak{z}^{\bff, \bfg, N}_m$ in \cite[Theorem 6.9.5]{LLZ14}. It is independent of the choice of $N$, because the classes $\cBF_{m, p N, 1}$ for different choices of $N$ are compatible under pushforward via $(\pr_1 \times \pr_1)_*$ (by the first case of Theorem \ref{thm:norm1}).
  \end{remark}

  \begin{theorem}[{Theorem \ref{lthm:iwasawaelt}}]
   \label{thm:BFeltsinterp2} \
   \begin{enumerate}[(a)]
    \item For $\ell$ a prime not dividing $c m N_f N_g p$, we have the norm-compatibility relation
    \[ \norm^{\ell m}_m\left( \cBF^{\bff, \bfg}_{\ell m} \right) = Q_\ell(\ell^{-\j} \sigma_\ell^{-1}) \cdot \cBF^{\bff, \bfg}_{m}, \]
    where $\sigma_\ell \in \Gal(\QQ(\mu_m) / \QQ)$ is the arithmetic Frobenius, and $Q_\ell \in (\Lambda_{\bff} \htimes \Lambda_{\bfg})[X, X^{-1}]$ is given by
    \begin{multline*}
     -X^{-1} + a_\ell(\bff) a_\ell(\bfg) + \left( (\ell + 1) \ell^{\k + \k'} \varepsilon_{\bff}(\ell)\varepsilon_{\bfg}(\ell) - \ell^{\k} \varepsilon_{\bff}(\ell) a_\ell(\bfg)^2 - \ell^{\k'} \varepsilon_{\bfg}(\ell) a_\ell(\bff)^2 \right) X\\
     + \ell^{\k + \k'} a_\ell(\bff) a_\ell(\bfg) \varepsilon_{\bff}(\ell)\varepsilon_{\bfg}(\ell)  X^2
     - \ell^{1 + 2\k + 2\k'} \varepsilon_{\bff}(\ell)^2\varepsilon_{\bfg}(\ell)^2 X^3.
    \end{multline*}
    Here $\varepsilon_{\bff}$ denotes the prime-to-$p$ part of the Nebentypus of $\bff$ (which is a character $(\ZZ / N_f \ZZ)^\times \to \Lambda_{\bff}^\times$) and similarly for $\varepsilon_{\bfg}$.
    \item If $f,g$ are any specialisations of $\bff$, $\bfg$ respectively, of weights $k + 2, k' + 2$, and $j$ is an integer with $0 \le j \le \min(k, k')$, the image $\cBF^{\bff, \bfg}(f,g, j)$ of $\cBF^{\bff, \bfg}$ under the map
    \[ \operatorname{sp}_f \otimes \operatorname{sp}_g \otimes \mom^j_{\Gamma}: M(\bff \otimes \bfg)^*(-\j) \to M_{\Lp}(f \otimes g)^*(-j)
    \]
    is given by
    \[
     \cBF^{\bff, \bfg}(f, g, j) =  \frac{\left( 1 - \frac{p^j}{\alpha_f \alpha_g}\right)\left(c^2 - c^{2j-k-k'} \varepsilon_f(c)^{-1} \varepsilon_g(c)^{-1} \right)}{(-1)^j j! \binom{k}{j} \binom{k'}{j}} \Eis^{[f,g, j]}_\et.
    \]
    If $f, g$ are $p$-stabilisations of ordinary newforms $f_0, g_0$ of levels $N_f, N_g$, and we identify $M_{\Lp}(f \otimes g)^*$ with $M_{\Lp}(f_0 \otimes g_0)^*$ via $(\Pr^\alpha \times \Pr^\alpha)_*$ as in \S \ref{sect:lstab}, then
    \[
      \cBF^{\bff, \bfg}(f, g, j) =  \frac{\left( 1 - \frac{p^j}{\alpha_f \alpha_g}\right)\left( 1 - \frac{\alpha_f \beta_g}{p^{1 + j}}\right) \left( 1 - \frac{\beta_f \alpha_g}{p^{1 + j}}\right) \left( 1 - \frac{\beta_f \beta_g}{p^{1 + j}}\right)\left(c^2 - c^{2j-k-k'} \varepsilon_f(c)^{-1} \varepsilon_g(c)^{-1} \right)}{(-1)^j j! \binom{k}{j} \binom{k'}{j}} \Eis^{[f_0, g_0, j]}_{\et}.
    \]
   \end{enumerate}
  \end{theorem}

  \begin{remark}
   Note that the factor $Q_\ell$ appearing in part (a) can be written as
   \[ Q_\ell(X) = X^{-1}( (\ell - 1)(1 - \ell^{\k + \k' + 2} \varepsilon_f(\ell) \varepsilon_g(\ell) X^2) - \ell P_\ell(X) )\]
   where $P_\ell(X)$ is the Euler factor of $M(\bff \otimes \bfg)(1)$ at $\ell$. In particular, we have $Q_\ell(X) = -X^{-1} P_\ell(X)$ modulo $\ell - 1$. See \S 8.2 of \cite{LLZ14} for a conceptual interpretation of this factor $Q_\ell$, in terms of a conjectural rank-two Euler system.
  \end{remark}

  \begin{proof}
   Part (a) of the theorem is immediate from the definition of the classes $\cBF^{\bff, \bfg}_m$ and Proposition \ref{prop:cyclonorm}. Let us prove part (b).

   Firstly, by the definition of $\cBF^{\bff, \bfg}$, we have
   \[ \cBF^{\bff, \bfg}(f, g, j) = \pr_{f, g}\left( (\mom^k \boxtimes \mom^{k'} \otimes \mom_\Gamma^j) \cBF_{1, Np, 1}\right), \]
   and Theorem \ref{thm:BFeltsinterp1} allows us to rewrite this as
   \begin{multline*}
    \pr_{f,g}\left( \frac{(1 - p^j (U_p', U_p')^{-1})(c^2 - c^{2j-k-k'} (\langle c \rangle, \langle c \rangle))}{(-1)^j j! \binom{k}{j} \binom{k'}{j}} (e'_{\ord}, e'_{\ord}) \Eis^{[k, k', j]}_{\et, 1, Np}\right) \\=  \frac{\left( 1 - \frac{p^j}{\alpha_f \alpha_g}\right) \left(c^2 - c^{2j-k-k'} \varepsilon_f(c)^{-1} \varepsilon_g(c)^{-1} \right)}{(-1)^j j! \binom{k}{j} \binom{k'}{j}}\left(\Eis^{[f,g, j]}_{\et}\right)
    \end{multline*}
    (since $\pr_{f,g}$ factors through the ordinary projector).

    In the setting where $f,g$ are $p$-stabilisations of eigenforms $f_0$, $g_0$ of prime-to-$p$ level, we can use Theorem \ref{thm:lstab-eigen} (for $\ell = p$) to rewrite this in terms of $\Eis^{[f_0, g_0, j]}_{\et}$ multiplied by three extra Euler factors, which gives the second form of the theorem.
  \end{proof}

  We now study the interaction between the localisation of these classes $\cBF^{\bff, \bfg}_m$ and the filtration on the Galois representations $M(\bff)^*$ and $M(\bfg)^*$ at $p$. We begin with a local lemma:

  \begin{lemma}
   \label{lemma:iwacohotorsion}
   Let $K$ be a finite unramified extension of $\Qp$, and $T$ a finite-rank free $\Zp$-module with a continuous, unramified action of $G_K$. Let $V = T[1/p]$, and let $z \in H^1(K, T \otimes \Lambda_{\Gamma}(-\j))$.

   If the image of $z$ in $H^1(K, V(\chi))$ lies in the Bloch--Kato $H^1_{\mathrm{g}}$ subspace, for all finite-order characters $K$ of $\Gamma$, then we must have $z = 0$.
  \end{lemma}

  \begin{proof}
   Let us write $\varphi_K$ for the arithmetic Frobenius element of $\Gal(K^{\mathrm{nr}} / K)$. We note first that the $\Lambda_\Gamma$-torsion submodule of $H^1(K, T \otimes \Lambda_\Gamma(-\j))$ is isomorphic to $T^{\varphi_K = 1}$, by \cite[Proposition 2.1.6]{PerrinRiou-hauteurs}. In particular, it is $p$-torsion-free; so it suffices to check that $z = 0$ as an element of $H^1(K, V \otimes \Lambda_\Gamma(-\j))$.

   We compute readily that $H^1_{\mathrm{g}}(K, V(\chi)) = 0$ for any non-trivial finite-order character $\chi$ of $\Gamma$, while $H^1_{\mathrm{g}}(K, V)$ coincides with the unramified cohomology $H^1(K^{\mathrm{nr}} / K, V) \cong V / (1 - \varphi_K) V$. Consequently, if $z$ is as above, the image of $z$ in $H^1(K, V(\chi))$ must be zero for all non-trivial finite-order $\chi$. If $P_\chi$ denotes the ideal of $\Lambda_\Gamma[1/p]$ generated by the elements $\gamma - \chi(\gamma)$ for $\gamma \in \Gamma$, this shows that $z \in P_\chi \cdot H^1(K, V \otimes \Lambda_\Gamma(-\j))$. Since $\Lambda_\Gamma[1/p]$ is a finite direct product of principal ideal domains, and the $P_\chi$ for non-trivial finite-order $\chi$ are Zariski dense in $\Spec \Lambda_\Gamma[1/p]$, it follows that $z$ must be torsion.

   On the other hand, the composite map
   \[ V^{\varphi_K = 1} \cong H^1(K, T \otimes \Lambda_\Gamma(-\j))_{\mathrm{tors}} \rInto H^1(K, V)\]
   is given by cup-product with the class in $H^1(K, \Qp)$ mapping a topological generator of $\Gamma$ to $1$; so its image has zero intersection with $H^1_\mathrm{g}$. So we may conclude that $z = 0$ as required.
  \end{proof}

  \begin{notation}
   \label{notation:fplusminus}
   Let us write
   \[ \sF^{++} M(\bff \otimes \bfg)^* = \sF^+ M(\bff)^* \htimes \sF^+ M(\bfg)^*\]
   and similarly for $\sF^{+-}, \sF^{-+}, \sF^{--}$. We also use the notation $\sF^{+ \circ} = \sF^+ M(\bff)^* \htimes_{\Zp} M(\bfg)^*$.
  \end{notation}

  \begin{proposition}
   \label{prop:BFeltisSelmer}
   The inclusion $\sF^+ M(\bfg)^* \into M(\bfg)^*$ induces an injection
   \[ H^1\left(\Qp, \sF^{-+} M(\bff \otimes \bfg)^* \otimes \Lambda_\Gamma(-\j)\right) \into
   H^1\left(\Qp, \sF^{-\circ} M(\bff \otimes \bfg)^* \otimes \Lambda_\Gamma(-\j)\right),\]
   and the image of $\cBF^{\bff, \bfg}$ in the module $H^1\left(\Qp, \sF^{-\circ} M(\bff \otimes \bfg)^* \otimes \Lambda_\Gamma(-\j)\right)$ lies in the image of this injection.
  \end{proposition}

  \begin{proof}
   The obstruction to injectivity comes from the module
   \[ H^0\left(\Qp, \sF^{--} M(\bff \otimes \bfg)^* \otimes \Lambda_\Gamma(-\j)\right), \]
   which is clearly zero (because Iwasawa cohomology over $\Zp$-extensions always vanishes in degree 0). So it suffices to check that the image of $\cBF^{\bff, \bfg}$ in the module $H^1\left(\Qp, \sF^{--}M(\bff \otimes \bfg)^* \otimes \Lambda_\Gamma(-\j)\right)$ is zero.

   By definition, the modules $M(\bff)^*$ and $M(\bfg)^*$ are equal to the inverse limits of the corresponding modules $M(\bff)^*_{r}$, $M(\bfg)^*_{r}$ at finite level $Np^r$, $r \ge 1$ (the localisations at $\bff$ and $\bfg$ of the cohomology of $\overline{Y_1(N_f p^r)} \times \overline{Y_1(N_g p^r)}$ with coefficients in $\Zp(1)$). Moreover, all the modules concerned are compact, so inverse limits are exact and commute with Galois cohomology. So it suffices to prove that the image of $\cBF^{\bff, \bfg}$ in the module
   \[ H^1\left(\Qp, \sF^{--}M(\bff \otimes \bfg)_{r}^* \otimes \Lambda_\Gamma(-\j)\right) \]
   is zero for every $r \ge 1$. Note that the module $\sF^{--}M(\bff \otimes \bfg)_{r}^*$ is finitely-generated and free over $\Zp$, with an unramified Galois action.

   However, for every finite-order character the image of $\cBF^{\bff, \bfg}$ in $H^1\left(\Qp, \sF^{--}M(\bff \otimes \bfg)_{r}^* \otimes \Qp(\chi) \right)$ lies in the Bloch--Kato $H^1_\mathrm{g}$ subspace, by Proposition \ref{prop:BFeltisgeom}. By Lemma \ref{lemma:iwacohotorsion}, this shows that the image of $\cBF^{\bff, \bfg}$ in $H^1\left(\Qp, \sF^{--}M(\bff \otimes \bfg)_{r}^* \otimes \Lambda_\Gamma(-\j)\right)$ is zero as required.
  \end{proof}


 \subsection{The Perrin--Riou big logarithm}
  \label{sect:PRreg}

  \begin{definition}
   If $M$ is a unramified, $p$-adically complete $\Zp[G_{\Qp}]$-mod\-ule, we define
   \[ \DD(M) = \left( M \htimes_{\Zp} \widehat{\ZZ}_p^{\mathrm{nr}}\right)^{G_{\Qp}}.\]
   We write $\varphi$ for the operator on $\DD(M)$ arising from the arithmetic Frobenius on $\widehat{\ZZ}_p^{\mathrm{nr}}$.
  \end{definition}

  \begin{remark}
   If $M$ is free of finite rank as a $\Zp$-module, then $\DD(M)$ is a lattice in $\DD_{\mathrm{cris}}(M \otimes \Qp)$.
  \end{remark}

  \begin{theorem}
   \label{thm:PRreg}
   Suppose $M$ is an unramified, profinite $\Zp[G_{\Qp}]$-module. Then there is a map
   \[ \cL_M : H^1(\Qp, M \htimes \Lambda_\Gamma(-\j)) \to \DD(M) \htimes I^{-1} \Lambda_\Gamma, \]
   where $I$ is the ideal of $\Lambda_{\Gamma}$ that is the kernel of specialisation at $\j = -1$, with the following properties:
   \begin{itemize}
    \item The construction of $\cL_M$ is functorial in $M$, and in particular $\cL_M$ commutes with the action of $\operatorname{End}_{\Zp[G_{\Qp}]}(M)$ on both sides.
    \item If $M$ is finitely-generated and free, then $\cL_M$ is Perrin-Riou's big logarithm map for the unramified Galois representation $V = M[1/p]$.
    \item The kernel of $\cL_M$ is isomorphic to $H^0(\Qp, M)$.
    \item The image of $\cL_M$ in $\DD(M) \otimes \frac{I^{-1} \Lambda_\Gamma}{ \Lambda_\Gamma } \cong \DD(M)$ is contained in the submodule
    \( \DD(M)^{\varphi = 1} \cong H^0(\Qp, M) \).
   \end{itemize}
  \end{theorem}

  \begin{proof}
   This is an easy consequence of Coleman and Perrin-Riou's theory of big logarithm\footnote{This map has a confusing variety of names. Perrin-Riou refers to it as the ``$p$-adic regulator'', but it is perhaps best to avoid this notation here to avoid confusion with the use of ``regulator'' for the natural maps from motivic cohomology to other cohomology theories. The term ``big dual exponential'' is also used in some sources.} maps, as extended by the second and third authors in \cite{LZ}.

   We consider the module $\NN(M)$ \emph{defined} as $\DD(M)[[\pi]]$, where $\pi$ is a formal variable. This rather brutal definition of $\NN(M)$ coincides with the usual Wach module functor when $M$ is finite free over $\Zp$. For any $M$, the module $\NN(M)$ admits a Frobenius $\varphi$ defined as the tensor product of the Frobenius map on $\DD(M)$ and the map $\pi \mapsto (1 + \pi)^p - 1$ on $\Zp[[\pi]]$; and a left inverse $\psi$ of $\varphi$, defined by tensoring the inverse Frobenius of $\DD(M)$ with the usual trace operator $\psi: \Zp[[\pi]] \to \Zp[[\pi]]$ of $(\varphi, \Gamma)$-module theory.

   Then $1 - \varphi$ gives an exact sequence
   \[ 0 \rTo \DD(M)^{\varphi = 1} \rTo \NN(M)^{\psi = 1} \rTo^{1 - \varphi} \NN(M)^{\psi = 0} \rTo \frac{\DD(M)}{(1 - \varphi) \DD(M)} \rTo 0, \]
   where the last map is given by evaluation at $\pi = 0$. Moreover, we have the Fontaine isomorphism
   \[ H^1(\Qp, M \otimes \Lambda_{\Gamma}(-\j)) \cong \left( \pi^{-1} \NN(M) \right)^{\psi = 1}.\]
   (Since both sides commute with inverse limits, it suffices to prove this for $M$ finite, and this follows by exactly the same argument as in \cite[Appendix A]{berger03}.) This isomorphism depends on identifying the formal variable $\pi$ with the element of the same name in Fontaine's ring $\mathbf{A}^+_{\Qp}$, whose definition depends on a choice of compatible system of $p$-power roots of unity $(\zeta_{p^r})_{r \ge 0}$ in $\overline{\QQ}_p$; we have already fixed roots of unity $\zeta_m$ for all $m \ge 1$ in $\overline{\QQ}$, and an embedding $\overline{\QQ} \into \overline{\QQ}_p$, so this choice is already made.

   Since $\varphi$ is invertible on $\DD(M)$, we have $\NN(M)^{\varphi = 0} = \DD(M) \htimes \Zp[[\pi]]^{\psi = 0} \cong \DD(M) \htimes \Lambda_{\Gamma}$. This gives a map
   \[ \NN(M)^{\psi = 1} \to \DD(M) \htimes \Lambda_\Gamma, \]
   and since the ideal $I$ annihilates the quotient $\pi^{-1} \NN(M) / \NN(M)$, and this gives the required map. The functoriality of the construction is obvious, and for the last statement we note that since $\psi(\pi^{-1}) = \pi^{-1}$ we have
   \[
    \left(\pi^{-1} \NN(M)\right)^{\psi = 1} =
    \pi^{-1} \DD(M)^{\varphi = 1} + \NN(M)^{\psi = 1},
   \]
   so that the image of $\cL_M$ is contained in
   \[
    \DD(M)^{\varphi = 1} \htimes I^{-1} + \DD(M) \htimes \Lambda(\Gamma).\qedhere
   \]
  \end{proof}

  \begin{remark}
   The cokernel of $\cL_M$ as a map into $\DD(M)^{\varphi = 1} \htimes I^{-1} + \DD(M) \htimes \Lambda(\Gamma)$ is isomorphic to $M / (1 - \varphi) M$, with trivial $\Gamma$-action.
  \end{remark}

  \begin{remark}
   When $M = T \otimes \Lambda_U(\mathbf{u}^{-1})$, for $U$ a $p$-adic Lie quotient of $\Gal(\QQ_p^{\mathrm{nr}} / \QQ)$ with Galois acting on $\Lambda_U = \Zp[[U]]$ via the inverse of the canonical character $\mathbf{u}$, and $T$ finitely-generated over $\Zp$, this coincides with the construction of \cite{LZ}. In this case, the module $\NN(M)$ above is the $\NN_\infty(T)$ of \emph{op.cit.}.

   In \cite{LZ} one can allow $T$ to be any lattice in a crystalline $G_{\Qp}$-represent\-ation, not necessarily unramified; and it seems reasonable to envisage a common generalisation of the above theorem and the results of \cite{LZ}, where $M$ is allowed to be any inverse limit of lattices in crystalline $G_{\Qp}$-representations with Hodge--Tate weights in some specified range. We do not know how to prove this at present.
  \end{remark}

  We are interested in the case of the module $M = \sF^{-+} M(\bff \otimes \bfg)^*$, where $\bff$ and $\bfg$ are Hida families as before. This is not unramified, but the module $M(-1-\k')$ is unramified, where we write $\k'$ for the canonical character $\Gamma \to \Lambda_D^\times \to \Lambda_{\bfg}^\times$ (to distinguish it from the corresponding construction for $\Lambda_\bff$). So the module $\DD(M(-1-\k'))$ is well-defined.

  \begin{lemma}
   \label{lemma:plusminusinvariants}
   We have $H^0\left(\Qp, \sF^{-+} M(\bff \otimes \bfg)^*(-1-\k')\right) = 0$, for any two Hida families $\bff$ and $\bfg$.
  \end{lemma}

  \begin{proof}
   This is similar to the proof of Proposition \ref{prop:BFeltisSelmer}, but with the important difference that the Eisenstein series do not contribute to the plus filtration step $\sF^+ M(\bfg)^*$. We choose integers $k, k' \ge 0$ and write
   \[ M(\bff)^* = \varprojlim_{r \ge 1} M(\bff)^* / I_{k, r},\quad M(\bfg)^* = \varprojlim_{r \ge 1} M(\bfg)^* / I_{k', r} \]
   where $I_{k, r}$ is the ideal of $\Lambda_D$ appearing in Proposition \ref{prop:controlthm}. Note that $M(\bff)^* / I_{k, r}$ is isomorphic, by the control theorem, to the localisation of $H^1(Y_1(Np^r)_{\overline{\QQ}}, \TSym^k(\sH_{\Zp})(1))$ at $\bff$; in particular, it is a finitely-generated free $\Zp$-module. If we assume $k, k' > 0$ (so no specialisations of $\bff$ or $\bfg$ in weights $k, k'$ can be classical $p$-new forms), then every eigenvalue of Frobenius on $\sF^+ M(\bfg)^*(-1-\k') / I_{k', r}$ is a Weil number of weight $-1-k'$, and every eigenvalue of Frobenius on $\sF^- M(\bff)^* / I_{k, r}$ has weight either $0$ or $1 + k$. So if we also assume $k \ne k'$, it follows that no Frobenius eigenvalue on the module $\sF^- M(\bff)^*_{k, r} \otimes \sF^+ M(\bfg)^*_{k', r}(-1-\k')$ can be equal to 1, and passing to the limit over $r$ gives the result.
  \end{proof}

  \begin{definition}
   We shall write $\DD(\sF^{-+} M(\bff \otimes \bfg)^*)$ for the module $\DD(\sF^{-+} M(\bff \otimes \bfg)^*(-1-\k'))$, equipped with the non-trivial action of $\Gamma$ given by the character $1 + \k'$.
  \end{definition}

  This slightly contrived definition implies that the specialization of $\DD(\sF^{-+} M(\bff \otimes \bfg)^*) \htimes \Lambda_{\Gamma}$ at a triple $(\k, \k', \j) = (k, k', j)$ of integers is exactly $\DD_{\mathrm{cris}}$ of the corresponding specialization of $M$.

  \begin{theorem}
   \label{thm:bigloginterp}
   There is an injective morphism of $(\Lambda_\bff \htimes \Lambda_\bfg  \htimes \Lambda_\Gamma)$-modules
   \[
    \cL: H^1\left(\Qp, \sF^{-+} M(\bff \otimes \bfg)^* \htimes \Lambda_\Gamma(-\j)\right)
    \to \DD(\sF^{-+} M(\bff \otimes \bfg)^*) \htimes \Lambda_{\Gamma}
   \]
   with the following property: for all classical specialisations $f, g$ of $\bff, \bfg$, and all characters of $\Gamma$ of the form $\tau = j + \eta$ with $\eta$ of finite order and $j \in \ZZ$, we have a commutative diagram
   \begin{diagram}
    H^1\left(\Qp, \sF^{-+} M(\bff \otimes \bfg)^* \htimes \Lambda_{\Gamma}(-\j)\right) & \rTo^{\cL} & \DD(\sF^{-+} M(\bff \otimes \bfg)^*) \htimes \Lambda_{\Gamma} \\
    \dTo & & \dTo\\
    H^1(\Qp, \sF^{-+} M(f \otimes g)^*(-j-\eta)) & \rTo & \DD_{\mathrm{cris}}(\sF^{-+} M(f \otimes g)^*(-\varepsilon_{g, p}))
   \end{diagram}
   in which the bottom horizontal map is given by
   \[
    \left.\begin{cases}
     \left(1 - \frac{p^j}{\alpha_f \beta_g}\right)\left(1 - \frac{\alpha_f \beta_g}{p^{1 + j}}\right)^{-1} & \text{if $r = 0$} \\
     \left(\frac{p^{1 + j}}{\alpha_f \beta_g}\right)^{r} G(\varepsilon)^{-1} & \text{if $r > 0$}
    \end{cases}\right\}
     \cdot
    \begin{cases}
     \tfrac{(-1)^{k'-j}}{(k'-j)!} \log & \text{if $j \le k'$,} \\ (j-k'-1)! \exp^* & \text{if $j > k'$.}
    \end{cases}
   \]
   Here $\exp^*$ and $\log$ are the Bloch--Kato dual-exponential and logarithm maps, $\varepsilon$ is the finite-order character $\varepsilon_{g, p} \cdot \eta^{-1}$ of $\Gamma$, $r \ge 0$ is the conductor of $\varepsilon$, and $G(\varepsilon) = \sum_{a \in (\ZZ / p^r \ZZ)^\times} \varepsilon(a) \zeta_{p^r}^a$ is the Gauss sum.
  \end{theorem}

  \begin{remark}
   Here $\beta_g$ is defined by $\alpha_g \beta_g = p^{1 + k'} \varepsilon_{g, N}(p)$, where $\varepsilon_{g, N}$ is the prime-to-$p$ part of the nebentypus of $g$; this definition makes sense even when $g$ has level divisible by $p$, and the arithmetic Frobenius acts on $\sF^+ M(g)^*(-1-k'-\varepsilon_{g, p})$ as multiplication by $p^{-1-k'} \beta_g$.  In the proof of the explicit reciprocity law, we shall only use the interpolating property in the simplest case, when $f$ and $g$ have level prime to $p$ and $\tau(z) = z^j$, but we have given the general formula as this may be required in some applications of this work.

   In the case $r = 0$, we assume $\alpha_f \beta_g / p^{1 + j} \ne 1$ for simplicity (although a precise formula for the interpolating property of $\cL$ can be given without this assumption, cf.\ \cite[Theorem 3.1.2]{LVZ}).
  \end{remark}

  \begin{proof}
   We apply Theorem \ref{thm:PRreg} to the unramified module $M(-1-\k')$. By Lemma \ref{lemma:plusminusinvariants} the resulting big logarithm map $\cL$ is injective and has no poles, and by functoriality it is compatible with the $(\Lambda_{\bff} \htimes \Lambda_{\bfg}\htimes\Lambda_{\Gamma})$-module structure.

   We now pull back by the automorphism of this ring given by mapping $[d] \in \Lambda_{\Gamma}$ to $d^{-1-\k'} [d]$. The pullback of $M(-1-\k') \htimes \Lambda_\Gamma(-\j)$ is then $M \htimes \Lambda_\Gamma(-\j)$; and the pullback of $\DD(M(-1-\k)) \htimes \Lambda_{\Gamma}$ is the module $\DD(M) \htimes \Lambda_{\Gamma}$ defined above. The interpolating property is now an elementary exercise from the standard formulae for the Perrin-Riou logarithm, cf.\ \cite[Appendix B]{LZ}.
  \end{proof}

  Via Proposition \ref{prop:BFeltisSelmer}, we can map $\cBF^{\bff, \bfg}$ into $H^1(\Qp, \sF^{-+} M(\bff \otimes \bfg)^* \htimes \Lambda_{\Gamma}(-\j))$, and hence we can define $\cL\left(\cBF^{\bff, \bfg}\right) \in \DD(\sF^{-+} M(\bff \otimes \bfg)^*) \otimes \Lambda_\Gamma$. The goal of the next two sections of this paper will be Theorem \ref{lthm:explicitrecip}, which interprets this object as a $p$-adic $L$-function.


 \section{Comparison of Eichler--Shimura isomorphisms}
  \label{sect:ohtacompat}

 We now fill in a crucial technical ingredient required for the proof of the explicit reciprocity law (Theorem \ref{lthm:explicitrecip} of the introduction): a compatibility between the de Rham comparison isomorphisms $\comp_{\dR}$ of \eqref{eq:comparisoniso} for cusp forms of different weights in Hida families.

 While this can be proved directly, by purely geometric methods, we give an alternative argument based on Kato's explicit reciprocity law for his $\GL_2$ Euler system, as this is more in keeping with the flavour of the present paper and involves less translation between different normalisations. This argument is closely based on work of Daniel Delbourgo \cite{delbourgo08}, in particular Theorem 6.4 of \emph{op.cit.}.

 \subsection{Kato's \texorpdfstring{$\GL_2$}{GL(2)} Euler system}

  We now invoke Kato's Euler system theory. All references here are to \cite{Kato-p-adic}.

  Let $N, m \ge 1$, $k \ge 0$, and $r \in \ZZ$. Let $A \ge 1$ and $a \in \ZZ / A\ZZ$. Choose auxiliary integers $(c, d) > 1$ such that $(c, 6pmA) = (d, 6pmAN) = 1$. Following the constructions in \S 8.1.2, 8.9 of \emph{op.cit.} we construct cohomology classes
  \[ {}_{c, d} z_{1, N, m}^{(p)}(k+2, r, a(A), S) \in H^2\left(Y_1(N) \times \mu_m^\circ, \TSym^k(\sH_{\Zp})(2 - r)\right),
  \]
  which are compatible with norm maps in both $N$ and $m$ (8.7). We shall always assume $m$ is a power of $p$.

  \begin{remark}
   Kato defines slightly more general classes ${}_{c, d} z_{1, N, m}^{(p)}(k+2, r, r', a(A), S)$; we have chosen $r' = 1$ and $S$ the set of prime factors of $mpA$, and suppressed it from the notation.
  \end{remark}

  Kato shows the following explicit reciprocity law (Theorem 9.6) describing the image of this class under the dual exponential map, for $r = k+1$:

  \begin{theorem}[Kato's explicit reciprocity law]
   \label{thm:katorecip}
   For $r = k + 1$ we have
   \[
    \comp_{\dR}\left(\exp^*\left({}_{c, d} z_{1, N, m}^{(p)}(k+2, k+1, a(A))\right)\right)
    = R \cdot {}_{c, d} z_{1, N, m}(k+2, k+1, a(A))
   \]
   where ${}_{c, d} z_{1, N, m}(k+2, k+1, a(A))$ is a certain modular form of weight $k+2$, built from holomorphic Eisenstein series (5.2); $\comp_{\dR}$ is the Faltings--Tsuji comparison isomorphism (11.3.6, or \ref{eq:comparisoniso} in the present paper); and $R$ is an Euler factor (5.3 (2)) given by
   \[ R = \begin{cases}
     1 & \text{if $p \mid mA$,}\\
     1 - p^{-1-k} U_p'\sigma_p^{-1} & \text{if $p \nmid mA$ and $p \mid N$,}\\
     1 - p^{-1-k} T'_p \sigma_p^{-1} + p^{-1-k}\langle p \rangle^{-1} \sigma_p^{-2}& \text{if $p \nmid mAN$.}
    \end{cases}
   \]
  \end{theorem}
  We shall use this theorem to show that the maps $\comp_{\dR}$ must interpolate in families, since the terms $\exp^*\left({}_{c, d} z_{1, N, m}^{(p)}(k+2, k+1, a(A))\right)$ and ${}_{c, d} z_{1, N, m}(k+2, k+1, a(A))$ both interpolate, and the submodule that they span is large.

 \subsection{Interpolation}

  \begin{proposition}
   There is a \emph{``Beilinson--Kato class''}
   \[  {}_{c, d} \mathcal{BK}_{N}(a(A)) \in H^2_{\et}\Big(Y_1(N), \Lambda(\sH_{\Zp} \times \ZZ_p^\times(1))(2)\Big) \]
   whose image under the moment map
   \[ \mom^{k, -j}: \Lambda(\sH_{\Zp} \times \ZZ_p^\times(1)) \to \TSym^k(\sH_{\Zp}) \otimes \Zp(-j),\]
   for any $k \ge 0$ and $j \in \ZZ$, is the class ${}_{c, d} z_{1, N, m}^{(p)}(k+2, j, a(A))$.
  \end{proposition}

  \begin{proof}
   For $a(A) = 0(1)$ this is clear from the construction of the classes ${}_{c, d} z_{1, N, m}^{(p)}(k, r)$, see (8.4.3). For $A > 1$ we need to check that the moment map commutes with the pushforward along the map $t_{m, a(A)}$ used to define $z_{1, N, m}^{(p)}(k, r, a(A))$; but this is clear from the definitions.
  \end{proof}

  We shall relate this class to modular forms using the Perrin-Riou big logarithm, as in the $\GL_2 \times \GL_2$ theory of the previous sections. Using the Hochschild--Serre spectral sequence, we can and do interpret ${}_{c, d} \mathcal{BK}_{N}(a(A))$ as an element of the module
  \[ H^1\left(\ZZ[1/Np], H^1_{\ord}(Np^\infty) \htimes \Lambda_\Gamma(1-\j)\right).\]
  The Perrin-Riou logarithm in this context is a map of $(\Lambda_D \htimes \Lambda_\Gamma)$-modules
  \[ \cL: H^1\left(\Qp, \sF^- H^1_{\ord}(Np^\infty) \htimes \Lambda_\Gamma(1-\j)\right) \rTo \DD\left(\sF^- H^1_{\ord}(Np^\infty)\right)\htimes I^{-1} \Lambda_{\Gamma}, \]
  where $I$ is the fractional ideal of $\Lambda_\Gamma$ as in \S \ref{sect:PRreg} above, with the following interpolation property: for any $k \ge 0$, any $r \ge 1$, any $j \ge 1$ and any Dirichlet character $\chi$ of $p$-power conductor $p^s$, there is a commutative diagram
  \begin{diagram}
   H^1\left(\Qp, \sF^- H^1_{\ord}(Np^\infty) \htimes \Lambda_\Gamma(1-\j)\right) &\rTo^{\cL}& \DD\left(\sF^- H^1_{\ord}(Np^\infty)\right)\htimes I^{-1}\Lambda_{\Gamma}\\
   \dTo^{\pr_{k, r} \otimes \operatorname{ev}_{\j = j + \chi}} & & \dTo_{\pr_{k, r} \otimes \operatorname{ev}_{\j = j + \chi}}\\
   H^1\left(\Qp, \sF^- V_{k, r} (1-j-\chi)\right) &\rTo^{S \cdot \exp^*} & \DD\left(\sF^- V_{k, r}\right)
  \end{diagram}
  where $V_{k, r} = e'_{\ord} H^1_{\et}\left(Y_1(Np^r)_{\overline{\QQ}_p}, \TSym^k(\sH_{\Zp})(1)\right)$, which acquires from the control theorem a two-step filtration such that $\sF^- V_{k, r}$ is unramified, and $\pr_{k, r}$ is the natual map $H^1_{\ord}(Np^\infty) \to V_{k, r}$ given by the control theorem. The factor $S$ is given by
  \[
   S = (j-1)!
   \begin{cases}
    (1 - p^{j-1} (U_p')^{-1})(1 - p^{-j} U_p')^{-1} & \text{if $s = 0$},\\
    G(\chi^{-1})^{-1} p^{sj}(U_p')^s & \text{if $s \ge 1$.}
   \end{cases}
  \]
  Here $G(\chi^{-1}) = \sum_{a \in (\ZZ / p^s\ZZ)^\times} \chi^{-1}(a) \zeta_{p^s}^a$ is the Gauss sum, as usual. We will apply this with $s \ge 1$ and $j = k+1$, so that the factor $S$ is given by $\frac{k! p^{(k+1)s}}{G(\chi^{-1}) (U_p')^{s}}$.

 \subsection{Atkin--Lehner operators on Kato's Eisenstein series}

  Let $N \ge 1$ be coprime to $p$. For $\chi_1, \chi_2$ primitive Dirichlet characters of $p$-power conductor (possibly trivial), and $t \in \ZZ / N\ZZ$,  let us introduce the notations
  \begin{align*}
   \zeta\left(t(N), \chi_1, s\right) &\coloneqq \sum_{\substack{n \ge 1 \\ n = t \bmod N}} \chi_1(n) n^{-s},\\
   \sigma_{k-1} \left(t(N), \chi_1, \chi_2, n\right) &\coloneqq\sum_{\substack{uv = n \\ u = t \bmod N}} \chi_1(u) \chi_2(v) v^{k-1}.
  \end{align*}
  As usual, we understand $\chi(n) = 0$ if $n$ is not coprime to the conductor of $\chi$.

  \begin{proposition}
   Let $k \ge 1$ and write $\pm = (-1)^k \chi_1(-1)\chi_2(-1)$. Suppose that either $k = 1$, or $\chi_1$ and $\chi_2$ are both nontrivial. Then the series
   \[
    G^{(k)}\left(t(N), \chi_1, \chi_2\right) \coloneqq
    a_0 + \sum_{n \ge 1} q^n \Big(\sigma_{k-1}(t(N), \chi_1, \chi_2, n) \pm \sigma_{k-1}(-t(N), \chi_1, \chi_2, n)\Big)
   \]
   is a modular form of weight $k$ and level $N \operatorname{cond}(\chi_1) \operatorname{cond}(\chi_2)$, where the constant term is given by
   \[ a_0 =
    \begin{cases}
     0 & \text{if $\chi_1$ and $\chi_2$ are both nontrivial,}\\
     \tfrac{1}{2}\left(\zeta\left(t(N), \chi_1, 0\right) \pm \zeta\left(-t(N), \chi_1, 0\right)\right)
     &\text{if $k = 1$ and $\chi_2$ is trivial.}
    \end{cases}
   \]
  \end{proposition}

  \begin{proof}
   This is standard. (The restriction on $k, \chi_1, \chi_2$ can be relaxed, of course, but this covers all the cases we shall use.)
  \end{proof}

  \begin{theorem}
   \label{thm:ALaction}
   Suppose $L \ge 1$ with $p \nmid L$, and let $r,s,t$ be integers with $s, t \ge 1$ and $r \ge s + t$.

   Let $\chi$, $\nu$ and $\varepsilon$ be Dirichlet characters modulo $p^s$, $p^t$, and $p^r$ respectively, with $\chi$ primitive, satisfying the sign condition $\chi(-1) \nu(-1) = (-1)^{k + 1}$. Then we have
   \begin{multline*}
    \sum_{\alpha \in (\ZZ / p^r \ZZ)^\times} \varepsilon(\alpha)^{-1} (\langle \alpha\rangle_p
    \circ U_p^r \circ W_{Np^r}^{-1})\left[\sum_{\substack{a \in (\ZZ / p^t\ZZ)^\times \\ b \in (\ZZ / p^s\ZZ)^\times}} \nu(a)^{-1} \chi(b)^{-1} \sigma_b \cdot {}_{c, d} z_{1, Np^r, p^s}(k+2, k+1, a(p^t))\right]  \\
    = \frac{G(\chi^{-1})  U_p^{s + t}}{k! p^{(k+1)s} \nu(N)^{-1}} \cdot \left[(c^2 - c^{k + 1} \chi(c) \nu(c))G^{(k+1)}(0(N), \nu^{-1}, \chi)\right] \cdot\, \\
    \left[ d^2 G^{(1)} \left(1(N), \varepsilon \nu \chi^{-1}, \mathrm{id}\right) - d\, \varepsilon \nu \chi^{-1}(d^{-1}) G^{(1)} \left(d(N), \varepsilon \nu \chi^{-1}, \mathrm{id}\right)\right]
   \end{multline*}
  \end{theorem}

  \begin{proof}
   This can be shown via a lengthy explicit calculation from the formulae for the modular form ${}_{c, d} z_{1, N, m}(k+2, k+1, a(A))$ given in \cite[Proposition 5.8]{Kato-p-adic}.
  \end{proof}

  We now fix a finite extension $E / \Qp$ containing the values of the character $\nu$.

  \begin{proposition}
   There are $p$-adic modular forms ${}_c \mathcal{G}(0(N), \nu^{-1}, \j)$ and ${}_d \mathcal{G}(1(N), 2 + \nu + \k - \j, \mathrm{id})$ with coefficients in $\Lambda_D \htimes \Lambda_\Gamma \otimes \cO_E$ and weight-characters $- \nu + \j$ and $2 + \nu + \k - \j$ respectively, whose images under evaluation at $\k = k + \varepsilon, \j = k + 1 + \chi$, for any $k \ge 0$ and Dirichlet characters $\chi, \varepsilon$, are the two factors in square brackets on the right-hand side of Theorem \ref{thm:ALaction}.
  \end{proposition}

  \begin{proof}
   We write
   \[ \mathcal{G}(0(N), \nu^{-1}, \j) \coloneqq 2 \sum_{\substack{n \ge 1 \\ p \nmid n}} q^n \left(\sum_{\substack{uv = n \\ u = 0 \bmod N}} \nu^{-1}(u) v^{\j-1}\right),\]
   and ${}_c \mathcal{G}(0(N), \nu^{-1}, \j) = (c^2 - c^{\j} \nu(c)) \mathcal{G}(0(N), \nu^{-1}, \j)$. Then evaluating at $\j = k + 1 + \chi$ maps this to the first Eisenstein series $(c^2 - c^{k + 1} \chi(c) \nu(c))G^{(k+1)}(0(N), \nu^{-1}, \chi)$ appearing in Theorem \ref{thm:ALaction}.

   For the second factor, for $n \ge 1$ the coefficient of $q^n$ in $G^{(1)} \left(t(N), \varepsilon \nu \chi^{-1}, \mathrm{id}\right)$ is the image of
   \[ \left(\sum_{\substack{uv = n \\ u = t \bmod N \\ p \nmid u}} \nu(u) u^{1 + \k - \j} + (-1)^{\k - \j} \nu(-1)\sum_{\substack{uv = n \\ u = -t \bmod N \\ p \nmid u}} \nu(u) u^{1 + \k - \j}\right) \in \Lambda_D \htimes \Lambda_\Gamma \]
   under evaluation at $\j = k + 1 + \chi$, $\k = k + \varepsilon$; so it suffices to treat the constant term. A standard computation using Bernoulli polynomials shows that the values $\zeta(t(N), \eta, 0)$ for $p$-power Dirichlet characters $\eta$ are the values at $\j = \eta$ of an element $\zeta_p(t(N), \j) \in \Frac \Lambda_\Gamma$, with a pole at the character $\j = -1$; and that for $d > 1$ coprime to $6pN$ the element
   \[ {}_d \zeta_p(t(N), \j) \coloneqq d^2 \zeta_p(t(N), \j) - d^{1-\j}\zeta_p(dt(N), \j) \]
   is in $\Lambda_\Gamma$. The constant term of ${}_d \mathcal{G}(1(N), 2 + \nu + \k - \j, \mathrm{id})$ is then given by
   \[ \tfrac{1}{2} \left[  {}_d \zeta_p(t(N), 1 + \nu + \k - \j) + (-1)^{\k - \j} \nu(-1)\, {}_d \zeta_p(-t(N), 1 + \nu + \k - \j)\right].\]
  \end{proof}

  \begin{corollary}
   \label{cor:wtkohta}
   Let $k \ge 0$, and let
   \[ {}_{c, d} h_{k, r, \nu} \in e'_{\ord} M'_{k + 2}(Np^r, E) \otimes I^{-1}\Lambda_{\Gamma} \]
   be the unique modular form such that
   \[ (\pr_{\sF^-} \circ \comp_{\dR})({}_{c, d} h_{k, r, \nu}) = \sum_{\substack{a \in (\ZZ / p^t\ZZ)^\times}} \nu(a)^{-1} \pr_{k, r} \cL({}_{c,d} \mathcal{BK}(a(A))) .\]
   Then, for each non-trivial Dirichlet character $\varepsilon$ modulo $Np^r$, the image of ${}_{c, d} h_{k, r, \nu}$ under the operator
   \[ \sum_{\alpha \in (\ZZ / p^r \ZZ)^\times} \varepsilon(\alpha)^{-1} \left(\langle \alpha\rangle_p \circ U_p^r \circ W_{Np^r}^{-1}\right) \]
   is the specialisation at $\k = k + \varepsilon$ of the $p$-adic modular form
   \[ U_p^t\, \nu(N)^{-1}\, e_{\ord} \Big[ {}_c \mathcal{G}(0(N), \nu^{-1}, \j) \cdot {}_d \mathcal{G}(1(N), 2 + \nu + \k - \j, \mathrm{id}) \Big] \in M_{\k + 2}^{\ord}(N, \Lambda_D) \htimes \Lambda_{\Gamma} \otimes \cO_E.\]
  \end{corollary}

  \begin{proof}
   Let us write $V_{k, r} = e'_{\ord} H^1_{\et}\left(Y_1(Np^r)_{\overline{\QQ}_p}, \TSym^k(\sH_{\Zp})(1)\right)$ as before. By weak admissibility, for each $r \ge 1$ the natural map
   \[ \pr_{\sF^-} \circ \comp_{\dR}: e'_{\ord} M'_{k + 2}(Np^r, \Qp) \to \DD(\sF^- V_{k, r}) \otimes \Qp \]
   is an isomorphism, so there exists a unique form
   \[ {}_{c, d} h_{k, r, \nu} \in e'_{\ord} M'_{k + 2}(Np^r, E) \otimes I^{-1} \Lambda_{\Gamma}\]
   mapping to the image of the Beilinson--Kato element as in the statement of the proposition.

   To check that its image under $\sum_{\alpha \in (\ZZ / p^r \ZZ)^\times} \varepsilon(\alpha)^{-1} \left(\langle \alpha\rangle_p \circ U_p^r \circ W_{Np^r}^{-1}\right)$ coincides with the stated $p$-adic modular form, it suffices to check that they agree on specialising $\j$ to a Zariski-dense set of characters of $\Gamma$. We choose the characters of the form $\j = k + 1 + \chi$ with $\chi$ a nontrivial Dirichlet character. If the conductor $s$ of $\chi$ satisfies $s + t \le r$, then this follows from Kato's explicit reciprocity law (Theorem \ref{thm:katorecip}) together with Theorem \ref{thm:ALaction}, since the since the factor relating the Perrin-Riou logarithm to the dual exponential map for $H^1_{\ord}(Np^\infty)(-k)$ is given by $\frac{k! p^{(k + 1)s}}{G(\chi^{-1}) (U_p')^s}$, which cancels out most of the factors on the right-hand side of Theorem \ref{thm:ALaction}.

   If $r < s + t$, then we use the fact that the $h_{k, r, \nu}$ are norm-compatible in $r$ (by definition) to replace $r$ with some $r' > r$ such that $r' \ge s + t$, and the argument proceeds as before.
  \end{proof}

  \begin{remark}
   Crucially, the $p$-adic modular form appearing in the corollary is independent of the integer $k$, even though its defining property involves the comparison isomorphism $\comp_{\dR}$ (which is defined using a Kuga--Sato variety whose dimension depends on the weight $k$).
  \end{remark}

 \subsection{Special values}

  The modular form ${}_{c, d} z_{1, N, m}(k+2, k+1, a(A))$ appearing in Kato's explicit reciprocity law is related to the $L$-values $L(f, \chi, k+1) L(f, \nu, 1)$, for characters $\chi$ modulo $m$ and $\nu$ modulo $A$, and newforms $f \in S_{k+2}(N)$.

  \begin{proposition}
   Let $\chi$ and $\nu$ be primitive Dirichlet characters modulo $A$ and $m$ respectively such that $\chi(-1) \nu(-1) = (-1)^{k+1}$. Then there is a nonzero constant $C$, depending on $k, \nu, \chi$, such that for every $f \in S_{k+2}(\Gamma_1(N), \CC)$ we have
   \begin{multline*}
    \sum_{\substack{a \in (\ZZ / A\ZZ)^\times \\ b \in (\ZZ / m\ZZ)^\times}} \nu(a)^{-1} \chi(b)^{-1} \int_{\Gamma_1(N) \backslash \HH} f(-\bar{\tau})\, {}_{c, d} z_{1, N, m}(k+2, k+1, a(A))(\tau) \operatorname{Im}(\tau)^k \mathrm{d}z \wedge \mathrm{d}\bar z \\
    = \Big[(c^2 - c^{k+1} \chi(c) \nu(c)) (d^2 - d \nu(d)^{-1} \chi(d) \langle d^{-1} \rangle) Z_{1, N}(k+2, \chi^{-1}, k+1) Z_{1, N}(k+2, \nu^{-1}, 1) f\Big]_1,
   \end{multline*}
   where the notation $[g]_1$ signifies the coefficient of $q$ in the $q$-expansion of $f$, and $Z_{1, N}(k+2, \chi^{-1}, s)$ is the Dirichlet series with values in $\operatorname{End}_\CC M_{k+2}(\Gamma_1(N), \CC)$ defined by
   \[ \sum_{(n, m) = 1}  \chi(n)^{-1} T_n n^{-s}. \]
  \end{proposition}

  \begin{proof}
   This can be extracted from the last formula in Kato's Theorem 5.6 (2), using the fact that the Betti cohomology classes $\delta_{1, N}(k+2, 1, a(A))$ are given by integration along the path from $a/A$ to $\infty$, and thus the class $\sum_{a} \nu(a)^{-1} \delta_{1, N}(k+2, 1, a(A))$ paired with $f$ computes the value at $s = 1$ of the $L$-series $\sum_{n \ge 1} \nu(a)^{-1} a_n(f) n^{-s} = \left[Z_{1, N}(k+2, \nu^{-1}, 1) f\right]_1$.
  \end{proof}

  \begin{lemma}
   For any $k \ge 0$, $N \ge 1$, and prime $p$, we may find Dirichlet characters $\chi, \nu$ of $p$-power conductor such that the Hecke operator
   \[ (c^2 - c^{k+1} \chi(c) \nu(c)) (d^2 - d \nu(d)^{-1} \chi(d) \langle d^{-1} \rangle) Z_{1, N}(k+2, \chi^{-1}, k+1) Z_{1, N}(k+2, \nu^{-1}, 1) \]
   is invertible on $S_{k+2}(\Gamma_1(N), \CC)$.
  \end{lemma}

  \begin{proof}
   It suffices to prove the result on each eigenspace $S_{k+2}(\Gamma_1(N), \CC)[f]$, for $f$ a newform of level dividing $N$, since the space $S_{k+2}(\Gamma_1(N), \CC)$ is equal to the direct sum of these eigenspaces and each is stable under the Hecke operators.

   We can easily arrange that the factors $(c^2 - c^{k+1} \chi(c) \nu(c))$ and $(d^2 - d \nu(d)^{-1} \chi(d) \langle d^{-1} \rangle)$ are non-zero (indeed this is automatic unless $k = 1$, in which case we need only assume that $\chi(c) \nu(c) \ne 1$).

   The term $Z_{1, N}(k+2, \chi^{-1}, k+1)$ has an Euler product, convergent for $\Re(s) \gg 0$; so it can in particular be written as a product of terms at primes $\ell \mid mN$ and $\ell \nmid mN$. The ``tame'' part of the series acts as a constant on each direct summand $S_{k+2}(\Gamma_1(N), \CC)[f]$; these constants are the $L$-values $L_{\{mN\}}(f, \chi^{-1}, k+1)$. These differ by finitely many Euler factors from the full $L$-series $L(f, \chi^{-1}, k+1)$; we can arrange that the missing Euler factors are non-zero by avoiding a finite set of bad characters $\chi$, and the $L$-values $L(f, \chi^{-1}, k+1)$ are non-zero for almost all $\chi$ by Theorem 13.5 (indeed for all $\chi$ if $k > 0$).

   This leaves only the ``wild'' part of $Z_{1, N}(k+2, \chi^{-1}, k+1)$, which is the product of the finitely many Euler factors at primes $\ell \ne p$ dividing $mN$. Each such factor is equal to a nonzero polynomial with coefficients in the Hecke algebra, evaluated at $\chi(\ell)$; for all but finitely many $\chi$ of $p$-power conductor this polynomial will not vanish.

   A similar analysis applies to the factor $Z_{1, N}(k+2, \nu^{-1}, 1)$, using the functional equation of the completed $L$-function relating values at $s = 1$ and $s = k+1$.
  \end{proof}

  \begin{corollary}
   \label{cor:kato-nonvanishing}
   For any $k \ge 0$, $N \ge 1$, and prime $p$, the subspace of $M_{k+2}(\Gamma_1(N), \CC)$ spanned by the modular forms ${}_{c, d} z_{1, N, m}(k+2, k+1, a(A))$, for $p$-power values of $m$ and $A$, and their translates under the Hecke operators $T_n'$, contains the subspace $S_{k+2}(\Gamma_1(N), \CC)$ of cusp forms.
  \end{corollary}

 \subsection{The compatibility}

  We now put the pieces together. Our starting point is the following theorem:

  \begin{theorem}[{Ohta, \cite[Theorem 2.1.11]{Ohta-ordinaryII}}]
   \label{thm:ohtawt2}
   For any $N \ge 1$, there is an isomorphism
   \[ \operatorname{Oh}: e'_{\ord} \mathfrak{M}_2'(N, \Zp) \rTo \DD(\sF^- H^1_{\ord}(Np^\infty)) \]
   such that for each $r \ge 1$ the following diagram commutes:
   \begin{diagram}
     e'_{\ord} \mathfrak{M}_2'(N, \Zp) & \rTo^{\operatorname{Oh}}_\cong & \DD(\sF^- H^1_{\ord}(Np^\infty))\\
    \dTo & & \dTo \\
     M_2'(Np^r, \Zp)& \rTo^{\pr_{\sF^-} \circ \comp_{\dR}}_{\cong} & \DD\left(\sF^- e'_{\ord} H^1(Y_1(Np^r)_{\overline{\QQ}_p}, \Zp(1))\right)
   \end{diagram}
   where the vertical arrows are the projections to the $r$-th level of the inverse limits, and the bottom horizontal arrow is given by the comparison isomorphism $\comp_{\dR}$ and projection to $\sF^-$.
  \end{theorem}

  Our goal is to extend this interpolating property of the map $\operatorname{Oh}$ to higher weights.  Note that for any $k \ge 0$, we have an isomorphism
  \begin{equation}
   \label{eq:indepofweight}
   e'_{\ord} \mathfrak{M}_{2}'(N, \Zp) \cong e'_{\ord} \mathfrak{M}_{k + 2}'(N, \Zp)
  \end{equation}
  for every $k \ge 0$, characterised by the compatibility with the isomorphism between both sides and $M^{\ord}_{\k+2}(N, \Lambda_D)$ via the maps $\cW_{0}$, $\cW_{k}$ of \ref{lemma:lambdaadicmf}.

  \begin{theorem}
   \label{thm:ohtacompat}
   For every $k \ge 0$ and $r \ge 1$, the diagram
   \begin{diagram}
    e'_{\ord} \mathfrak{M}_2'(N, \Zp) & \rTo^{\operatorname{Oh}}_\cong & \DD(\sF^- H^1_{\ord}(Np^\infty))\\
    \dTo & & \dTo \\
     M_{k+2}'(Np^r, \Zp) & \rTo^{\pr_{\sF^-} \circ \comp_{\dR}}_\cong & \DD\left(\sF^-  e'_{\ord} H^1(Y_1(Np^r)_{\overline{\QQ}_p}, \TSym^k(\sH_{\Zp})(1))\right)
   \end{diagram}
   commutes modulo the Eisenstein subspace of $M_{k+2}'(Np^r, \Zp)$. Here the right-hand vertical arrow is given by the moment map $\mom^k$, and the left-hand vertical arrow is given by the isomorphism \eqref{eq:indepofweight} and projection to the $r$-th term in the inverse limit.
  \end{theorem}

  \begin{remark}
   It is clear from Ohta's work that the map $\operatorname{Oh}$ induces \emph{some} isomorphism between the spaces in the bottom row of the above diagram, at least after restricting to cuspidal parts; this is parts (a) and (a*) of the Corollary on p51 of \cite{ohta95}. The novel content of the above theorem is that if $k > 0$ then this map coincides with a second, very differently defined isomorphism between these two spaces after inverting $p$: the one given by the comparison isomorphism $\comp_{\dR}$, which is defined by applying the Faltings--Tsuji comparison isomorphism to the $(k+1)$-dimensional Kuga--Sato variety over $Y_1(Np^r)$. This second map does not appear in Ohta's work (which uses only the theory of $p$-divisible groups, rather than more general results in $p$-adic Hodge theory).
  \end{remark}

  \begin{proof}
   We are trying to show commutativity of the diagram
   \begin{diagram}
    M^{\ord}_{\k + 2}(N, \Lambda_D) &\rTo^{\cW_0^{-1}}_{\cong} & \mathfrak{M}_2'(N, \Zp)
     & \rTo^{\operatorname{Oh}} & \DD(\sF^- H^1_{\ord}(Np^\infty))\\
    & \rdTo_{\pr_{r} \circ \cW_{k}^{-1}} & \dTo & & \dTo^{\pr_{k, r}} \\
    &&  M_{k+2}'(Np^r, \Zp)& \rTo^{\pr_{\sF^-} \circ \comp_{\dR}} & \DD(\sF^-  e'_{\ord} H^1(Y_1(Np^r), \TSym^{k}(\sH_{\Zp})(1)))
   \end{diagram}
   modulo the Eisenstein subspace. We shall in fact show commutativity of the diagram after taking the completed tensor product over $\Zp$ with the module $I^{-1} \Lambda_{\Gamma} \otimes \cO_E$, for a finite extension $E / \Qp$. This is clearly sufficient.

   Let $\nu$ be a nontrivial Dirichlet character of $p$-power conductor, and write
   \[ \mathbf{z}_\nu = \sum_{a} \nu(a)^{-1} {}_{c, d} \mathcal{BK}_{N}(a(A)).\]
   Theorem \ref{thm:ALaction} in the case $k = 0$ shows that
   \[ \cL\left( \mathbf{z}_\nu \right) = (\operatorname{Oh}\ \circ\ \cW_0^{-1})({}_{c, d} \cF_{N, \nu}) \]
   where ${}_{c, d} \cF_{N, \nu} \in M_{\k + 2}^{\ord}(N, \Lambda_D) \htimes \Lambda_{\Gamma} \otimes \cO_E$ is the $\Lambda$-adic modular form from Corollary \ref{cor:wtkohta}. On the other hand, the case $k > 0$ of Theorem \ref{thm:ALaction} shows that
   \begin{align*}
    \pr_{k, r} \cL(\mathbf{z}_\nu) &= \left(\pr_{\sF^-} \circ \comp_{\dR}\right)({}_{c, d} h_{k, r, \nu})\\
    &= \left( \pr_{\sF^-} \circ \comp_{\dR} \circ \pr_{r} \circ \cW_k^{-1}\right) \left({}_{c, d} \cF_{N, \nu} \right)
   \end{align*}
   for the same $\Lambda$-adic modular form ${}_{c, d} \cF_{N, \nu}$. This shows that the morphisms
   \[
    \pr_{\sF^-} \circ \comp_{\dR} \circ \pr_{r} \circ \cW_k^{-1}
   \]
   and
   \[ \pr_{k, r} \circ \operatorname{Oh} \circ \cW_0^{-1} \]
   agree on the submodule of $\frac{M^{\ord}(N, \Lambda_D)}{I_{k, r} M^{\ord}(N, \Lambda_D)} \htimes \Frac \Lambda_{\Gamma} \cong M^{\ord}_{k+2}(Np^r, \Zp) \htimes \Frac \Lambda_{\Gamma}$ generated by the image of the form ${}_{c, d} \cF_{N, \nu}$, for every Dirichlet character $\nu$. Since both maps are Hecke-equivariant, this implies that they agree on the subspace spanned by the image of ${}_{c, d} \cF_{N, \nu}$ under all Hecke operators. However, we can always choose $\nu$ such that this span contains all ordinary cusp forms, by Corollary \ref{cor:kato-nonvanishing} above. This completes the proof.
  \end{proof}

\section{Proof of Theorem B}

 \subsection{Lambda-adic differentials attached to Hida families}

  We now give a reformulation of Theorem \ref{thm:ohtacompat} which is more convenient for our present purposes. Recall that if $f$ is a normalised newform of some level $N$, then its image under the Atkin--Lehner operator $W_N$ is a scalar multiple of the conjugate eigenform $f^*$, and we define the Atkin--Lehner pseudo-eigenvalue $\lambda_N(f)$ by $W_N(f) = \lambda_N(f) f^*$.

  \begin{proposition}
   \label{prop:etaomegainterp}
   Let $\bff$ be a Hida family of tame level $N$.
   \begin{enumerate}
    \item There is a canonical isomorphism of $\Lambda_{\bff}$-modules
    \[ \omega_{\bff}: \DD\left(\sF^+ M(\bff)^*(-1-\k-\varepsilon_\bff)\right) \to \Lambda_{\bff}^{\mathrm{cusp}}, \]
    where $\Lambda_{\bff}^{\mathrm{cusp}}$ is the quotient of $\Lambda_{\bff}$ acting faithfully on cuspidal $\Lambda$-adic modular forms, with the following interpolation property: for every cuspidal specialisation $f$ of $\bff$, of some weight $k + 2 \ge 2$ and level $Np^r$ ($r \ge 1$), the map
    \[ \DD\left(\sF^+ M_{\Lp}(f)^*(-1-k-\varepsilon_f)\right) \to \Lp\]
    obtained by specialising $\omega_{\bff}$ coincides with that given by pairing with the differential
    \[ \omega_f \in \Fil^1 M_{\dR}(f) \otimes_{\QQ} \QQ(\mu_{Np^r}) \]
    attached to the normalised eigenform $f$. If $r = 1$ and $f$ is the ordinary $p$-stabilisation of an eigenform $f_0$ of level $N$, then $\omega_{f} = (\Pr^\alpha)^*(\omega_{f_0})$, where $(\Pr^\alpha)^*$ denotes the isomorphism $M_{\dR}(f_0) \to M_{\dR}(f)$ given by $\pr_1^* - \frac{\beta}{p^{k + 1}} \pr_2^*$ (the dual of the map $(\Pr^\alpha)_*$ appearing in \S \ref{sect:hidaspec} above)

    \item Let $\bfa$ be a new, cuspidal branch of $\bff$, and let $I_{\bfa}$ be the associated congruence ideal. Then there is a morphism of $\Lambda_{\bfa}$-modules
    \[ \eta_{\bfa}: \DD\left(\sF^- M(\bff)^*\right) \otimes_{\Lambda_{\bff}} \Lambda_\bfa \rTo I_{\bfa}, \]
    with the following interpolation property: for every arithmetic prime $\pi$ of $\Lambda_\bff$ above $\bfa$, corresponding to an eigenform $f$ of some level $Np^r$ ($r \ge 1$), then $I_\bfa \subseteq (\Lambda_{\bfa})_\pi$ and we have a commutative diagram
    \begin{diagram}
     \DD\left(\sF^- M(\bff)^*\right) \otimes_{\Lambda_{\bff}} \Lambda_\bfa & \rTo^{\eta_{\bfa}} & I_\bfa \\
     \dTo & & \dTo \\
     \Fil^0 M_{\dR, \Lp}(f)^* & \rTo & \Lp
    \end{diagram}
    where the vertical arrows are given by reduction modulo $\pi$, and the bottom horizontal arrow is as follows.
    \begin{enumerate}
     \item If $f$ is new of level $Np^r$, then the bottom horizontal arrow is given by pairing with the class
     \[ \frac{\alpha^r}{\lambda_{Np^r}(f)} \cdot \eta_{f}, \]
     where $\alpha$ is the $U_p$-eigenvalue of $f$, $\lambda_{Np^r}(f)$ is its Atkin--Lehner pseudo-eigenvalue, and $\eta_f \in \tfrac{M_{\dR}(f)}{\Fil^1} \otimes_{\QQ} \QQ(\mu_{Np^r})$ is the unique class which pairs to 1 with $\omega_{f^*}$.
     \item If $r = 1$ and $f$ is the ordinary $p$-stabilisation of a newform $f_0$ of level $N$, then this arrow is given by pairing with the class
     \[ \frac{1}{\lambda_N(f_0) \cE(f_0) \cE^*(f_0)} \cdot (\Pr^\alpha)^*(\eta_{f_0}), \]
     where $\eta_{f_0} \in \tfrac{M_{\dR}(f_0)}{\Fil^1} \otimes_{\QQ} \QQ(\mu_{N})$ again denotes the unique class pairing to 1 with $\omega_{f_0^*}$, and $\cE(f_0) = \left(1 -\frac\beta{p\alpha}\right)$, $\cE^*(f_0) = \left(1 - \frac{\beta}{\alpha}\right)$ (as in Theorem \ref{thm:hida}).
    \end{enumerate}

   \end{enumerate}
  \end{proposition}

  \begin{proof}

   To construct $\omega_{\bff}$, we use the fact (shown in \cite{Ohta-ordinaryII}) that Ohta's isomorphism $\operatorname{Oh}$ restricts to an isomorphism
   \[ \DD\left(\sF^- M^{\mathrm{cusp}}(\bff)^*\right) \to S^{\ord}(N, \Lambda_D)_{\bff} \]
   where $S^{\ord}(N, \Lambda_D) \subseteq M^{\ord}(N, \Lambda_D)$ is the submodule of $\Lambda_D$-adic cusp forms, and $M^{\mathrm{cusp}}(\bff)^*$ is the analogue of $M(\bff)^*$ formed using the cohomology of the curves $X_1(Np^r)$ rather than $Y_1(Np^r)$.

   We now apply the functor $\Hom_{\Lambda_D}(-, \Lambda_D)$ to both sides. On the one hand, $\Hom_{\Lambda_D}(S^{\ord}(N, \Lambda_D)_{\bff}, \Lambda_D)$ is canonically isomorphic to $\Lambda^{\mathrm{cusp}}_{\bff}$, via the usual pairing $(T, \cF) \to a_1(T \cdot \cF)$. On the other hand, Ohta's pairing (Theorem \ref{thm:ordinaryfiltration}(v)) gives us an isomorphism
   \[ \Hom_{\Lambda_D}\left(\sF^- M^{\mathrm{cusp}}(\bff)^*, \Lambda_D\right) \cong \sF^+ M(\bff)^*(-1-\k-\varepsilon_\bff), \]
   where $\varepsilon_{\bff}$ is the prime-to-$p$ part of the character of $\bff$.

   After unravelling the definitions (using the fact that the duality pairing involves the same factor $W_{Np^r}^{-1} \circ (U_p')^r$ that appears in the construction of the Ohta isomorphism $\operatorname{Oh}$), we find that the claimed interpolating property of $\omega_\bff$ is exactly Theorem \ref{thm:ohtacompat}. In the case where $f$ is a $p$-stabilisation, the relation $\omega_{f} = (\Pr^\alpha)^*(\omega_{f_0})$ is clear from a $q$-expansion computation, since $\pr_1^*$ acts as the identity on $q$-expansions, while $\pr_2^*$ sends $\sum a_n q^n$ to $p^{k + 1} \sum a_n q^{np}$.

   We now construct $\eta_{\bfa}$. The construction is virtually immediate: the Ohta isomorphism shows that $\DD\left(\sF^- M(\bff)^*\right)$ is isomorphic to a space of $\Lambda_D$-adic cusp forms, and after tensoring with $\Frac \Lambda_{\bfa}$, the eigenspace corresponding to the eigenform $\sum T_n q^n \in \Lambda_{\bfa}[[q]]$ splits off as a direct summand, so there is a unique map to $\Frac \Lambda_{\bfa}$ which sends this eigenform to 1.

   It remains to check the interpolating property. From Theorem \ref{thm:ohtacompat}, the map $\eta_{\bfa}$ is compatible with the map $M_{\dR, L}(f)^* \to L$ given by the composition
   \[ \Fil^0 M_{\dR, L}(f)^* \rTo^{ U_p^r \circ W_{Np^r}^{-1}} S_{k+2}(N, L) \to L \]
   where the last map sends the normalised eigenform $f$ to $1$.

   If $f$ is new of level $Np^r$, $r \ge 1$, then $f^*$ is a generator of $\Fil^0 M_{\dR, L}(f)^*$, and since we have $W_{Np^r}(f) = \lambda_{Np^r}(f) f^*$, we obtain the above formula.

   If $r = 1$ and $f$ is the $p$-stabilisation of $f_0$, then $\Fil^0 M_{\dR, L}(f)^*$ is generated as an $L$-vector space by $W_{Np}(f)$. This is obviously sent to $\alpha f$ by the map $U_p \circ W_{Np}^{-1}$, so its image under the map obtained by specialising $\eta_{\bfa}$ is $\alpha$; on the other hand, we have
   \[ \left[ (\Pr_\alpha)_* \circ W_{Np}\right] (f) = \left[ (\Pr_\alpha)_* \circ W_{Np}\circ (\Pr^\alpha)^*\right](f_0), \]
   and a computation using the identities
   \begin{align*}
    W_{Np} \circ \pr_1^* &= \pr_2^* \circ W_N, \\
    W_{Np} \circ \pr_2^* &= p^k \pr_1^* \circ W_N
   \end{align*}
   shows that $(\Pr^\alpha)_* \circ W_{Np}\circ (\Pr^\alpha)^*$ acts on the $f_0$-eigenspace as multiplication by
   \[\alpha \left(1 -\tfrac\beta\alpha\right)\left(1 - \tfrac{\beta}{p\alpha}\right)  \lambda_N(f_0).\]
   Comparing these gives the interpolating property in (b).
  \end{proof}

  The presence of the Atkin--Lehner pseudo-eigenvalues $\lambda_N(f_0)$ in the last case of the theorem is not a problem for us, since they can be interpolated $p$-adically:

  \begin{proposition}
   There is a $\Lambda_D[\zeta_N]$-linear operator $W_N$ on $M^{\ord}(N, \Lambda_D[\zeta_N])$, with the property that for every arithmetic prime ideal $\nu = (k, \omega)$ of $\Lambda_D$, the resulting operator on the space
   \[
    M^{\ord}(N, \Lambda_D[\zeta_N]) \otimes_{\Lambda, \nu} \cO \cong e_{\ord} M_{k + 2}(Np^r, \omega, \cO[\zeta_N])
   \]
   is the usual Atkin--Lehner operator $W_N$.
  \end{proposition}

  (See Note 5.4.1 of \cite{LLZ14}, but note that the field extension to $\zeta_N$ was inadvertently omitted there.) It follows that for each new, cuspidal branch $\bfa$ of $\bff$, there is a $\Lambda$-adic pseudo-eigenvalue $\lambda_N(\bfa) \in \left(\Lambda_{\bfa} \otimes_{\QQ} \QQ(\mu_N)\right)^\times$, satisfying $\lambda_N(\bfa)^2 = (-N)^{\k}$, whose image under specialisation at any $p$-stabilised eigenform $f$ is equal to $\lambda_N(f_0)$.


 \subsection{Proof of the theorem}

  We now have all the necessary ingredients to complete the proof of Theorem \ref{lthm:explicitrecip}.

  Let $\bff$, $\bfg$ be two Hida families of tame levels $N_f, N_g \mid N$, and recall the space $\DD(\sF^{-+} M(\bff \htimes \bfg)^*) \htimes \Lambda_{\Gamma}$ of \S \ref{sect:PRreg} above. Choose a new, cuspidal branch $\bfa$ of $\bff$. Then pairing with $\eta_{\bfa} \otimes \omega_{\bfg}$ gives a map of $\left(\Lambda_{\bff} \htimes \Lambda_{\bfa} \htimes \Lambda_{\Gamma}\right)$-modules
  \[ \langle -, \eta_{\bfa} \otimes \omega_{\bfg} \rangle: \DD(\sF^{-+} M(\bff \htimes \bfg)^*) \htimes \Lambda_{\Gamma} \rTo \left(I_{\bfa} \htimes \Lambda_{\bfg}^{\mathrm{cusp}} \htimes \Lambda_{\Gamma}\right) \otimes_{\Zp} \Zp[\mu_N]. \]

  \begin{remark}
   The $\Zp[\mu_N]$ factor appears because $\omega_{\bfg}$ is a linear functional on $\DD(\sF^+ M(\bfg)^*(-1-\k-\varepsilon_{\bfg}))$, while it is $M(\bfg)^*(-1-\k)$ appearing in the definition of $\DD_\infty(M)$. Since $\varepsilon_{\bfg}$ has conductor dividing $N$, we have $\DD(\Zp(\varepsilon_{\bfg})) \subseteq \Zp[\mu_N]$.
  \end{remark}

  Recall our convention that $\k$ and $\k'$ denote the canonical characters into the two $\Lambda_D$ factors of the ring $\Lambda_D \htimes \Lambda_D \htimes \Lambda_{\Gamma}$; we shall write $\j$ for the canonical character into the $\Lambda_{\Gamma}$ factor.

  \begin{theorem}[Theorem \ref{lthm:explicitrecip}]
   \label{thm:explicitrecip}
   We have
   \begin{equation}
    \label{eq:thmB}
    \left\langle \cL\left(\cBF^{\bff, \bfg}\right), \eta_{\bfa} \otimes \omega_{\bfg} \right\rangle = \lambda_N(\bfa)^{-1} (-1)^{1 + \j} \left(c^2 - c^{-(\k + \k' - 2\j)} \varepsilon_\bff(c)^{-1} \varepsilon_\bfg(c)^{-1} \right)  L_p(\bfa, \bfg, 1 + \j).
   \end{equation}
  \end{theorem}

  \begin{proof}
   Since the module $I_\bfa \htimes \Lambda_\bfg^{\mathrm{cusp}} \htimes \Lambda_\Gamma$ is a torsion-free module over $\Lambda_{\bfa} \htimes \Lambda_\bfg^{\mathrm{cusp}} \htimes \Lambda_{\Gamma}$, it suffices to prove that the two sides of \eqref{eq:thmB} agree modulo $Q$ for a Zariski-dense set of primes $Q$ of this ring. We choose the set of primes $Q$ corresponding to triples $(f, g, j)$, where $j \in \ZZ$, and $f$ and $g$ are $p$-stabilizations of cusp forms $f_0, g_0$ of levels $N_f, N_g$ coprime to $p$, and any weights $k + 2, k' + 2$, such that $0 \le j \le \min(k, k')$ and we do not have $j = k = k'$. It is clear that this set $Q$ is indeed Zariski-dense in $\Spec\left(\Lambda_{\bfa} \htimes \Lambda_{\bfg} \htimes \Lambda_{\Gamma}\right)$.

   So, let $P = (f, g, j)$ be such a point. To save ink, let us write $\nu_c$ for the factor $(c^2 - c^{2j-k-k'} \varepsilon_f(c)^{-1} \varepsilon_g(c)^{-1})$, and $\lambda_f$ for the Atkin--Lehner pseudo-eigenvalue $\lambda_{N_f}(f_0)$. Then the value of the right-hand side of \eqref{eq:thmB} at $(f, g, j)$ (i.e.\ its image in the residue field of $P$) is
   \[
    (-1)^{1+j}\lambda_f^{-1}\nu_c \cdot L_p(f_0, g_0, 1 + j)\tag{\dag}
   \]
   by the interpolating property of the 3-variable Rankin--Selberg $L$-function $L_p(\bfa, \bfg)$ (Theorem \ref{thm:hida2}).

   We now compute the value of the left-hand side of \eqref{eq:thmB} at $P$. By Proposition \ref{prop:etaomegainterp}, the image of the pairing $\left\langle \cL\left(\cBF^{\bff, \bfg}\right), \eta_{\bfa} \otimes \omega_{\bfg} \right\rangle$ under evaluation at $P$ is given by
   \[ \frac{1}{\lambda_f \cE(f) \cE^*(f)} \left\langle \cL\left(\cBF^{\bff, \bfg}\right) \bmod P, (\Pr^\alpha \times \Pr^\alpha)^*\left(\eta_{f_0}^\alpha \otimes \omega_{g_0}\right) \right\rangle. \]
   (Note that this step is far from being purely formal, despite its near-tauto\-logical appearance; for $(k, k') \ne (0, 0)$ it relies crucially on the extension of Ohta's results developed in \S \ref{sect:ohtacompat} above.)

   Theorem \ref{thm:bigloginterp} tells us that
   \[\cL\left(\cBF^{\bff, \bfg}\right) \bmod P = \frac{(-1)^{k' - j}}{(k' - j)!} \cdot \frac{\left( 1 - \frac{p^j}{\alpha_f \beta_g}\right)}{\left( 1 - \frac{\alpha_f \beta_g}{p^{j + 1}}\right)} \cdot \log\left( \cBF^{\bff, \bfg} \bmod P\right),\]
   and Theorem \ref{thm:BFeltsinterp2} gives
   \[ \cBF^{\bff, \bfg} \bmod P = \frac{\left(1 - \tfrac{p^j}{\alpha_f \alpha_g}\right)\nu_c}{(-1)^j j! \binom{k}{j} \binom{k'}{j}} \Eis^{[f, g, j]}_{\et}.\]
   Combining the last three steps, we obtain the formula
   \begin{multline*}
    \left\langle \cL\left(\cBF^{\bff, \bfg}\right), \eta_{\bfa} \otimes \omega_{\bfg} \right\rangle \bmod P = \frac{(-1)^{k'} \nu_c\left(1 - \tfrac{p^j}{\alpha_f \alpha_g}\right) \left( 1 - \tfrac{p^j}{\alpha_f \beta_g}\right)}{k'! \binom{k}{j}\lambda_f \cE(f)\cE^*(f)\left( 1 - \tfrac{\alpha_f \beta_g}{p^{j+1}}\right)}\\ \times \left\langle \log\left(\Eis^{[f, g, j]}_{\et}\right),(\Pr^\alpha \times \Pr^\alpha)^*\left(\eta_{f_0}^\alpha \otimes \omega_{g_0}\right) \right\rangle.
   \end{multline*}

   By definition, the map $(\Pr^\alpha \times \Pr^\alpha)^*$ is the transpose of the map $(\Pr^\alpha \times \Pr^\alpha)_*$, and the comparison isomorphism $\comp_{\dR}$ commutes with the action of correspondences; so we may write the last term as
   \begin{align*} \left\langle (\Pr^\alpha \times \Pr^\alpha)_*\left(\log \Eis^{[f,g, j]}_{\et}\right),\eta_{f_0}^\alpha \otimes \omega_{g_0}\right\rangle
   &= \left\langle \log\left((\Pr^\alpha \times \Pr^\alpha)_* \Eis^{[f, g, j]}_{\et}\right),\eta_{f_0}^\alpha \otimes \omega_{g_0} \right\rangle\\
   &= \left( 1 - \tfrac{\alpha_f \beta_g}{p^{j+1}}\right)\left( 1 - \tfrac{\beta_f \alpha_g}{p^{j+1}}\right)\left( 1 - \tfrac{\beta_f \beta_g}{p^{j+1}}\right) \left\langle \log \Eis^{[f_0, g_0, j]}_{\et},\eta_{f_0}^\alpha \otimes \omega_{g_0} \right\rangle,
   \end{align*}
   using Theorem \ref{thm:lstab-eigen}. Thus we have shown that
   \[ \left\langle \cL\left(\cBF^{\bff, \bfg}\right), \eta_{\bfa} \otimes \omega_{\bfg} \right\rangle \bmod P = \frac{(-1)^{k'} \nu_c\, \cE(f_0, g_0, 1 + j)}{k'! \binom{k}{j}\lambda_f \cE(f_0)\cE^*(f_0)} \left\langle \log \Eis^{[f_0, g_0, j]}_{\et},\eta_{f_0}^\alpha \otimes \omega_{g_0} \right\rangle,\]
   where $\cE(f_0, g_0, 1 + j)$ is as in Theorem \ref{thm:hida}.

   Comparing the last formula with $(\dag)$, we see that the left and right sides of \eqref{eq:thmB} agree modulo $P$ if and only if
   \[ \left\langle \log \Eis^{[f_0, g_0, j]}_{\et},\eta_{f_0}^\alpha \otimes \omega_{g_0} \right\rangle =
   \frac{(-1)^{k'-j+1}\binom{k}{j}k'! \cE(f_0)\cE^*(f_0)}{\cE(f_0,g_0,1+j)} L_p(f_0, g_0, 1 + j).\]
   This is exactly the formula of Theorem \ref{thm:syntomicreg}, so we are done.
  \end{proof}

  \begin{remark}
   A special case of this theorem (for $g$ varying in a one-variable family, with $f$ a fixed weight 2 form and $j = 0$) has been proved by Bertolini, Darmon and Rotger \cite{BDR-BeilinsonFlach2}. Their method is rather different from ours, involving analytic continuation from highly ramified weight 2 points, rather than crystalline points of high weight as in the above argument. (In place of Theorem \ref{thm:syntomicreg}, they use a formula for syntomic regulators of weight 2 Rankin--Eisenstein classes on modular curves of high $p$-power level, based on work of Amnon Besser and two of the present authors \cite{besserloefflerzerbes16}.)
  \end{remark}


\section{Arithmetic applications: Bounding Selmer groups}
 \label{sect:selmerbound}

 \subsection{Hypotheses}
  \label{sect:hypotheses}

  In order to bound the Selmer groups of Rankin convolutions, we shall need to impose a number of technical hypotheses. The aim of this section is to introduce and define these.

  In this section, $f$ and $g$ are newforms (of some levels $N_f, N_g$, and any weights $r, r' \ge 1$), $L$ is a number field containing the coefficients of $f$ and $g$, and $\frP$ is a prime of $L$ above the rational prime $p$.

  The following hypothesis will be \emph{assumed} throughout section \ref{sect:selmerbound}:

  \begin{hypothesis}
   \mbox{~}
   \begin{itemize}
    \item The weights $r$ and $r'$ are not both equal to 1.
    \item The prime $p$ is $\ge 5$, and $p \nmid N_f N_g$.
    \item The forms $f$ and $g$ are ordinary at $\frP$, non-Eisenstein modulo $\frP$, and $p$-distinguished.
   \end{itemize}
  \end{hypothesis}

  We write $\alpha_f, \alpha_g$ for the unit roots (in $\Lp$) of the Hecke polynomials of $f$ and $g$. These are uniquely determined if $r, r' \ge 2$; in the weight 1 case both roots are units, and we choose one arbitrarily and denote it by $\alpha_f$ or $\alpha_g$ respectively. Then there are $p$-stabilisations of $f$ and $g$ with $U_p$-eigenvalues $\alpha_f$ and $\alpha_g$, and these are specialisations of some Hida families $\bff$ and $\bfg$.

  We write $\cO = \cO_{L, \frP}$, and we define $T = M_{\cO}(f)^* \otimes_{\cO} M_{\cO}(g)^*$, which is a free $\cO$-module of rank 4. Using the map $(\Pr^\alpha \times \Pr^\alpha)_*$ introduced in \S \ref{sect:lstab}, we can identify $T$ with a specialisation of the $\Lambda$-adic module $M(\bff \otimes \bfg)^*$, so it inherits filtration subspaces $\sF^{++} T$ etc.

  As well as these running hypotheses, we also state some other hypotheses which are needed in the Euler system argument. These will not be assumed implicitly, but rather will be stated explicitly when needed.

  \begin{hypothesis} [$\Hyp(BI)$, for ``big image'']\mbox{~}
   \begin{enumerate}[(i)]
    \item $T /\frP T$ is irreducible as a $\Gal(\overline{\QQ} / \QQ(\mu_{p^\infty}))$-module.
    \item There exists an element $\tau \in \Gal(\overline{\QQ} / \QQ(\mu_{p^\infty}))$ such that $T / (\tau - 1) T$ is free of rank one over $\cO$.
    \item There exists an element $\sigma \in \Gal(\overline{\QQ} / \QQ(\mu_{p^\infty}))$ which acts on $T$ as multiplication by $-1$.
   \end{enumerate}
  \end{hypothesis}

  Parts (i) and (ii) of $\Hyp(BI)$ are the hypothesis $\Hyp(K_\infty, T)$ of \cite{Rubin-Euler-systems}. The role of (iii) is to kill off the ``error terms'' $\mathfrak{n}_W$ and $\mathfrak{n}^*_W$ appearing in Theorem 2.2.2 of \emph{op.cit.}.

  \begin{remark}
   It is shown in the paper \cite{Loeffler-big-image} that Hypothesis $\Hyp(BI)$ can only be satisfied when $\varepsilon_f \varepsilon_g$ is nontrivial, but it is often satisfied if this occurs. In particular, in any of the following situations, $\Hyp(BI)$ is satisfied for all but finitely many primes $\frP$ of the coefficient field:
   \begin{itemize}
    \item if $(N_f, N_g) = 1$, $f$ and $g$ both have weight $\ge 2$, neither is of CM type, and $g$ has odd weight;
    \item if $(N_f, N_g) = 1$, $f$ and $g$ both have weight $\ge 2$, $f$ is not of CM type, $g$ is of CM type, and $\varepsilon_g$ is not either 1 or the quadratic character attached to the CM field;
    \item if $(N_f, N_g) = 1$, $f$ has weight $\ge 2$ and is not of CM type, and $g$ has weight 1.
   \end{itemize}

   (The existence of the element $\sigma$ is not mentioned explicitly in \cite{Loeffler-big-image}; but the arguments of \emph{op.cit.} show that in each of the above cases the image of $\Gal(\overline{\QQ} / \QQ(\mu_{p^\infty}))$ acting on $M_{\cO}(f) \oplus M_{\cO}(g)$ contains a conjugate of $\SL_2(\Zp) \times \{1\}$, and we simply take $\sigma$ to be any element acting as $-1$ on $M_{\cO}(f)$ and trivially on $M_{\cO}(g)$.)
  \end{remark}

  We also define the following purely local hypothesis:

  \begin{hypothesis}[$\Hyp(NEZ)$, for ``no exceptional zero'']
   Neither $\alpha_f \beta_g$ nor $\beta_f \alpha_g$ is a power of $p$.
  \end{hypothesis}

  Note that $\Hyp(NEZ)$ is automatic if $r \ne r'$, since in this case $\alpha_f \beta_g$ and $\beta_f \alpha_g$ have different $\frP$-adic valuations from their complex conjugates and hence cannot be in $\QQ$. The same reasoning also shows that $\alpha_f \alpha_g$ and $\beta_f \beta_g$ are never powers of $p$.

 \subsection{Generalities on Selmer complexes}
  \label{sect:selgeneralities}

  We now recall some ideas from Nekov\v{a}\'r's theory of Selmer complexes. Let $R$ be a commutative Noetherian complete local ring, with finite residue field of characteristic $p \ne 2$; let $K$ be a number field; and let $S$ be a finite set of primes of $K$ including all places above $p$.

  We write $G_{K, S}$ for the Galois group of the maximal extension of $K$ unramified outside $S$ and the infinite places (equivalently, the \'etale fundamental group of the $S$-integer ring $\cO_{K, S}$). For any finitely-generated $R$-module $M$ with a continuous action of $G_{K, S}$, we write $R\Gamma(\cO_{K, S}, M)$ for the class in the derived category of the complex $C^\bullet(G_{K, S}, M)$ of continuous $M$-valued cochains on $G_{K, S}$. Similarly we have local cohomology complexes $R\Gamma(K_v, M)$ for $v \in S$.

  \begin{proposition}[{Fukaya--Kato, see \cite[Proposition 1.6.5]{FukayaKato-formulation-conjectures}}]
   The complex $R\Gamma(\cO_{K, S}, M)$ is perfect, and it commutes with derived base-change, in the sense that if $f: R \to R'$ is a morphism of local rings, then we have
   \[ R\Gamma(\cO_{K, S}, R' \times_{R} M) = R' \otimes^{\mathbf{L}}_{R}  R\Gamma(\cO_{K, S}, M),\]
   where the $\otimes^{\mathbf{L}}_{R} $ denotes the derived tensor product. The same holds for the local cohomology complexes $R\Gamma(K_v, M)$.
  \end{proposition}

  \begin{definition}
   A \emph{local condition} for $M$ at $v$ consists of the data of a complex $U_v^+$ of $R$-modules and a homomorphism $i_v^+: U_v^+ \to C^\bullet(K_v, M)$. A \emph{Selmer structure} for $M$ is a collection $\Delta = (\Delta_v)_{v \in S}$, where $\Delta_v$ is a local condition.
  \end{definition}

  Local conditions of particular interest are
  \begin{itemize}
   \item the \emph{strict} local condition $U_v^+ = 0$;
   \item the \emph{relaxed} local condition $U_v^+ = C^\bullet(K_v, M)$;
   \item the \emph{unramified} local condition, given by
   \[ C^\bullet(\FF_v, M^{I_v}) \to C^\bullet(K_v, M);\]
   \item and the \emph{Greenberg} local condition, given by
   \[ C^\bullet(K_v, M_v^+) \to C^\bullet(K_v, M)\]
   for $M_v^+$ a submodule of $M$ stable under the decomposition group at $v$.
  \end{itemize}
  (The strict and relaxed local conditions are, of course, examples of Greenberg local conditions, by taking $M_v = 0$ and $M_v = M$ respectively.)

  Many, but not all, interesting local conditions are of the following form (cf.\ \cite[\S 6.1.4]{Nekovar-Selmer-complexes}):

  \begin{definition}
   We will say a local condition is \emph{simple} if the map $i_v^+: H^i(U_v^+) \to H^i(K_v, M)$ is an isomorphism for $i = 0$, injective for $i = 1$, and zero for $i = 2$. We say a Selmer structure $\Delta$ is simple if the local condition $\Delta_v$ is simple for all $v$.
   \end{definition}

   A simple local condition $\Delta_v$ is thus determined (up to quasi-isomor\-phism) by the subspace
   \[ H^1_{\Delta_v}(K_v, M) \coloneqq \iota_v^+(H^1(U_v^+)) \subseteq H^1(K_v, M).\]
   The unramified local condition is simple (while the strict and relaxed local conditions usually are not).  Another important example of a simple local condition is the \emph{Bloch--Kato local condition}. To define this we must assume that $R$ is the ring of integers of a finite extension of $\Qp$, and if $v \mid p$ then also that $M[1/p]$ is de Rham at $v$. Then the Bloch--Kato local condition is the simple local condition attached to the Bloch--Kato subspace $H^1_{\mathrm{f}}(K_v, M)$. (Recall that $H^1_{\mathrm{f}}(K_v, M)$ is the saturation of $H^1(\FF_v, M^{I_v})$ in $H^1(K_v, M)$ if $v \nmid p$, and the crystalline classes if $v \mid p$).

  \begin{definition}
   If $\Delta$ is a Selmer structure for $M$, we define a Selmer complex $\RGt(\cO_{K, S}, M; \Delta)$ as in \cite[\S 6.1.2]{Nekovar-Selmer-complexes}, as the mapping fibre of
   \[ \left[ R\Gamma(\cO_{K, S}, M) \oplus \bigoplus_{v \in S} U_v^+ \rTo^{\loc_v - i_v^+} \bigoplus_{v \in S} R\Gamma(K_v, M)\right].\]
  \end{definition}

  If $\Delta \to \Delta'$ is a morphism of Selmer structures (i.e.\ a collection of morphisms $U_v^+ \to (U_v')^+$ commuting with the morphism to $C^\bullet(K_v, M)$), then we have an exact triangle
  \begin{equation}
   \label{eq:exacttri}
   \RGt(\cO_{K, S}, M; \Delta) \to \RGt(\cO_{K, S}, M; \Delta') \to \bigoplus_v Q_v \to \dots
  \end{equation}
  where $Q_v$ is the mapping fibre of $U_v^+ \to (U_v')^+$. The strict local condition and the relaxed local condition are respectively the initial and terminal objects in the category of local conditions, so we obtain as special cases the exact triangles
  \begin{subequations}
  \begin{align}
   \label{eq:exacttri1}
   \RGt(\cO_{K, S}, M; \Delta) \to R\Gamma(\cO_{K, S}, M) \to \bigoplus_{v \in S} U_v^- \to \dots,\\
   \label{eq:exacttri2}
   R\Gamma_c(\cO_{K, S}, M) \to \RGt(\cO_{K, S}, M; \Delta) \to \bigoplus_{v \in S} U_v^+ \to \dots
  \end{align}
  \end{subequations}
  where $R\Gamma_c(\cO_{K, S}, M)$ denotes the compactly--supported cohomology (the Selmer complex with the strict local conditions at all $v \in S$) and $U_v^-$ is the mapping cone of $U_v^+ \rTo^{-i_v^+} C^\bullet(K_v, M)$.

  The formation of the Selmer complexes is compatible with change of the coefficient ring $R$, in the following sense. For $R \to R'$ a homomorphism of rings satisfying our conditions above, we can write $M'$ for the tensor product $R' \otimes_{R} M$; and via derived tensor product we obtain local conditions $\Delta'$ for $M'$. It is then clear that
  \[ \RGt(\cO_{K, S}, M'; \Delta') = R' \otimes^{\mathbf{L}}_{R} \RGt(\cO_{K, S}, M; \Delta).\]

  \begin{remark}
   In the above setting, if $\Delta_v$ is the unramified local condition for $M$ at some place $v$, then it does not necessarily follow that $\Delta'_v$ is the unramified local condition for $M'$. We have $\Delta'_v = R\Gamma(\FF_v, R' \otimes_R M^{I_v})$, and the natural map $R' \otimes_R M^{I_v} \to (R' \otimes_{R} M)^{I_v}$ is not necessarily an isomorphism.
  \end{remark}

  We now consider duality for Selmer complexes. Let us denote by $M^\vee$ the Pontryagin dual $\Hom(M, \Qp/\Zp)$ of $M$.

  \begin{definition}
   We say two local conditions $\Delta_v$ for $M$ and $\Delta_v^\vee$ for $M^\vee(1)$ are \emph{orthogonal complements} if local Tate duality gives a quasi-iso\-mor\-phism
   \[ U_v^\pm \cong R\Hom\left((U_v^\vee)^{\mp}, \Qp/\Zp\right)[2].\]
  \end{definition}

  Note that the unramified local condition for $M^\vee(1)$ is the orthogonal complement of the unramified local condition for $M$; orthogonal complements of simple local conditions are simple; and the Greenberg local condition for a submodule $M_v^+ \subseteq M$ is the orthogonal complement of the Greenberg condition for $(M/M_v^+)^\vee(1) \subseteq M^\vee(1)$. We then have the following global duality result:

  \begin{theorem}[{\cite[Theorem 6.3.4]{Nekovar-Selmer-complexes}}]
   If $\Delta$ and $\Delta^\vee$ are Selmer structures on $M$ and $M^\vee(1)$ respectively which are orthogonal complements in the sense above, then we have an isomorphism in the derived category
   \[ \RGt(\cO_{K, S}, M^\vee(1); \Delta^\vee) = R\Hom\left(\RGt(\cO_{K, S}, M; \Delta), \Qp/\Zp\right)[3].\]
  \end{theorem}

  We will be particularly interested in a consequence of this:

  \begin{proposition}
   \label{prop:nekovar-duality}
   The kernel of $\Ht^2(\cO_{K, S}, M; \Delta) \to \bigoplus_v H^2(U_v^+)$ is isomorphic to the Pontryagin dual of the kernel of
   \[ H^1(\cO_{K, S}, M^\vee(1)) \rTo \bigoplus_{v \in S} H^1((U_v^\vee)^-).\]
  \end{proposition}

  If the Selmer structure $\Delta$ is simple, then using the long exact sequence associated to \eqref{eq:exacttri1} and the previous proposition, one has a complete description of the cohomology of the Selmer complex:

  \begin{proposition}
   If $\Delta$ is a simple Selmer structure, determined by subspaces $H^1_{\Delta_v}(K_v, M) \subseteq H^1(K_v, M)$ for $v \in S$, we have
   \[
    \Ht^i(\cO_{K, S}, M; \Delta) =
    \begin{cases}
     H^0(\cO_{K, S}, M) & \text{if $i = 0$,}\\
     \ker\left(H^1(\cO_{K, S}, M) \to \bigoplus_{v \in S} \dfrac{H^1(K_v, M)}{H^1_{\Delta_v}(K_v, M)} \right) & \text{if $i = 1$,}\\
     \ker\left(H^1(\cO_{K, S}, M^\vee(1)) \to \bigoplus_{v \in S} \dfrac{H^1(K_v, M^\vee(1))}{H^1_{\Delta_v^\vee}(K_v, M^\vee(1))} \right)^\vee & \text{if $i = 2$.}
    \end{cases}
   \]
   Here $H^1_{\Delta_v^\vee}(K_v, M^\vee(1))$ is the orthogonal complement of $H^1_{\Delta_v}(K_v, M)$ under local Tate duality.
  \end{proposition}

  Thus we recover the classical notion of a Selmer group, as a subspace of $H^1(\cO_{K, S}, M)$ cut out by local conditions.

  \begin{remark}
   One of the key insights of \cite{Nekovar-Selmer-complexes} is that -- even if one is ultimately only interested in classical Selmer groups -- the more general theory of Selmer complexes is much more convenient and flexible to work with, since Selmer complexes are well-behaved under operations  such as base-change and duality.
  \end{remark}

 \subsection{Definition of the local conditions}

  We now return to the case at hand: we let $f, g$ be two newforms with coefficients in some number field $L$ as in \S\ref{sect:hypotheses} above. For any place $\frP$ of $L$ above $p$, we have a four-dimensional $\cO$-linear Galois representation $T = M_{\cO}(f)^* \otimes M_{\cO}(g)^*$. Let $S$ be the set of primes dividing $p N_f N_g \infty$, so that $T$ is unramified outside $S$.

  \begin{definition}
   We define Selmer structures $\Delta^{(?)}$ on $T$, for $? \in \{f, g, \varnothing\}$, as follows:
   \begin{itemize}
    \item for $v \in S \setminus \{p\}$ (and any $?$), we let $\Delta^{(?)}_v$ be the unramified local condition;
    \item for $v = p$, we let $\Delta^{(?)}_v$ be the Greenberg local condition
    \[ C^\bullet(\Qp, T^{(?)}) \to C^\bullet(\Qp, T)\]
    where the $M^{(?)}$ are the $G_{\Qp}$-invariant submodules of $M$ given as follows:
    \begin{align*}
     T^{(f)} &= \sF^{+\circ} T = \sF^+ M_{\cO}(f)^* \otimes_{\cO} M_{\cO}(g)^*;\\
     T^{(g)} &= \sF^{\circ+} T = M_{\cO}(f)^* \otimes_{\cO} \sF^+ M_{\cO}(g)^*;\\
     T^{(\varnothing)} &= T^{(f)} + T^{(g)}.
    \end{align*}
   \end{itemize}
  \end{definition}

  We can also define, similarly, Selmer structures on $T(\tau^{-1})$, for any $\cO$-valued character $\tau$ of the group $\Gamma = \Gal(\QQ(\mu_{p^\infty}) / \QQ)$; or on the ``universal twist''
  \( T \otimes_{\Zp} \Lambda_{\Gamma}(-\j), \)
  where as before $\j$ denotes the canonical character $G_{\QQ} \to \Gamma \to \Lambda_\Gamma^\times$.

  \begin{proposition}
   If $\Delta$ is any of the above Selmer structures, then for any $\cO$-valued character $\tau$ of $\Gamma$, we have
   \[ \RGt(\ZZ[1/S], T(\tau^{-1}); \Delta) = \cO \otimes^{\mathbf{L}}_{\Lambda_{\Gamma}, \tau} \RGt(\ZZ[1/S],  T \otimes \Lambda_{\Gamma}(-\j); \Delta). \]
  \end{proposition}

  \begin{proof}
   It suffices to check that the formation of the local conditions $\Delta^{(?)}$ commutes with derived base-change, which is clear since the canonical character is unramified outside $p$.
  \end{proof}

 \subsection{Main conjectures ``without p-adic zeta-functions''}

  Let us write $\cBF^{f, g}_{m} \in H^1(\ZZ[1/S, \mu_m], T \otimes \Lambda_{\Gamma}(-\j))$ for the image in $T$ of the Beilinson--Flach class $\cBF^{\bff, \bfg}_m$.

  \begin{theorem}
   \label{thm:modifyBFelts}
   Fix an integer $c > 1$ coprime to $6p N_f N_g$. Then there exists a collection of elements
   \[ c_m \in H^1\left(\ZZ[1 / S, \mu_{m}], T \otimes \Lambda_{\Gamma}(-\j)\right)\]
   for all $m \ge 1$ coprime to $p c N_f N_g$, with $c_1 = \cBF^{f, g}_1$, such that we have the Euler system compatibility relation
   \[
    \norm^{\ell m}_m\left(c_{\ell m}\right)=
    \begin{cases}
     P_\ell(\ell^{-1} \sigma_\ell^{-1})\cdot c_m & \text{if $ \ell \nmid pm$},\\
     c_m & \text{if $\ell \mid p m$}.
    \end{cases}
   \]
   Here $P_\ell$ is the Euler factor of $M_{\Lp}(f \otimes g)$ at $\ell$. Furthermore, the localisation $\loc_p(c_m)$ lies in the image of the natural injection
   \[ H^1\left(\QQ(\mu_m) \otimes \Qp, T^{(\varnothing)} \otimes \Lambda_{\Gamma}(-\j)\right) \into H^1\left(\QQ(\mu_m) \otimes \Qp, T \otimes \Lambda_{\Gamma}(-\j)\right). \]
  \end{theorem}

  \begin{proof}
   We know from the first part of Theorem \ref{thm:BFeltsinterp2} that the elements $\cBF^{f, g}_{\ell m}$ satisfy an ``almost Euler system'' compatibility relation
   \[ \norm^{\ell m}_m\left(\cBF^{f, g}_{\ell m}\right) = -\sigma_\ell \cdot Q_\ell(\sigma_\ell^{-1})\cdot \cBF^{f, g}_{m}, \]
   where $Q_\ell$ is some polynomial congruent to $P_{\ell}$ modulo $\ell - 1$. As explained in \cite[\S 7.3]{LLZ14}, using \cite[Lemma 9.6.1]{Rubin-Euler-systems}, we can modify these classes by appropriate elements of $\Zp[\Gal(\QQ(\mu_{mp^\infty}) / \QQ)]^\times$ in such a way as to obtain the ``correct'' Euler system relation. This gives the classes $c_m$. Moreover, the class $\cBF^{f, g}_m$ vanishes after localisation at $p$ and projection to $\sF^{--}$, exactly as in the case $m = 1$ considered in Proposition \ref{prop:BFeltisSelmer}; hence the same is true of the modified element $c_m$.
  \end{proof}

  If $\Hyp(BI)$ holds, which we will assume from this point onwards, then we may get rid of the factor $c$: the hypothesis forces $\varepsilon_f \varepsilon_g$ to be non-trivial modulo $\frP$, so by \cite[Remark 6.8.11]{LLZ14} there exist classes $\BF^{f, g}_m$ such that $\cBF^{f, g}_m = (c^2 - c^{2\j - k - k'} \varepsilon_f(c) \varepsilon_g(c)) \BF^{f, g}_m$.

  Let $\mathcal{K}$ be the maximal abelian extension of $\QQ$ unramified at the primes dividing $c N_f N_g$. If $K$ is a finite extension of $\QQ$ contained in $\mathcal{K}$, we define a class $c_K \in H^1(\cO_{K, S}, T)$ as the image of $c_m$ under the corestriction map, for a suitable integer $m$ such that $K \subseteq \QQ(\mu_{mp^\infty})$ and every prime dividing $m$ is ramified in $K$. Then the collection $\mathbf{c} = (c_K)$ is an Euler system for $(T, \mathcal{K}, c N_f N_g p)$ in the sense of \cite[Definition 2.1.1]{Rubin-Euler-systems}.

  \begin{remark}
   The definition of an Euler system in \cite{MazurRubin-Kolyvagin-systems} is actually slightly different from this, as the Euler factors are $P_\ell(\sigma_\ell^{-1})$ rather than $P_\ell(\ell^{-1} \sigma_\ell^{-1})$. As noted in Remark 3.2.3 of \emph{op.cit.} the theory of \S 9.6 of \cite{Rubin-Euler-systems} allows one to easily switch back and forth between the two normalisations, and in practice one is interested in primes $\ell$ which are highly congruent to $1$ modulo $p$ anyway.
  \end{remark}

  For a character $\eta$ of the finite group $\Gamma_{\mathrm{tors}}$, let $e_\eta$ be the corresponding idempotent in $\Lambda_{\Gamma}$.

  \begin{theorem}[Main Conjecture without zeta-functions]
   \label{thm:cycloIMC}
   Suppose $\Hyp(BI)$ holds, and $\eta$ is a character of $\Gamma_{\mathrm{tors}}$ such that $e_\eta \cdot \BF^{f, g}_1$ is non-zero. Then:
   \begin{enumerate}[(i)]
    \item $e_{\eta} \cdot \Ht^2(\ZZ[1/S], T \otimes \Lambda_\Gamma(-\j); \Delta^{(\varnothing)})$ is a torsion $\Lambda$-module;
    \item $e_{\eta} \cdot \Ht^1(\ZZ[1/S], T \otimes \Lambda_\Gamma(-\j); \Delta^{(\varnothing)})$ is torsion-free of rank 1;
    \item the characteristic ideal $\Char_{\Lambda_\Gamma} e_{\eta} \Ht^2(\ZZ[1/S], T \otimes \Lambda_\Gamma(-\j); \Delta^{(\varnothing)})$ divides
     \[ \Char_{\Lambda_\Gamma} \left( \frac{e_\eta \cdot \Ht^1(\ZZ[1/S], T \otimes \Lambda_\Gamma(-\j); \Delta^{(\varnothing)})}{e_\eta \Lambda_{\Gamma} \cdot \BF^{f, g}_1}\right) \cdot \Char_{\Lambda_\Gamma} e_\eta H^2(\Qp, T^{(\varnothing)} \otimes \Lambda_\Gamma(-\j)). \]
   \end{enumerate}
   If $\Hyp(NEZ)$ holds, then the final factor $\Char_{\Lambda_\Gamma} e_\eta H^2(\Qp, T^{(\varnothing)} \otimes \Lambda_\Gamma(-\j))$ is a unit.
  \end{theorem}

  \begin{proof}
   This follows by an ``Euler system argument'' adapted to take into account the local condition at $p$; parts (i)--(iii) are exactly Corollary \ref{cor:ESargument} in the appendix, applied to $T(\eta^{-1})$, the submodule $T^{(\varnothing)}(\eta^{-1})$, and the generalised Kolyvagin system $\boldsymbol\kappa$ constructed from $\mathbf{c}$ using Proposition \ref{prop:EStoKS}.

   We briefly check the hypotheses of these statements. We take the set of primes $\mathcal{P}$ to be the primes $\ell \nmid p c N_f N_g$ for which $T / (\sigma_\ell - 1) T$ is cyclic as an $\cO$-module. Any prime $\ell \in \mathcal{P}$ has the property that $\sigma_\ell^{p^k}$ is injective on $T$ for every $k \ge 1$, since the eigenvalues of $\sigma_\ell$ are $\ell$-Weil numbers of weight $r + r' - 2 > 0$. We have seen that the Euler system $\mathbf{c}$ respects the local condition given by $T^{(\varnothing)}$; so Proposition \ref{prop:EStoKS} applies, and we obtain a Kolyvagin system satisfying the local condition.

   The Mazur--Rubin hypotheses (H.0)--(H.4) needed to apply Corollary \ref{cor:ESargument} are clear from $\Hyp(BI)$, using the existence of the element $\sigma$ acting as $-1$ to give (H.3). Our set of primes $\mathcal{P}$ contains all but finitely many primes in the set $\mathcal{P}_1$ of Mazur--Rubin, so (H.5) is satisfied as well. The additional ``no local zero'' hypothesis that $H^0(\QQ_{p, \infty}, \sF^{--} T(\eta^{-1})) = 0$ follows from the fact that $\alpha_f \alpha_g$ cannot be 1. Hence we may apply Corollary \ref{cor:ESargument} to give the stated divisibility.

   It remains to check the final statement regarding the characteristic ideal of the local $H^2$ term. The module $H^2(\Qp, T^{(\varnothing)} \otimes \Lambda_\Gamma(-\j))$ is finite unless $H^0(\QQ(\mu_{p^\infty}), T^{(\varnothing)}) \ne 0$, which can only occur if one of the eigenvalues of crystalline Frobenius on $\DD_\mathrm{cris}(T^{(\varnothing)})$ is a power of $p$. These eigenvalues are the inverses of $\{\alpha_f \beta_g, \beta_f \alpha_g, \beta_f \beta_g\}$. It is clear that $\beta_f \beta_g$ cannot be a power of $p$, and the others are covered by $\Hyp(NEZ)$.
  \end{proof}

 \subsection{Finiteness of Selmer groups at finite level}

  \begin{theorem}
   \label{thm:Sel0bound-finite}
   Suppose $\Hyp(BI)$ and $\Hyp(NEZ)$ hold. Let $\tau$ be a $\cO$-valued character of $\Gamma$, and suppose that the image of $\BF^{f, g}_1$ in $H^1(\ZZ[1/S], T(\tau^{-1}))$ is non-zero.

   Then $\Ht^2(\ZZ[1/S], T(\tau^{-1}); \Delta^{(\varnothing)})$ is finite, $\Ht^1(\ZZ[1/S], T(\tau^{-1}); \Delta^{(\varnothing)})$ is free of rank 1 over $\cO$, and we have the bound
   \[ \# \Ht^2(\ZZ[1/S], T(\tau^{-1}); \Delta^{(\varnothing)}) \le \# \left(\frac{\Ht^1(\ZZ[1/S], T(\tau^{-1}); \Delta^{(\varnothing)})}{\cO \cdot \tau(\BF^{f, g}_{1})}\right).\]
  \end{theorem}

  \begin{proof}
   We will deduce this from Theorem \ref{thm:cycloIMC} using a descent argument, which is essentially an elaboration of \S 14.14 of \cite{Kato-p-adic}.

   It is easy to check using Tor spectral sequences that if $C^\bullet$ is a perfect complex of $\Lambda$-modules supported in degrees $\{0, 1, 2\}$, $H^0(C)$ and $H^1(C)$ are torsion, and we are given an element $z \in H^1(C)$ such that $H^1(C) / z$ is torsion, then the formation of the fractional ideal
   \[ \frac{\Char_{\Lambda}(H^1(C) / z)}{\Char_\Lambda H^0(C) \Char_\Lambda H^2(C)}\]
   commutes with base-change in $\Lambda$, in the sense that if $\tau: \Lambda \to \cO$ is a homomorphism whose kernel is not in the support of any of the modules $H^0(C)$, $H^2(C)$, or $H^1(C) / z$, then the image of this ideal under $\tau$ is the fractional ideal
   \[ \frac{\Char_{\cO}(H^1(C') / \tau(z))}{\Char_\cO H^0(C') \Char_\cO H^2(C')}\]
   where $C' = \cO \otimes_{\Lambda, \tau} C$.

   We apply this with $C = \RGt(\ZZ[1/S], T \otimes \Lambda(-\j); \Delta^{(\varnothing)})$ (or, more accurately, a complex representing this object in the derived category). The quotient of $\Lambda$-characteristic ideals is contained in $\Lambda$, by Theorem \ref{thm:cycloIMC}; hence any prime not in the support of $H^1(C) / z$ is also not in the support of $H^2$ (or, vacuously, of $H^0(C) = 0$). This gives the above theorem.
  \end{proof}

 \subsection{Iwasawa Main Conjectures ``with p-adic zeta functions''}

  We shall now explain how the bounds obtained above for the Selmer complex of $\Delta^{(\varnothing)}$ translate into bounds for the Selmer complexes of $\Delta^{(f)}$ and $\Delta^{(g)}$ in terms of $p$-adic $L$-functions. Recall that we have $T^{(\varnothing)} / T^{(f)} = \sF^{-+} M_{\cO}(f \otimes g)^*$.

  In this section we shall assume $\Hyp(NEZ)$.

  \begin{definition}
   Let $\Col^{(f)}$ be the \emph{Coleman map}
   \begin{align*}
    \Col^{(f)}: H^1\left(\Qp, \frac{T^{(\varnothing)}}{T^{(f)}} \otimes \Lambda_{\Gamma}(-\j)\right) & \rTo \Lambda_\Gamma \otimes_{\Zp} \Lp(\mu_N)
   \end{align*}
   given by $\langle \cL(\sim), \eta_f \otimes \omega_g\rangle$.
  \end{definition}

  \begin{proposition}
   We have
   \[ \Col^{(f)}\left(\BF^{f, g}_{1}\right) = L_p(f,g, 1 + \j),\]
   where $L_p(f,g)$ is the Rankin--Selberg $p$-adic $L$-function.
  \end{proposition}

  \begin{proof}
   This is a special case of the explicit reciprocity law of Theorem \ref{thm:explicitrecip}.
  \end{proof}

  It will be convenient to renormalise to remove the possible denominators, which arise from the fact that $\eta_f^\alpha$ may fail to be in the natural integral lattice.

  \begin{definition}
   We let $\xi$ be any generator of the free rank 1 $\cO$-module
   \[ \Hom_{\cO}(\sF^{-+} M_{\cO}(f \otimes g)^*, \cO).\]
  \end{definition}

  Note that $\xi$ is an $\Lp$-multiple of $G(\varepsilon_f^{-1})G(\varepsilon_g^{-1}) \cdot \eta_f \otimes \omega_g$. It is not necessarily a $\cO$-multiple; in fact, the ratio
  \[ \frac{\xi}{G(\varepsilon_f^{-1})G(\varepsilon_g^{-1}) \cdot \eta_f \otimes \omega_g}\]
  generates the congruence ideal $I_{\frP}(f)$ of $f$ at $\frP$. The choice of $\xi$ determines an ``integral'' Coleman map
   \[ \Col^{(f, \xi)} : H^1(\Qp, (T^{(\varnothing)} / T^{(f)}) \otimes \Lambda_{\Gamma}(-\j)) \to \Lambda_{\Gamma}, \]
  which differs from $\Col^{(f)}$ by a non-zero scalar $\Omega$, so that
  \[ \Col^{(f, \xi)}\left(\BF^{f, g}_1\right) = \frac{L_p(f,g, 1 + \j)}{\Omega}.\]
  The map $\Col^{f, \xi}$ is injective, and its cokernel is pseudo-null, because of $\Hyp(NEZ)$.

  \begin{theorem}
   \label{thm:SelFbound}
   Assume that $\Hyp(NEZ)$ and $\Hyp(BI)$ are satisfied, and that $\eta$ is a character of $\Gamma_{\mathrm{tors}}$ such that $e_\eta \cdot L_p(f, g, 1 + \j) \neq 0$. Then
   \begin{enumerate}[(i)]
    \item $\Ht^1(\ZZ[1/S], T \otimes \Lambda_\Gamma(-\j); \Delta^{(f)})=0$,
    \item $\Ht^2(\ZZ[1/S], T \otimes \Lambda_\Gamma(-\j),\Delta^{(f)})$ is torsion,
    \item we have the divisibility
    \[
     \Char_{\Lambda} e_{\eta} \Ht^2\left(\ZZ[1/S], T \otimes \Lambda_\Gamma(-\j); \Delta^{(f)}\right) \mid e_{\eta} \cdot \frac{L_p(f,g)(1 + \j)}{\Omega}.
    \]
   \end{enumerate}
   Moreover, equality holds if and only if equality holds in Theorem \ref{thm:cycloIMC}.
  \end{theorem}

  \begin{proof}
   Let us write $\TT = T \otimes \Lambda_{\Gamma}(-\j)$, and similarly for the submodules $\TT^{(?)}$. We have an exact triangle
   \[
    \RGt(\ZZ[1/S], \TT; \Delta^{(f)}) \to
    \RGt(\ZZ[1/S], \TT; \Delta^{(\varnothing)}) \to
    R\Gamma(\Qp, \TT^{(\varnothing)} / \TT^{(f)}) \to \dots.
   \]
   The module $e_{\eta}\Ht^1(\ZZ[1/S], \TT; \Delta^{(f)})$ injects into $e_{\eta}\Ht^1(\ZZ[1/S], \TT; \Delta^{(\varnothing)})$; but the hypotheses imply that the latter is torsion-free of rank 1, and contains an element whose image in $e_{\eta} H^1(\Qp, \TT^{(\varnothing)} / \TT^{(f)})$ is non-torsion. Hence $e_{\eta}\Ht^1(\ZZ[1/S], \TT; \Delta^{(f)}) = 0$, which is (i).

   The exact triangle therefore gives us a four-term exact sequence
   \begin{multline*}
    0 \rTo \frac{e_{\eta}\Ht^1(\ZZ[1/S], \TT; \Delta^{(\varnothing)})}{e_\eta \Lambda_\Gamma \cdot  \BF^{f,g}} \rTo  \frac{e_{\eta} H^1(\Qp, \TT^{(\varnothing)} / \TT^{(f)})}{e_\eta \Lambda_\Gamma \cdot \BF^{f,g}}\\
    \rTo \Ht^2(\ZZ[1/S], \TT; \Delta^{(f)})
     \rTo \Ht^2(\ZZ[1/S], \TT; \Delta^{(\varnothing)})
   \end{multline*}
   in which the cokernel of the last map is pseudo-null (being a submodule of $e_{\eta} H^2(\Qp, \TT^{(\varnothing)} / \TT^{(f)})$, which is finite by $\Hyp(NEZ)$). On the other hand, we have an exact sequence
   \[ 0 \to \frac{e_{\eta} H^1(\Qp, \TT^{(\varnothing)} / \TT^{(f)})}{e_\eta \Lambda_\Gamma \cdot \BF^{f,g}} \to \frac{e_\eta \Lambda_\Gamma}{e_{\eta} \Col^{(f, \xi)}(\BF^{f, g})} \to e_{\eta} \operatorname{coker}(\Col^{(f, \xi)}) \to 0\]
   in which the last term is again pseudo-null. Taking characteristic ideals we see that the desired divisibility is equivalent to Theorem \ref{thm:cycloIMC}.
  \end{proof}

  \begin{remark}
   Note that if $r - r' \ge 2$ then the assumption that $e_\eta \cdot L_p(f, g, 1 + \j) \ne 0$ is automatically satisfied (for all $\eta$), because of Proposition \ref{prop:shahidi}.
  \end{remark}

  As before, we also obtain a result at ``finite level''.

  \begin{theorem}
   \label{thm:SelFbound-finite}
   Suppose $\Hyp(NEZ)$ and $\Hyp(BI)$ hold, and that $\tau$ is an $\cO$-valued character of $\Gamma$.
   \begin{enumerate}[(i)]
    \item We have
    \[ \Rank_{\cO} \Ht^1(\ZZ[1/S], T(\tau^{-1}); \Delta^{(f)}) = \Rank_{\cO} \Ht^2(\ZZ[1/S], T(\tau^{-1}); \Delta^{(f)}) \le \ord_{\j = \tau} L_p(f, g, 1 + \j). \]
    \item If $L_p(f, g, 1 + \tau) \ne 0$, then we have $\Ht^1(\ZZ[1/S], T(\tau^{-1}); \Delta^{(f)}) = 0$, and $\Ht^2(\ZZ[1/S], T(\tau^{-1}); \Delta^{(f)})$ is a finite $\cO$-module, whose length is bounded above by
   \[ v_{\frP}\left( \frac{L_{p}(f, g, 1 + \tau)}{\Omega}\right).\]
   \end{enumerate}
  \end{theorem}

  \begin{proof}
   We will derive this from \ref{thm:SelFbound} via descent. An Euler characteristic computation (together with the vanishing of $\Ht^0$ by $\Hyp(BI)$) shows that $\Ht^1(\ZZ[1/S], T(\tau^{-1}); \Delta^{(f)})$ and $\Ht^2(\ZZ[1/S], T(\tau^{-1}); \Delta^{(f)})$ have the same rank, and $\Hyp(BI)$ and $\Hyp(NEZ)$ together also imply that $H^1(\ZZ[1/S], T(\tau^{-1}))$ is a free $\cO$-module and $\Ht^1(\ZZ[1/S], T(\tau^{-1}); \Delta^{(f)})$ injects into it.

   Let $\eta$ be the restriction of $\tau$ to $\Gamma_{\mathrm{tors}}$. If $e_\eta L_p(f, g, 1 + \j) = 0$, then the order of vanishing on the right-hand side of (i) is infinite, and there is nothing to prove. Otherwise, Theorem \ref{thm:SelFbound} and the compatibility of Selmer complexes with derived base-change combine to show that $\Ht^1$ and $\Ht^2$ are, respectively, the maximal submodule and maximal quotient of the torsion module $\Ht^2(\ZZ[1/S], T \otimes \Lambda_{\Gamma}(-\j))$ on which $\Gamma$ acts via $\tau$. So the $\cO$-ranks of these modules are both equal to the order of vanishing of $\Char_{\Lambda} e_\eta \Ht^2(\ZZ[1/S], T \otimes \Lambda_{\Gamma}(-\j))$ at $\j = \tau$, and this characteristic ideal divides $e_\eta L_p(f, g, 1 + \j)$.

   This gives (i) immediately, and (ii) follows from the compatibility of the ratio of characteristic ideals with derived base-change, as in Theorem \ref{thm:Sel0bound-finite}.
  \end{proof}

 \subsection{Bounds for Bloch--Kato Selmer groups}

  In order to link this with more classical results, we need to show that the $\Ht^2$ groups appearing in Theorem \ref{thm:SelFbound-finite} are related to the Bloch--Kato Selmer groups.

  We assume (without loss of generality) that the weights $r, r'$ of $f$ and $g$ satisfy $r > r'$. (In fact the results of this section will be vacuous if $r = r'$, and the case $r < r'$ follows by interchanging the roles of $f$ and $g$.) Thus $\Hyp(NEZ)$ is automatically satisfied.

  \begin{definition}
   We shall say a character $\tau$ is \emph{critical} if it is of the form $j + \chi$ (i.e., $z \mapsto z^j \chi(z)$), where $j$ is an integer with $r' - 1 \le j \le r - 2$ and $\chi$ is of finite order.
  \end{definition}

  These are exactly the characters for which $T(\tau^{-1})$ is the \'etale realisation of a critical motive. If $\tau$ is critical, then $T(\tau^{-1})$ is de Rham at $p$, so we may consider the Bloch--Kato Selmer structure $\Delta^{\mathrm{BK}}$ on $T(\tau^{-1})$.

  \begin{proposition}
   If $\tau$ is critical, and $\Hyp(BI)$ holds, there are natural maps
   \[ \Ht^i(\ZZ[1/S], T(\tau^{-1}); \Delta^{(f)}) \to \Ht^i(\ZZ[1/S], T(\tau^{-1}); \Delta^{\mathrm{BK}}) \]
   for $i = 1,2$, and these maps have finite kernels and cokernels.
  \end{proposition}

  \begin{proof}
   Via the exact triangle \eqref{eq:exacttri}, it suffices to check that there are morphisms of complexes $(U_v^{(f)})^+ \to (U_v^{\mathrm{BK}})^+$ for each prime $v$, and the induced maps on $H^i$ have finite kernels and cokernels for $i = 0, 1, 2$.

   The map on $H^0$ is trivially an isomorphism as its source and target are both zero. This is also true for $H^2$ except for $v = p$, in which case the target is zero and the source is $H^2(\Qp, T^{(f)})$, which is finite. This leaves the case $i = 1$. For $v \ne p$ this map is the inclusion of the unramified cohomology in its saturation, so it is injective with finite cokernel.

   It remains to prove that the image of the injection $H^1(\Qp, T^{(f)}) \into H^1(\Qp, T)$ is a finite-index submodule of $H^1_\mathrm{f}(\Qp, T)$, or equivalently that $H^1(\Qp, V^{(f)}) = H^1_{\mathrm{f}}(\Qp, V)$ where $V = T[1/p]$. Since we have $H^0(\Qp, V / V^{(f)}) = H^2(\Qp, V^{(f)}) = 0$, and $V^{(f)}$ has Hodge--Tate weights $\ge 1$ while $V / V^{(f)}$ has Hodge--Tate weights $\le 0$, this holds by \cite[Lemma 4.1.7]{FukayaKato-formulation-conjectures}.
  \end{proof}

  \begin{theorem}
   Let $\tau$ be a critical character, and suppose that $\Hyp(BI)$ holds. If $L(f, g, \chi^{-1}, 1 + j)$ is non-zero, then the group
   \[ \Ht^2(\cO_{K, S}, M_{\cO}(f \otimes g)(\tau^{-1}); \Delta^{\mathrm{BK}}) \]
   is finite.
  \end{theorem}

  \begin{proof}
   The preceding proposition shows that $\Ht^2(\cO_{K, S}, M_{\cO}(f \otimes g)(\tau^{-1}); \Delta^{\mathrm{BK}})$ is finite if and only if $\Ht^2(\cO_{K, S}, M_{\cO}(f \otimes g)(\tau^{-1}); \Delta^{(f)})$ is finite, and the interpolating property of the $L$-function in Theorem \ref{thm:hida} (together with $\Hyp(NEZ)$, which is automatic, as we have noted) shows that the $p$-adic $L$-value $L_p(f, g, 1 + \tau)$ is $L(f, g, \chi^{-1}, 1 + j)$ multiplied by a non-zero factor. So this statement follows directly from Theorem \ref{thm:SelFbound-finite}.
  \end{proof}

  A special case of this statement is the following theorem, which extends
  results by Bertolini--Darmon--Rotger \cite{BDR-BeilinsonFlach2}.

  \begin{theorem}
   \label{thm:sha}
   Let $E / \QQ$ be an elliptic curve without complex multiplication, and $\rho$ a 2-dimensional odd irreducible Artin representation of $G_{\QQ}$ (with values in some finite extension $L/\QQ$). Let $\frP$ be a prime of $L$ above some rational prime $p$. Suppose that the following technical hypotheses are satisfied:
   \begin{enumerate}[(i)]
    \item The conductors $N_E$ and $N_\rho$ are coprime;
    \item $p \ge 5$;
    \item $p \nmid N_E N_\rho$;
    \item the map $G_\QQ \to \operatorname{Aut}_{\Zp}(T_p E)$ is surjective;
    \item $E$ is ordinary at $p$;
    \item $\rho(\operatorname{Frob}_p)$ has distinct eigenvalues modulo $\frP$.
   \end{enumerate}
   If $L(E, \rho, 1) \ne 0$, then the group
   \[ \Hom_{\Zp[\Gal(F / \QQ)]}(\rho, \operatorname{Sel}_{p^\infty}(E / F))\]
   (where $F$ is the splitting field of $\rho$) is finite.
  \end{theorem}

  \begin{proof}
   It is a standard result that
   \[ \operatorname{Sel}_{p^\infty}(E / F) = \Ht^1\left(\cO_{F, S}, E_{p^\infty}; (\Delta^{\mathrm{BK}})^\vee\right) \]
   where $(\Delta^{\mathrm{BK}})^\vee$ is the orthogonal complement of the Bloch--Kato Selmer structure for $T_p(E)$. So, if $\rho^*$ is the contragredient of $\rho$, and $H = \Gal(F/\QQ)$, we can interpret
   $\Hom_{\Zp[H]}(\rho, \operatorname{Sel}_{p^\infty}(E / F))$ as the $H$-invariant classes in $\Ht^1(\cO_{F, S}, E_{p^\infty}; (\Delta^{\mathrm{BK}})^\vee) \otimes \rho^* = \Ht^1(\cO_{F, S}, E_{p^\infty} \otimes \rho^*; (\Delta^{\mathrm{BK}})^\vee)$. Since $H$ is finite, the restriction map
   \[ \Ht^1(\ZZ[1/S], E_{p^\infty} \otimes \rho^*; (\Delta^{\mathrm{BK}})^\vee) \to \Ht^1(\cO_{F, S}, E_{p^\infty} \otimes \rho^*; (\Delta^{\mathrm{BK}})^\vee)^H \]
   has finite kernel and cokernel (in fact the kernel is even trivial, because of the ``big image'' hypothesis (iv)). So it suffices to show that the former group is finite.

   One knows that both $E$  and $\rho^*$  are modular, associated to modular forms $f$ of weight $2$ (\cite{BCDT}) and $g$ of weight 1 (\cite{KhareWintenberger}). Then we have an isomorphism of $G_{\QQ}$-representations $E_{p^\infty} \otimes \rho^* \cong T^\vee(1)$, where $T = M_{\cO}(f \otimes g)^*$.

   By Nekov\v{a}\'r duality, we have
   \[ \Ht^1(\ZZ[1/S], T^\vee(1); (\Delta^{\mathrm{BK}})^\vee) = \Ht^2(\ZZ[1/S], T; \Delta^{\mathrm{BK}})^\vee. \]
   The hypotheses we have imposed imply that the preceding theorem applies to $f$ and $g$; the hypotheses (i) and (iv) are sufficient to imply $\Hyp(BI)$ (by the results of \cite{Loeffler-big-image}) and the final hypothesis (vi) implies that $g_\rho$ is $p$-distinguished. So $\Ht^2(\ZZ[1/S], T; \Delta^{\mathrm{BK}})$ is finite, as required.
  \end{proof}

  \begin{remark}\mbox{~}
   \begin{enumerate}[(i)]
    \item Because of the fundamental exact sequence
    \[ 0 \rTo E(F) \otimes \Qp/\Zp \rTo \operatorname{Sel}_{p^\infty}(E/F) \rTo \Sha_{p^\infty}(E / F) \rTo 0, \]
    where $\Sha_{p^\infty}(E / F)$ is the $p$-part of the Tate--Shafarevich group, $\Hom_{H}(\rho, \operatorname{Sel}_{p^\infty}(E/F))$ is finite if and only if the same is true of the $\rho$-parts of both $E(F)$ and $\Sha_{p^\infty}(E / F)$. The finiteness of the $\rho$-part of $E(F)$ in this setting has been shown by Bertolini--Darmon--Rotger \cite{BDR-BeilinsonFlach2}. Our theorem extends this to obtain a finiteness result of the $\rho$-part of the Tate--Shafarevich group for a large set of primes $p$.

    \item Although we do not give the details here, one can check that if the hypotheses of the theorem are satisfied, then after possibly throwing away an additional density-0 set of ``anomalous'' primes (those at which the Euler factor $\cE(f, g, 1)$ fails to be a $p$-adic unit), the $\rho$-part of $\operatorname{Sel}_{p^\infty}(E / F)$ is not only finite but trivial.

    \item A result of this type is already known if $\rho$ is one-dimensional (using Kato's Euler system, \cite[Corollary 14.3]{Kato-p-adic}); or if $\rho$ is induced from a ring class character of an imaginary quadratic field (using an Euler system constructed from Heegner points, \cite[Theorem 1.2]{LongoVigni}, building on earlier work of Bertolini--Darmon \cite{BertoliniDarmon-anticyclotomic}).

    \item Since the image of $\rho$ is finite, for all but finitely many $\frP$ the assumption (vi) is equivalent to requiring that $\rho(\operatorname{Frob}_p)$ is non-scalar, or equivalently that $p$ does not split in the Galois extension of $\QQ$ cut out by the projective representation $\bar\rho$. The ordinary primes for $E$ have density 1, and the other hypotheses each rule out only finitely many primes $p$; so the set of primes $p$ to which the theorem applies has density $1 - \tfrac{1}{N}$, where $N$ is the size of the image of $\bar\rho$.
   \end{enumerate}
  \end{remark}


\section{Appendix: Kolyvagin systems with Greenberg local conditions}

 In this section we extend some of the results of \cite{MazurRubin-Kolyvagin-systems} to study Euler systems and Kolyvagin systems with a non-trivial local condition at $p$.

 \subsection{Setup}

  In this section, we fix a prime $p \ge 5$, a finite extension $E / \Qp$ with ring of integers $\cO$ and residue field $\FF$, and a finite-rank free $\cO$-module $T$ with a continuous action of $G_{\QQ}$, unramified outside some finite set $S$ (which we shall assume contains $p$ and $\infty$). We assume the following condition:

  \begin{hypothesis}
   [``Greenberg condition''] There exists a saturated $\cO$-submodule $T^+ \subseteq T$ which is stable under the action of $G_{\Qp}$.
  \end{hypothesis}

  We fix such a submodule $T^+$, and write $T^- = T / T^+$. Let $\TT = T \otimes_{\cO} \Lambda(-\j)$, where $\Lambda = \cO[[\Gamma_1]]$ and $\j$ is the canonical character, and similarly for $\TT^+$ and $\TT^-$. Let $\mathfrak{m}$ be the maximal ideal of $\Lambda$.

  \begin{definition}
   Let $\Delta^+$ be the Selmer structure on $\TT$ for which $\Delta_v^+$ is the unramified local condition for $v \ne p$, and is the Greenberg local condition associated to $\TT^+ \into \TT$ for $v = p$.
  \end{definition}

  We use the same symbol also for the corresponding local condition for $T$ rather than $\TT$, or for $T$ over any finite extension of $\QQ$.

  For technical reasons we will also need the following alternative Selmer structures. Recall that a \emph{simple} Selmer structure on a module $M$ (over $\QQ$) is determined by the data of a subspace $H^1_{\Delta}(\QQ_v, M) \subseteq H^1(\QQ_v, M)$ for every $v \in S$.

  \begin{definition}
   Let $I$ be any ideal of $\Lambda$. We define simple Selmer structures $\Delta_{\Lambda}$ and $\Delta^+_{\Lambda}$ on $\TT / I \TT$ (over $\QQ$) as follows.
   \begin{itemize}
    \item For $v \ne p$, the local condition at $v$ is given by the image of $H^1(\QQ_v, \TT) \to H^1(\QQ_v, \TT / I \TT)$.
    \item The local condition at $p$ is given by the image of $H^1(\QQ_v, \TT) \to H^1(\QQ_v, \TT / I \TT)$ for $\Delta_{\Lambda}$, and the image of $H^1(\QQ_v, \TT^+) \to H^1(\QQ_v, \TT / I \TT)$ for $\Delta^+_\Lambda$.
   \end{itemize}
  \end{definition}

  \begin{remark}
   Note that our notion of a ``simple Selmer structure'' coincides with the non-derived approach to Selmer structures followed in \cite{MazurRubin-Kolyvagin-systems}, and $\Delta_\Lambda$ is exactly the Selmer structure denoted $\cF_\Lambda$ in \S 5.3 of \emph{op.cit.}.
  \end{remark}

 \subsection{Euler systems and Kolyvagin systems}

  As in Definition 3.2.2 of \cite{MazurRubin-Kolyvagin-systems}, let $\cP$ be a set of primes with $S \cap \cP = \varnothing$, and $\mathcal{K}$ an abelian extension of $\QQ$ containing the maximal abelian $p$-extension unramified outside $\cP \cup \{ p\}$. (Thus, in particular, $\mathcal{K}$ contains the cyclotomic $\Zp$-extension $\QQ_\infty$ of $\QQ$.)

  \begin{definition}
   An \emph{Euler system with local condition $T^+$} is an Euler system $\mathbf{c} = \{ c_F : \QQ \subseteq_\mathrm{f} F \subset \mathcal{K}\}$ for $(T, \mathcal{K}, \mathcal{P})$ with the property that for every $K$ we have
   \[ c_F \in \operatorname{image} \left( \Ht^1(\cO_{F, S}, T; \Delta^+) \to H^1(\cO_{F, S}, T) \right).\]
   We denote the module of such systems by $\operatorname{ES}(T, \mathcal{K}, \mathcal{P}; \Delta^+)$.
  \end{definition}

  \begin{remark}
   The map $\Ht^1(\cO_{F, S}, T; \Delta^+) \to H^1(\cO_{F, S}, T)$ may not be injective, but the map
   \[ \varprojlim_n \Ht^1(\cO_{F_n, S}, T; \Delta^+) \to \varprojlim_n H^1(\cO_{F_n, S}, T)\]
   is always injective, where $F_n$ is the $n$-th layer of the cyclotomic $\Zp$-extension $F\QQ_{\infty} / F$. Thus for each $F$ we have a \emph{distinguished} lifting $\tilde c_F$ of $c_F$ to $\Ht^1(\cO_{F, S}, T; \Delta^+)$, and these $\tilde c_F$ satisfy the Euler system compatibility relation.
  \end{remark}

  We also have the notion of a \emph{Kolyvagin system}. We shall not define a Kolyvagin system for $\TT$ as such, but rather an element of the module of \emph{generalised Kolyvagin systems} (cf.\ Definition 3.1.6 of \cite{MazurRubin-Kolyvagin-systems}),
  \[ \overline{\operatorname{KS}}(\TT, \Delta_\Lambda, \mathcal{P}) \coloneqq \varprojlim_k \varinjlim_j \operatorname{KS}(\TT / \mathfrak{m}^k \TT, \Delta_\Lambda, \mathcal{P} \cap \mathcal{P}_j), \]
  where the $\mathcal{P}_j$ are certain sets of primes defined as in \emph{op.cit.}. An element of this module is a collection of classes
  \[ \boldsymbol{\kappa} = (\kappa_{n, k}) \in \Ht^1(\QQ, \mathbf{T} / \mathfrak{m}^k \mathbf{T}; \Delta_\Lambda) \otimes G_n, \]
  where $n$ is a square-free product of primes lying in the set subset $\mathcal{P} \cap \mathcal{P}_{j(k)}$ (with $j(k) \ge k$ an integer depending on $k$), and $G_n = \bigotimes_{\ell \mid n} \mathbf{F}_\ell^\times$.

  \begin{proposition}
   \label{prop:EStoKS}
   Suppose that the hypotheses of Appendix A of \cite{MazurRubin-Kolyvagin-systems} are satisfied. Then the canonical homomorphism
   \[ \operatorname{ES}(T, \mathcal{K}, \mathcal{P}) \to \overline{\operatorname{KS}}(\TT, \Delta_{\Lambda}, \mathcal{P})\]
   defined in Theorem 5.3.3 of \emph{op.cit.}~restricts to a map
   \[ \operatorname{ES}(T, T^+, \mathcal{K}, \mathcal{P}) \to \overline{\operatorname{KS}}(\TT, \Delta_{\Lambda}^+, \mathcal{P}).\]
  \end{proposition}

  \begin{proof}
   Let us recall how $\kappa_{n, k}$ is defined. One first constructs an auxilliary collection of classes $\kappa_{n, k}'$ which are a ``weak Kolyvagin system''. These are constructed as follows. In Chapter 4 of \cite{Rubin-Euler-systems}, one defines a collection of classes $\kappa_{[F, n, M]} \in H^1(F, T / M T) \otimes G_n$, for each power $M$ of $p$, each finite extension $F / \QQ$ contained in the cyclotomic $\Zp$-extension, and each square-free product $n$ of primes in some set $\mathcal{R}_{F, M}$.

   As shown in Appendix A of \cite{MazurRubin-Kolyvagin-systems}, for each $k$, one may find a field $F$, and a power $M$ of $p$, such that $\Lambda / \mathfrak{m}^k \Lambda$ is a quotient of $(\cO / M \cO)[\Gal(F / \QQ)]$, so we obtain a natural map
   \[ H^1(\cO_{F, S}, T / M T) = H^1\Big(\ZZ[1/S], \mathbf{T} \otimes_{\Lambda} (\cO / M \cO)[\Gal(F / \QQ)]\Big) \to H^1(\ZZ[1/S], \mathbf{T} / \mathfrak{m}^k  \mathbf{T}).\]
   For sufficiently large $j$ we have $\mathcal{P}_j \subseteq \mathcal{R}_{F, M}$, and one then defines $\kappa'_{n, k}$ to be the image of $\kappa_{[F, n, M]}$.

   We claim that the class $\loc_p \kappa_{[F, n, M]}$ lies in the image of
   \[ H^1(F \otimes \Qp, T^+ / M T^+) \otimes G_n \to H^1(F \otimes \Qp, T / M T) \otimes G_n.\]

   If we impose the assumption that $H^0(\Qp, T^- \otimes \FF) = 0$, where $\FF$ is the residue field of $\cO$, then this is exactly the result of \cite[Theorem B.1.4]{LLZ2}. In order to prove this result in the general case, we must delve further into the details of Rubin's construction of the classes $\kappa_{[F, n, M]}$. We adopt the notation of Chapter 4 of \cite{Rubin-Euler-systems}.

   As in Proposition 6.4.8 of \emph{op.cit.}, the fact that $\loc_p \mathbf{c}_{F(n)}$ lies in $H^1(F(n) \otimes \Qp, T^+)$, for \emph{every} finite extension $F / \QQ$ contained in the cyclotomic $\Zp$-extension $\QQ_\infty$, implies that there is a system of maps
   \[
    \mathbf{d}_{F, p}^{+}: X_{F(n)} \to
    \mathbb{W}_M^{+} / \Ind_{\mathcal{D}} W_M^{+}
   \]
   lifting $\mathbf{c}$, where $X_{F(n)}$ is the ``universal Euler system'', $W_M^{+} = T^{+}  / M T^{+}$, etc; and the collection of maps $\mathbf{d}_{F, p}^{+}$ for varying $F$ are determined uniquely up to an element of $\Hom_{\cO[G_{\QQ}]}(X_{F(n)}, \mathbb{W}_M^{+})$. Exactly as in the proof of Theorem 4.5.1 of \emph{op.cit.}, we deduce a compatibility between $\mathbf{d}_{F, p}^{+}$ and the corresponding global maps $\mathbf{d}_{F}$, from which it follows that $\kappa_{[F, n, M]}$ lies in the image of $H^1(F \otimes \Qp, W_M^{+}) \otimes G_n$.

   Using the compatibility of cohomology with inverse limits, and the fact that $\mathbf{T} / \mathfrak{m}^k \mathbf{T}$ is finite, it follows that if we choose $M$ and $F$ sufficiently large, the image of $H^1(F \otimes \Qp, T^{+} / M T^{+})$ in $H^1(\Qp, \mathbf{T} / \mathfrak{m}^k \mathbf{T})$ coincides with the image of $H^1(\Qp, \mathbf{T}^+)$. Hence the classes $\kappa'_{n, k}$ lie in the image of this module, whenever $n$ is a square-free product of primes in $\mathcal{P}_j$ for some sufficiently large $j$. This shows that
   \[ \kappa'_{n, k} \in \Ht^1(\ZZ[1/S], \TT / \mathfrak{m}^k \TT; \Delta^+_\Lambda).\]

   Having constructed the ``weak Kolyvagin system'' $\kappa'_{n, k}$, the classes $\kappa_{n, k}$ are defined as a $\Lambda$-linear combination of the classes $\kappa'_{d, k}$ for integers $d \mid n$, and since $\Ht^1(\ZZ[1/S], \TT / \mathfrak{m}^k \TT; \Delta^+_\Lambda)$ is a $\Lambda$-submodule of $H^1(\ZZ[1/S], \mathbf{T} / \mathfrak{m}^k \TT)$, it follows that the modified classes still lie in this submodule.
  \end{proof}

  \begin{remark}
   The fact that $\loc_p \kappa_{[F, n, M]}$ lies in the image of the cohomology of $T^+$ implies that there is a lifting of $ \kappa_{[F, n, M]}$ to $\Ht^1(\cO_{F, S}, T / MT; \Delta^+)$. Although this lifting is not generally unique, one can always find a larger field $F \subseteq_{\mathrm{f}} F' \subset \QQ_\infty$ such that the kernel of $\Ht^1(\cO_{F', S}, T / MT; \Delta^+) \to H^1(\cO_{F', S}, T / MT)$ is annihilated by the corestriction map for $F' / F$. So, after possibly increasing the bound $k(j)$, we obtain a canonical lifting of each $\kappa_{n, j}$ to $\Ht^1(\ZZ[1/S], \TT / \mathfrak{m}_k \TT; \Delta^+)$.

   We do not know at present how to use this extra information, but it may be relevant in the study of exceptional-zero phenomena. Cf.\ the remarks in \cite[\S 0.19.3]{Nekovar-Selmer-complexes}.
  \end{remark}

 \subsection{Descent lemmas}

  We now show that the existence of a Kolyvagin system implies a bound for a Selmer group over $\Lambda$. In this section we impose a further local assumption:

  \begin{hypothesis}[``no local zero'']
   We have $H^0(\QQ_{p, \infty}, T^-) = 0$.
  \end{hypothesis}

  Let $\frQ$ be a height-1 prime ideal of $\Lambda$. For convenience, we assume that $S_{\frQ} = \Lambda / \frQ \Lambda$ is integrally closed; then $S_{\frQ}$ is a discrete valuation ring. (It is either a finite integral extension of $\cO$, or it is isomorphic to a power series ring $\FF[[X]]$ where $\FF$ is the residue field of $\cO$). We define a simple Selmer structure $\Delta^+_{\mathrm{can}}$ on $T \otimes S_\frQ$ by taking the Bloch--Kato local condition for every $v \ne p$, and at $v = p$ the local condition given by the submodule
  \[ \ker\left(H^1(\Qp, T \otimes S_\frQ) \to H^1(\Qp, T^- \otimes \Frac S_{\frQ})\right). \]
  (This is a modification of the Selmer structure $\cF_{\mathrm{can}}$ of \cite[\S 5.3]{MazurRubin-Kolyvagin-systems}, and reduces to it when $T^+ = T$.) There is a natural morphism $\Delta^+_{\Lambda} \to \Delta^+_{\mathrm{can}}$ of Selmer structures on $T \otimes S_{\frQ}$.

  \begin{definition}[{cf.\ \cite[Definition 5.3.12]{MazurRubin-Kolyvagin-systems}}]
   We define an \emph{exceptional set} $\Sigma_{\Lambda}$ of height-1 prime ideals $\frQ$ of $\Lambda$ as the set of $\frQ$ such that one or more of the following holds:
   \begin{itemize}
    \item $H^2(\ZZ[1/S], \TT)[\frQ]$ is infinite;
    \item $H^2(\Qp, \TT^{+})[\frQ]$ is infinite;
    \item $\frQ = \varpi \Lambda$, where $\varpi$ is a uniformiser of $\cO$.
   \end{itemize}
  \end{definition}

  \begin{lemma}[{cf.\ \cite[Lemma 5.3.13]{MazurRubin-Kolyvagin-systems}}]
   For $\frQ$ a height-1 prime ideal, not in $\Sigma_{\Lambda}$, such that $\Lambda / \frQ$ is integrally closed, the cokernel of the natural injection
   \[ H^1_{\Delta^+_{\Lambda}}(\Qp, T \otimes S_\frQ) \into H^1_{\Delta^+_{\mathrm{can}}}(\QQ_p, T \otimes S_\frQ) \]
   is finite, with order bounded above by a constant depending only on $T$ and the $\cO$-rank $d$ of $S_{\frQ}$.

   If $\frQ = \frQ_k$ is the specific prime $(X^k + \varpi) \Lambda$, where $X = \gamma - 1$ for $\gamma$ a generator of $\Gamma_1$, then the cokernel is bounded independently of $k$.
  \end{lemma}

  \begin{proof}
   Let us write ``$O(1)$'' for a module whose order is bounded independently of $\frQ$.

   We consider the composition
   \[H^1(\Qp, \TT^+) / \frQ  \rInto H^1(\Qp, T^+ \otimes S_{\frQ}) \rTo H^1(\Qp, T \otimes S_{\frQ}).\]
   By definition, the space $H^1_{\Delta^+_{\Lambda}}(\Qp, T \otimes S_{\frQ})$ is the image of the composite of these two maps. The first map is easily seen to be injective, and its cokernel is $O(1)$, bounded above by the order of the largest finite submodule of $H^2(\Qp, T^+)$ (which is a finitely-generated $\cO$-module). So it suffices to show that the cokernel of the map
   \[ H^1(\Qp, \TT^+ \otimes S_{\frQ}) \to H^1_{\Delta_{\mathrm{can}}^+}(\Qp, \TT \otimes S_{\frQ}) \]
   is bounded. This cokernel is precisely the torsion submodule of $H^1(\Qp, T^- \otimes S_{\frQ})$, which can be identified with $H^0(\Qp, T^- \otimes \Frac(S_{\frQ}) / S_{\frQ})$ (modulo its divisible part, but this is zero by assumption).

   The module $H^0(\QQ_{p,\infty}, T^- \otimes_{\cO} E/\cO)$ is a finite $\cO$-module, so there is some $n$ such that $\varpi^n$ annihilates this module. Hence $H^0(\Qp, T^- \otimes \Frac(S_{\frQ}) / S_{\frQ})$ has order bounded above by a constant multiple of $(S_{\frQ} / \varpi^n S_{\frQ})^\Gamma$, where $n$ is independent of $\frQ$, which is clearly bounded above by some function of $\Rank_{\cO} S_{\frQ}$.

   In the case $\frQ = \frQ_k$, we can explicitly write down the action of $X$ on  $S_{\frQ} / \varpi^n S_{\frQ} \cong (\cO / \varpi^n \cO)^k$; from this description it is clear that the cokernel is isomorphic to $\cO / \varpi \cO$, for any $k$ and $n$, so it is bounded independently of $k$.
  \end{proof}

  For $\boldsymbol{\kappa} \in \overline{\operatorname{KS}}(\TT, \Delta_{\Lambda}^+, \mathcal{P})$, let $\Ind(\boldsymbol{\kappa})$ be the index of divisibility of $\kappa_1$ in $H^1(\ZZ[1/S], \TT)$, as in Definition 5.3.8 of \emph{op.cit.}.

  \begin{theorem}
   Suppose that the hypotheses (H.0)--(H.5) of \cite[\S 3.5]{MazurRubin-Kolyvagin-systems} hold. If $\boldsymbol{\kappa} \in \overline{\operatorname{KS}}(\TT, \Delta_{\Lambda}^+, \mathcal{P})$ satisfies $\kappa_1 \ne 0$, then $\Ht^2(\ZZ[1/S], \TT; \Delta^+)$ is a torsion $\Lambda$-module, and we have the divisibility
   \[ \Char_{\Lambda}\left(\ker \Ht^2(\ZZ[1/S], \TT; \Delta^+) \to H^2(\Qp, \TT^+)\right)\, \mid\, \Ind(\boldsymbol{\kappa}).\]
  \end{theorem}

  \begin{proof}
   Since the Selmer structure $\Delta^+_{\Lambda}$ on $\TT$ is just the truncation of $\Delta^+$ in degree $\le 1$, and all the modules $U_v^+$ for the Selmer structure $\Delta^+$ are zero unless $v = p$, we have
   \[ \Ht^2(\ZZ[1/S], \TT; \Delta^+_{\Lambda}) = \ker\left(\Ht^2(\ZZ[1/S], \TT; \Delta^+) \to H^2(\Qp, \TT^+)\right).\]
   Moreover, $\Ht^2(\ZZ[1/S], \TT; \Delta^+_{\Lambda})$ is identified with the Pontryagin dual of $\Ht^1$ for $\TT^\vee(1)$ by Proposition \ref{prop:nekovar-duality}. So we must bound the latter in terms of $\Ind(\boldsymbol\kappa)$.

   This follows by exactly the same argument as in Theorems 5.3.6 and 5.3.10 of \emph{op.cit.}. The Selmer structure $\Delta^+_{\mathrm{can}}$ satisfies the hypotheses of Theorem 5.2.2 of \emph{op.cit.}, so we may use that theorem to bound the Selmer group
   \[ \Ht^2(\ZZ[1/S], T \otimes S_{\frQ}; \Delta^+_{\mathrm{can}}) = \left[\Ht^1(\ZZ[1/S], (T \otimes S_\frQ)^\vee(1); (\Delta^+_{\mathrm{can}})^\vee)\right]^\vee. \]
   (This last isomorphism is a consequence of Nekov\v{a}\'r's duality theorem, cf.\ Proposition \ref{prop:nekovar-duality}, using the fact that the spaces $H^2(U_v^+)$ are defined to be zero for a simple local condition.) Our bounds on the kernel and cokernel of the descent map, while fractionally weaker than those obtained in Lemma 5.3.13 of \emph{op.cit.}, are sufficient to show that if $\frQ \ne \varpi \Lambda$ is a prime of $\Lambda$ such that $\Lambda / \frQ$ is integrally closed, we have the requisite inequality of orders of vanishing at $\frQ$. However, our Selmer complexes commute with flat base-extension in $\cO$, so this gives the result for all $\frQ \ne \varpi \Lambda$. For $\frQ = \varpi \Lambda$ the argument proceeds by considering the primes $\frQ_k = X^k + \varpi$, which we considered above as a special case.
  \end{proof}

  A slightly different formulation of this will be useful above. A simple computation from Tate's local and global Euler characteristic formulae shows that
  \[ \Rank_\Lambda \Ht^1(\ZZ[1/S], \TT; \Delta^+) - \Rank_\Lambda \Ht^2(\ZZ[1/S], \TT; \Delta^+) = \Rank_{\cO} T^{(c = -1)} - \Rank_{\cO} T^-, \]
  where $c$ is complex conjugation; and we define $\chi(\TT, \Delta^+)$ to be this value. Then $\chi(\TT, \Delta^+)$ is also the ``core rank'' of $(T \otimes S_{\frQ}, \Delta^+_{\mathrm{can}})$ in the sense of Definition 5.2.4 of \emph{op.cit.}, for any $\frQ \notin \Sigma_{\Lambda}$.

  \begin{corollary}
   \label{cor:ESargument}
   If $\chi(\TT, \Delta^+) = 1$ and $\boldsymbol{\kappa} \in \overline{\operatorname{KS}}(\TT, \Delta_{\Lambda}^+, \mathcal{P})$ satisfies $\kappa_1 \ne 0$, then $\Ht^1(\ZZ[1/S], \TT; \Delta^+)$ is torsion-free of rank 1, and we have a divisibility of characteristic ideals
   \[ \Char_{\Lambda}\left(\ker \Ht^2(\ZZ[1/S], \TT; \Delta^+) \to H^2(\Qp, \TT^+)\right)\, \mid\, \Char_{\Lambda}\left( \frac{\Ht^1(\ZZ[1/S], \TT; \Delta^+)}{\Lambda \cdot \kappa_1}\right).\]
  \end{corollary}

  (Cf.\ \cite[Theorem 13.4]{Kato-p-adic}.)

  \begin{proof}
   Since the quotient $H^1(\ZZ[1/S], \TT)/\Ht^1(\ZZ[1/S], \TT; \Delta^+)$ injects into $H^1(\Qp, \TT^-)$, whose $\Lambda$-torsion subgroup is pseudo-null by assumption, we deduce that if $\chi(\TT, \Delta^+) = 1$ and $\kappa_1 \ne 0$, then the quotient $\Ht^1(\ZZ[1/S], \TT; \Delta^+) / \Lambda \kappa_1$ is torsion and its characteristic ideal is equal to $\Ind(\boldsymbol\kappa)$.
 \end{proof}

\newcommand{\noopsort}[1]{\relax}

\providecommand{\bysame}{\leavevmode\hbox to3em{\hrulefill}\thinspace}
\providecommand{\MR}[1]{\relax}
\renewcommand{\MR}[1]{%
 MR \href{http://www.ams.org/mathscinet-getitem?mr=#1}{#1}.
}
\providecommand{\href}[2]{#2}
\newcommand{\articlehref}[2]{\href{#1}{#2}}

\clearpage

\section{Correction (November 2023)}

In our paper \cite{KLZ17}, on page 70 of the published version, the following passage appears:\medskip

 \begin{quotation}
  ``\textbf{Definition 7.5.1}. \emph{We say the branch $\mathbf{a}$ [of the Hida family $\mathbf{f}$] is \emph{cuspidal} if the natural map $M(\mathbf{f})^*_{\mathrm{par}} \to M(\mathbf{f})$ becomes an isomorphism after tensoring with $\Lambda_{\mathbf{a}}$.}\smallskip

  A branch $\mathfrak{a}$ is cuspidal if one, or equivalently every, arithmetic prime of $\Lambda_{\mathbf{f}}$ above $\mathfrak{a}$ corresponds to a cuspidal modular form.''
 \end{quotation}

 This is erroneous as stated: the two conditions in the final sentence of the extract are equivalent to each other, but they are both strictly weaker than Definition 7.5.1. We are grateful to Rodolfo Venerucci for bringing this error to our attention.\medskip

 To correct the error, we should \emph{define} a branch to be cuspidal if it satisfies the condition that one (or equivalently all) of its arithmetic specialisations are cuspidal. This is equivalent to $\mathfrak{a}$ being in the support of the quotient $\Lambda_{\mathbf{f}}^{\mathrm{par}}$ of $\Lambda_{\mathbf{f}}$ which acts faithfully on parabolic cohomology; and it implies that $M(\mathbf{f})^*_{\mathrm{par}} \to M(\mathbf{f})$ becomes an isomorphism after tensoring with the \emph{fraction field} of $\Lambda_{\mathfrak{a}}$ (but not necessarily with $\Lambda_{\mathfrak{a}}$ itself).

 This necessitates a few slight modifications further down the paper involving the congruence ideal $I_{\mathbf{a}}$:

 \begin{itemize}
 \item In Notation 7.7.1 we have defined a $\Lambda_{\mathbf{a}}$-submodule $I_{\mathbf{a}} \subset \operatorname{Frac} \Lambda_{\mathbf{a}}$ to be the image of $\Lambda$-adic modular forms $M_{\mathbf{k} + 2}(N, \Lambda_{\mathbf{a}})$ under the map to $\operatorname{Frac} \Lambda_{\mathbf{a}}$ sending the normalised eigenform to 1. However, the paper of Hida which we refer to, \cite{hida88}, considers the image of $S_{\mathbf{k} + 2}(N, \Lambda_{\mathbf{a}})$ in $\operatorname{Frac} \Lambda_{\mathbf{a}}$. So our module $I_{\mathbf{a}}$ may be slightly larger than the module that Hida considers; however, the same arguments apply \emph{mutatis mutandis} to show that our module $I_{\mathbf{a}}$ is a fractional ideal, and that the localisation of $I_{\mathbf{a}}$ at any arithmetic point is the unit ideal. Of course, if $\mathbf{f}$ is non-Eisenstein (i.e.~its residual Galois representation is irreducible), our $I_{\mathbf{a}}$ does coincide with Hida's congruence ideal.

 \item In the proof of Proposition 10.1.1 on page 92, the phrase ``is isomorphic to a space of $\Lambda_D$-adic cusp forms'' should read ``... space of $\Lambda_D$-adic modular forms'' (not necessarily cuspidal), consistently with our definition of $I_{\mathbf{a}}$.
 \end{itemize}

 We also take the opportunity to correct an unrelated typographical error on page 93: in the second paragraph of \S 10.2, ``$\Lambda_{\mathbf{f}} \mathop{\hat\otimes} \Lambda_{\mathbf{a}} \mathop{\hat\otimes} \Lambda_{\Gamma}$'' should read ``$\Lambda_{\mathbf{a}} \mathop{\hat\otimes} \Lambda_{\mathbf{g}} \mathop{\hat\otimes} \Lambda_{\Gamma}$''.

 \medskip

 \emph{Remark}. The discrepancy between the modules $M(\mathbf{f})^*_{\mathrm{par}}$ and $M(\mathbf{f})$ plays an important role in the paper \cite{BDV22}, in the setting of a Hida family of CM forms whose weight 1 specialisation is an Eisenstein series.

\vspace{3cm}
\end{document}